\documentclass[a4paper,12pt]{amsart}
\usepackage[margin=2cm]{geometry}
\usepackage{amsmath}
\usepackage{a4wide}
\usepackage[utf8]{inputenc}
\usepackage{amssymb}
\usepackage{amsopn}
\usepackage{epsfig}
\usepackage{amsfonts}
\usepackage{latexsym}
\usepackage{amsthm}
\usepackage{enumerate}
\usepackage[UKenglish]{babel}
\usepackage{verbatim}
\usepackage{color}
\usepackage{calrsfs}
\usepackage{pgf}
\usepackage{tikz}
\usepackage{breqn}
\usepackage{mathtools}
\usepackage{tikz}
\usepackage{hyperref}
\usepackage{pgfplots}
\hypersetup{
	colorlinks=true,  
	linkcolor=blue,   
	citecolor=blue,   
	urlcolor=blue     
}
\usetikzlibrary{arrows,automata}

\DeclareMathAlphabet{\pazocal}{OMS}{zplm}{m}{n}

\newtheorem{theorem}{Theorem}[section]
\newtheorem{lemma}[theorem]{Lemma}
\newtheorem{proposition}[theorem]{Proposition}
\newtheorem{corollary}[theorem]{Corollary}

\theoremstyle{definition}

\newtheorem{example}[theorem]{Example}
\newtheorem{examples}[theorem]{Examples}

\theoremstyle{remark}
\newtheorem{remark}[theorem]{Remark}

\numberwithin{equation}{section}

\numberwithin{equation}{section}
\newcommand{\set}[1]{\left\{#1\right\}} 
\newcommand{\bj}[1]{\left(#1\right)}
\newcommand{\qtq}[1]{\quad \text{#1}\quad }
\providecommand{\abs}[1]{\left\lvert#1\right\rvert}

\newcommand{\N}{\ensuremath{\mathbb{N}}} 

\newcommand{\A}{\pazocal{A}} 

\newcommand{\U}{\pazocal{U}} 

\newcommand{\V}{\pazocal{V}} 

\newcommand{\C}{\pazocal{C}} 

\newcommand{\f}{\infty}

\makeatletter
\def\@tocline#1#2#3#4#5#6#7{\relax
	\ifnum #1>\c@tocdepth
	\else
	\par \addpenalty\@secpenalty\addvspace{#2}
	\begingroup \hyphenpenalty\@M
	\@ifempty{#4}{
		\@tempdima\csname r@tocindent\number#1\endcsname\relax
	}{
		\@tempdima#4\relax
	}
	\parindent\z@ \leftskip#3\relax \advance\leftskip\@tempdima\relax
	\rightskip\@pnumwidth plus4em \parfillskip-\@pnumwidth
	#5\leavevmode\hskip-\@tempdima
	\ifcase #1
	\or\or \hskip 1em \or \hskip 2em \else \hskip 3em \fi
	#6\nobreak\relax
	\dotfill\hbox to\@pnumwidth{\@tocpagenum{#7}}\par
	\nobreak
	\endgroup
	\fi}
\makeatother

\begin{document}
		
\title{Topology of univoque sets in double-base expansions} 
		\vspace{-5em}
\author[V. Komornik]{Vilmos Komornik}
\address{Département de mathématique,
	Université de Strasbourg,
	7 rue René Descartes,
	67084 Strasbourg Cedex, France}
\email{vilmos.komornik@math.unistra.fr}

\author[Y. Li]{Yichang Li}
\address{\footnotesize{School of Mathematical Sciences, Shenzhen University, Shenzhen 518060, People's Republic of China.}}
\email{2200201006@email.szu.edu.cn} 

\author[Y. Zou]{Yuru Zou}
\address{School of Mathematical Sciences,
	Shenzhen University,
	Shenzhen 518060,
	People's Republic of China}
\email{yuruzou@szu.edu.cn}

\begin{abstract}
Given two real numbers $q_0,q_1>1$ satisfying $q_0+q_1\geq q_0q_1$ and two real numbers $d_0\ne d_1$, by a {double-base expansion}  of a real number $x$ we mean a sequence $(i_k)\in \{0,1\}^{\infty}$ such that
\begin{equation*}
x=\sum_{k=1}^{\infty}\frac{d_{i_k}}{q_{{i_1}}q_{{i_2}}\cdots q_{{i_k}}}.
\end{equation*}
We denote by  $\U_{{q_0,q_1}}$ the set of numbers $x$ having a unique expansion.
The topological properties of  $\U_{{q_0,q_1}}$ have been investigated in the equal-base case  $q_0=q_1$ for a long time. 
We extend this research to the case  $q_0\neq q_1$. 
While many results remain valid, a great number of new phenomena  appear due to the increased complexity of double-base expansions.
\end{abstract}
		\maketitle


\section{Introduction}\label{s1}
The study of \emph{non-integer base expansions}  started with the pioneering papers of Rényi \cite{Ren1957} and Parry \cite{Par1960}. Since then hundreds of papers have been devoted to the study of expansions of real numbers of the form 
\begin{equation}\label{11}
x=\pi_{q,D}((d_i)):=\mathop{\sum}\limits_{i=1}^\f\dfrac{d_i}{q^i},
\end{equation}
where $q>1$ is a given real number, and $(d_i)$ is a sequence of \emph{digits}, belonging to a finite \emph{alphabet} $D$ of real numbers.
Many remarkable results have been discovered, revealing deep connections to various fields of mathematics, including number theory \cite{Sid2009, KomLor2007}, topology \cite{DeVKomLor2016, DeVKomLor2022}, ergodic theory \cite{KalKonLanLi2020}, Diophantine approximation and dynamical systems \cite{ChaCisDaj2023}. 
	
Concerning the original alphabet $\{0,1\}$, Erdős et al. \cite{ErdHorJoo1991, ErdJoo1992} discovered in the 1990's that for each $k \in \mathbb{N} \cup \{\aleph_0\} \cup \{2^{\aleph_0} \}$ there exist infinitely many bases $q \in (1,2)$ such that $x=1$ has exactly $k$ different expansions of the form \eqref{11}. 
Subsequently the unique expansions have been intensively studied, and a surprisingly rich theory has emerged \cite{ErdJooKom1990, KomLor1998, GleSid2001, Ped2005, KomLor2007, DeVKom2009, KomLaiPed2011, DeVKom2011, Kom2012, Bak2014, DeVKomLor2016, BakSte2017, KomKonLi2017, All2017, AllKon2019, AlcBarBakKon2019, Ste2020, ZouLiLuKom2021, DeVKomLor2022}.
An essentially complete theory was presented in the papers \cite{DeVKomLor2016, KomKonLi2017, AllKon2019, DeVKomLor2022};  it was also shown that the theory remains valid for the  more general alphabets $\set{0,1,\ldots,M}$, where $M$ is an arbitrary positive integer.
The paper \cite{DeVKomLor2016} was devoted to the study of bases in which the number $1$ has a unique expansion.
Based on these results, the papers \cite{KomKonLi2017, AllKon2019, DeVKomLor2022} were devoted to the sets of numbers having unique expansions in a fixed base.

In the past few years the expansions \eqref{11} have been generalized by Neuh\"auserer \cite{Neu2021}, Li \cite{Li2021} and in \cite{KomLuZou2022} to \emph{multiple-base expansions} of the form
\begin{equation}\label{12}
x=\pi_S((i_k)):=\sum_{k=1}^{\infty}\frac{d_{i_k}}{q_{{i_1}}q_{{i_2}}\cdots q_{{i_k}}},\quad  (i_k)\in \{0,1,\ldots, M\}^\infty,	
\end{equation}
where $S=\set{(d_0,q_0), (d_1,q_1),\ldots, (d_M,q_M)}$ is a given finite \emph{digit-base system} of pairs of real numbers satisfying $q_0, q_1,\ldots, q_M>1$.
Although these generalized expansions have a much higher complexity
(see, e.g., \cite{KomSteZou2024}), most theorems of \cite{DeVKomLor2016} could be generalized in \cite{HuBarZou2024}
to all \emph{double-base expansions}, i.e., to expansions of the form \eqref{12} with $M=1$.
A lot of new phenomena have appeared that do not occur in the equal-base case $q_0=\cdots=q_M$.
The purpose of this paper is to similarly extend many theorems of \cite{DeVKomLor2022} to this more general framework.

Before stating the main results of this paper, let us recall the theorems of \cite{DeVKomLor2022} that we are going to generalize.
We need some definitions and notations.
Unless stated otherwise, in this paper by a \emph{sequence} we always mean an element of $\set{0,1}^{\infty}$, i.e., a sequence of zeros and ones.
We systematically use the notations of symbolic dynamics for sequences $(x_i)$ like $x_1x_2\cdots ,$ $0^{\infty}$, $1^{\infty}$, $(10)^{\infty}$ or $(10)^k1^{\infty}$.

We  systematically use the lexicographical order between sequences: we write $(x_i)\prec (y_i)$ or $(y_i)\succ (x_i)$ if there exists an index $n\in \mathbb{N}$ such that $x_i=y_i$ for all $i<n$, and $x_n<y_n$.
Furthermore, we write $(x_i)\preceq(y_i)$ or $(y_i)\succeq(x_i)$ if
$(x_i)\prec (y_i)$ or if $(x_i)=(y_i)$. 
The \emph{reflection}  of a sequence $(x_i)$ is defined by the formula $\overline{(x_i)}:=(1-x_i)$, i.e., we exchange the digits $0$ and $1$.
We denote by $\sigma$ the \emph{right shift} of sequences, so that 
\begin{equation*}
\sigma^n(x_1x_2\cdots)=x_{n+1}x_{n+2}\cdots \qtq{for every integer}n\ge 0.
\end{equation*}

We  also consider the lexicographical order between finite words of digits of the same length, and the reflection of a word is defined similarly to the reflection of sequences.

A sequence $(x_i)$ is called 
\begin{itemize}
\item 	\emph{finite} if it ends with $10^{\infty}$, and \emph{infinite} otherwise;
\item  \emph{co-finite} if its reflection is finite, i.e., if it ends with $01^{\infty}$, and \emph{co-infinite} otherwise;
\item  \emph{doubly infinite} if it is both infinite and co-infinite, i.e., if it contains infinitely many zero digits and infinitely many one digits. 
\end{itemize}

\begin{remark}\label{r11}
There are only countably many finite or co-finite sequences, so that ``most'' sequences are doubly infinite.
\end{remark}

Now we consider the expansions of the form
\begin{equation}\label{13}
x=\pi_{q}((d_i)):=\mathop{\sum}\limits_{i=1}^\f\dfrac{d_i}{q^i}, \quad (d_i)\in\set{0,1}^{\infty}
\end{equation}
with a given base $q>1$ on the alphabet $\{0,1\}$.
Observe that if $x$ has an expansion, then $x\in J_q:=[0,\frac{1}{q-1}]$.
The converse is not true in general:
\begin{equation*}
\set{\pi_{q}((d_i))\ :\ (d_i)\in\set{0,1}^{\infty}}=J_q
\Longleftrightarrow q\in (1,2].
\end{equation*}
Moreover, if $q\in (1,2]$, then every $x\in J_q:=[0,\frac{1}{q-1}]$ has a lexicographically largest expansion $b(x,q)=(b_i(x,q))$, and a lexicographically largest infinite expansion $a(x,q)=(a_i(x,q))$, called the \emph{greedy} and \emph{quasi-greedy} expansions of $x$ in base $q$, respectively.

Following \cite{KomLor2007} and \cite{DeVKomLor2016} we introduce the  sets
\begin{align*}
&\U:=\set{q\in (1,2]\ :\ 1\text{ has a unique expansion in base }q},\\
&\V:=\set{q\in (1,2]\ :\ 1\text{ has a unique doubly infinite expansion in base }q}.
\intertext{Then the topological closure of $\U$ has an analogous characterization:}
&\overline{\U}=\set{q\in (1,2]\ :\ 1\text{ has a unique  infinite expansion in base }q}.
\end{align*}

We recall that
\begin{equation*}
\U\subsetneqq\overline{\U}\subsetneqq\V\qtq{with}\abs{\V\setminus\overline{\U}}=\abs{\overline{\U}\setminus\U}=\aleph_0;
\end{equation*}
here and in the sequel $\abs{A}$ denotes the cardinality of a set $A$.
Furthermore, $\V$ is compact  and $\overline{\U}$ is a \emph{Cantor set}, i.e., a non-empty compact set having neither isolated, nor interior points.
Their smallest elements are the Golden ratio and the \emph{Komornik--Loreti constant}, respectively, and their largest element is $2$, also belonging to $\U$.

As in \cite{DeVKom2009} and \cite{DeVKomLor2022} we  introduce the following sets for each fixed base $q\in (1,2]$:
\begin{align*}
&\U_q:=\set{x\in J_q\ :\ x\text{ has a unique expansion in base }q},\\
&\overline{\U}_q\text{ is the topological closure of }\U_q,\\
&\V_q:=\set{x\in J_q\ :\ x\text{ has at most  one doubly infinite expansion in base }q}.
\intertext{Then $\V_2:=J_2=[0,1]$, and
}
&\V_q:=\set{x\in J_q\ :\ x\text{ has a unique doubly infinite expansion in base }q}\text{ if }q\in (1,2).
\end{align*}
We recall that
\begin{equation*}
\U_q\subseteqq\overline{\U}_q\subseteqq\V_q\qtq{with}\abs{\V_q\setminus\U_q}\le\aleph_0,
\end{equation*}
and that $\V_q$ is compact.
Finally, we introduce the following partition of $\V_q\setminus\U_q$, where $\alpha(q)=a(1,q)$ denotes the quasi-greedy expansion of $1$ in base $q$:
\begin{equation}\label{e:one-base-A-B}
\begin{split}
&A_q:=\set{x\in \V_q\setminus\U_q: \sigma^i(a(x,q))=\alpha(q)\text{ for at least one digit } a_i(x,q)=0},\\
&B_q:=\set{x\in \V_q\setminus\U_q: \sigma^i(a(x,q))\prec\alpha(q)\text{ for all $i$ with } a_i(x,q)=0}.\\
\end{split}
\end{equation}
Equivalently, $A_q$ and $B_q$ are the sets of numbers $x\in\V_q\setminus\U_q$ whose greedy expansions $b(x,q)$ are finite and infinite, respectively.

In the following two theorems we recall the results of \cite[Theorems 1.2, 1.4, 1.5, 1.10 and 1.12]{DeVKomLor2022} in the case of the alphabet $\set{0,1}$.
(The case of the more general alphabets $\set{0,1,\ldots,M}$ is completely analogous: we only have to define the reflection of a sequence by the formula $\overline{(x_i)}:=(M-x_i)$, and  change $2$ to $M+1$ in Theorem \ref{t:results2 of DeVKomLor2022}.)

\begin{theorem}\label{t:results1 of DeVKomLor2022}\label{t12} \
\begin{enumerate}[\upshape (i)]
\item If $q\in \U$, then every $x\in \V_q\setminus\U_q$ has exactly two expansions.

\item If $q\in \V\setminus\U$, then every $x\in \V_q\setminus\U_q$ has exactly $\aleph_0$ expansions.
\end{enumerate}
\end{theorem}

\begin{theorem}\label{t:results2 of DeVKomLor2022}\label{t13} \
\begin{enumerate}[\upshape (i)]
\item Let $q\in \overline{\U}.$ 

\begin{enumerate}[\upshape (a)]
\item $|\V_q\setminus\U_q|=\aleph_0$ and $\V_q\setminus\U_q$ is dense in $\V_q$.

\item If $q=2$, then $\overline{\U_q}=\V_q=J_q=[0,1]$.

\item If $q\in \overline{\U}\setminus\{2\}$, then  $\overline{\U_q}=\V_q$ is a Cantor set. 
Furthermore, $J_q\setminus \V_q$ is the union of infinitely many disjoint open intervals $(x_L,x_R)$, where $x_L$ and $x_R$ run over $A_q$ and $B_q$, respectively. 
More precisely,
\begin{equation*}
\text{if }b(x_L,q)=b_1\cdots b_n0^\infty\text{ with }b_n=1,
\text{ then }b(x_R,q)=b_1\cdots b_{n}\overline{\alpha(q)}.
\end{equation*}
\end{enumerate}

\item Let $q\in \V\setminus\overline{\U}$. 
\begin{enumerate}[\upshape (a)]

\item	The sets $\U_q$ and $\V_q$ are closed. 
\item $|\V_q\setminus\U_q|=\aleph_0$, and $\V_q\setminus\U_q$ is a discrete set, dense in $\V_q$. 

\item Each connected component $(x_L,x_R)$ of $J_q\setminus \U_q$ contains infinitely many elements of $\V_q$, forming an increasing sequence $(x_k)^\infty_{k=-\infty}$ satisfying 
\begin{equation*}
x_k\rightarrow x_L \text{ as } k\rightarrow-\infty,\text{ and }x_k\rightarrow x_R \text{ as } k\rightarrow\infty.
\end{equation*}
Moreover, each $x_k$ has a finite greedy expansion
\begin{equation*}
b(x_k,q)=b_1\cdots b_n0^\infty\text{ with } b_n=1,
\end{equation*}	
and then
\begin{equation*}
a(x_{k+1},q)=b_1\cdots b_n\overline{\alpha(q)}.
\end{equation*}
\end{enumerate}

\item If $q\in (1,2]\setminus\V$, then $\U_q=\overline{\U_q}=\V_q.$
\end{enumerate}
\end{theorem}

Table \ref{t:one-base-A-B} gives an overview of the main topological properties of $\U_q$, $\overline\U_q$ and $\V_q$ in the equal-base case, contained in Theorems \ref{t:results1 of DeVKomLor2022} and \ref{t:results2 of DeVKomLor2022}, with some further information on the number of expansions, proved in \cite{DeVKomLor2022}.
We also recall from \cite{DeVKomLor2022} that $A_q$ and $B_q$ always form a partition of $V_q\setminus\U_q$, i.e.,
\begin{equation*}
\V_q\setminus\U_q=A_q\cup B_q\qtq{and}A_q\cap B_q=\emptyset.
\end{equation*}
Furthermore,
\begin{itemize}
\item $\abs{A_q}=\aleph_0$ if $q\in\V$; otherwise $A_q=\emptyset$;
\item $\abs{B_q}=\aleph_0$ if $2\ne q\in\overline{\U}$; otherwise $B_q=\emptyset$.
\end{itemize}

In Table \ref{t:one-base-A-B} $|A'_x|$ and $|B'_x|$ denote the number of expansions of each $x\in A_q$ and $x\in B_q$, respectively.

\color{black}
\begin{table*}[h]
\begin{center}
\begin{tabular}{|c|c|c|c|}
\hline
$q\in $ 
& Inclusions 
& $|A'_x|$ 
& $|B'_x|$\\			
\hline

$\{2\}$
& $\U_q\subsetneqq\overline\U_{q}=\V_{q}$ 
& 2 
& $B_q=\emptyset$\\
			
${\U}\setminus\{2\}$ 
& $\U_q\subsetneqq\overline\U_{q}=\V_{q}$ 
& 2 
& $2$\\
			
$\overline{\U}\setminus\U$ 
& $\U_q\subsetneqq\overline\U_{q}=\V_{q}$ 
& $\aleph_0$ 
&$\aleph_0$\\
			
$\V\setminus\overline{\U}$ 
& $\U_{q}=\overline{\U}_{q}\subsetneqq \V_{q}$ 
& $\aleph_0$ 
& $B_q=\emptyset$\\
			
$(1,2]\setminus\V$ 
& $\U_q=\overline\U_{q}=\V_{q}$ 
& $A_q=\emptyset$ 
& $B_q=\emptyset$\\		
\hline
\end{tabular}	
\end{center}
\caption{Overview of the equal-base case
} \label{t:one-base-A-B}
\end{table*}

Now we proceed to the formulation of our generalizations to  double-base expansions.
Since every system $S = \set{(d_0, q_0),(d_1, q_1)}$ is isomorphic to $S = \set{(0, q_0),(1, q_1)}$ by \cite[Lemma 3.1]{KomSteZou2024}, throughout this paper we restrict ourselves to the simpler system $S=\set{(0, q_0),(1, q_1)}$, i.e., we consider expansions of the form
\begin{equation*}\label{14}
x=\pi_Q((i_k))
:=\sum_{k=1}^{\infty}\frac{i_k}{q_{{i_1}}q_{{i_2}}\cdots q_{{i_k}}},\quad  (i_k)\in \{0,1\}^\infty,
\end{equation*}
where $Q:=(q_0,q_1)\in(1,\infty)^2$ is a given \emph{double-base}.
In the equal-base case $q_0=q_1$ they reduce to the expansions \eqref{13}.

We recall from \cite{KomLuZou2022, KomSteZou2024} that
\begin{equation*}
0=\pi_Q(0^{\infty})\le\pi_Q((i_k))\le
\pi_Q(1^{\infty})=\frac{1}{q_1-1}
\end{equation*}
for every sequence $(i_k)$; therefore we now define
$J_Q:=[0,\frac{1}{q_1-1}]$.
The role of the interval $(1,2]$ of bases $q$ is taken by the set
\begin{equation*}
\A:=\set{Q=(q_0,q_1)\in (1,\infty)^2\ :\ q_0+q_1\ge q_0q_1}
\end{equation*}
(see Figure 1) because
\begin{equation*}
\set{\pi_Q((d_i))\ :\ (d_i)\in\set{0,1}^{\infty}}=J_Q
\Longleftrightarrow Q\in \A.
\end{equation*}

Furthermore, if $Q\in \A$, then every $x\in J_Q$ has a (lexicographically) largest  expansion $b(x,Q)=(b_i(x,Q))$, a largest infinite expansion $a(x,Q)=(a_i(x,Q))$, a smallest co-infinite expansion $m(x,Q)=(m_i(x,Q))$, and a smallest  expansion $l(x,Q)=(l_i(x,Q))$.
They are called the \emph{greedy}, \emph{quasi-greedy}, \emph{quasi-lazy} and \emph{lazy}  expansions of $x$ (in the double-base $Q$), respectively.
Finally, a sequence is called \emph{greedy} (\emph{quasi-greedy}, \emph{quasi-lazy}, \emph{lazy}) if it is the \emph{greedy} (\emph{quasi-greedy}, \emph{quasi-lazy}, \emph{lazy}) expansion of some number $x\in J_Q$.

For simplicity, instead of
\begin{equation*}
b(x,Q),\quad
a(x,Q),\quad
m(x,Q)\qtq{and}
l(x,Q)
\end{equation*}
we often write 
\begin{equation*}
b(x)=(b_i(x)),\quad
a(x)=(a_i(x)),\quad
m(x)=(m_i(x))\qtq{and}
l(x)=(l_i(x))
\end{equation*}
when $Q$ is fixed, and even
\begin{equation*}
(b_i),\quad (a_i),\quad (m_i)\qtq{and}(l_i)
\end{equation*}
when both $Q$ and $x$ are given.

The role of the critical base $q=2$ is taken over by the double-bases belonging to the curve
\begin{equation*}
\C:=\set{Q=(q_0,q_1)\in (1,\infty)^2\ :\ q_0+q_1=q_0q_1};
\end{equation*}
see Figure 1 again.

\begin{figure}
\centering
\begin{tikzpicture}
		\begin{axis}[ 
			xmin=1, xmax=6, 
			ymin=1, ymax=6, 
			axis lines=middle, 
			xlabel={$q_0$}, 
			ylabel={$q_1$}, 
			xtick={1,2,3,4,5}, 
			ytick={1,2,3,4,5}, 
			domain=1.01:6, 
			samples=200, 
			axis line style={-stealth}, 
			xlabel style={at={(ticklabel* cs:1)}, anchor=north west}, 
			ylabel style={at={(ticklabel* cs:1)}, anchor=south east}, 
			]
			
			\addplot [
			thick,
			blue,
			] {x / (x - 1)}; 
			
			fill between [of=curve and top];
						
		\end{axis}
		\begin{scope}
			\node at (-0.24,-0.335) {$1$};
		\end{scope}
		
		\draw [thick] (0,0) -- (1.25,1.25);
\end{tikzpicture}
\caption{The blue curve is  $\C$, the region below $\C$ is $\A\setminus\C$; the black segment shows the classical case $q_0=q_1$.}
\end{figure}
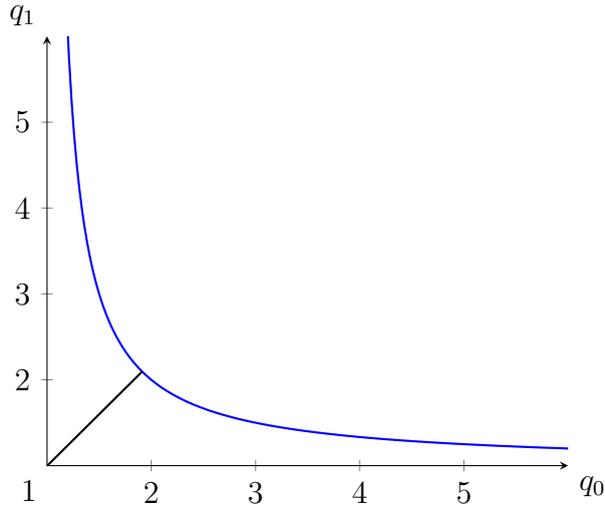

Observe that
\begin{equation*}
q_0+q_1=q_0q_1\Longleftrightarrow \frac{1}{q_0}+\frac{1}{q_1}=1,
\end{equation*}
so that $\C$ is formed by the pairs of \emph{conjugate exponents} in Young's classical inequality.

It was shown in \cite{KomLuZou2022} that the role played by $1$ and $1/(q_1-1)-1$ is now taken over by the two numbers
\begin{equation*}
r_Q:=\frac{q_0}{q_1}\qtq{and} \ell_Q:=\frac{q_1}{q_0(q_1-1)}-1.
\end{equation*} 
We let $\alpha(Q)$ and $\mu(Q)$ denote the quasi-greedy expansion of $r_Q$ and the quasi-lazy expansion of $\ell_Q$, respectively.
When  $Q$ is fixed, also write $\alpha=(\alpha_i)$ and $\mu=({\mu}_i)$ for simplicity.

\begin{remark}\label{r14}\ 
We often use the following observations in the sequel.

\begin{enumerate}[\upshape (i)]
\item If $q_0+q_1<q_0q_1$, then $r_Q$ and $\ell_Q$ have no expansions because
\begin{equation*}
r_Q=\frac{q_0}{q_1}>\frac{1}{q_1-1}\qtq{and}
\ell_Q=\frac{q_1}{q_0(q_1-1)}-1<0
\end{equation*}
by a direct computation.

\item If $q_0+q_1=q_0q_1$, i.e., if $Q\in\C$, then  $r_Q=1/(q_1-1)$ and $\ell_Q=0$.
They have the unique expansions $1^{\infty}$ and $0^{\infty}$, respectively, so that
\begin{equation*}
\alpha(Q)=1^{\infty}\qtq{and}
\mu(Q)=0^{\infty}.
\end{equation*}

\item If $q_0+q_1>q_0q_1$, i.e., if $Q=(q_0,q_1)\in\A\setminus\C$, then  $r_Q$ and $\ell_Q$ belong to the interior of the interval $J_Q$ by a similar computation, and hence their expansions are different from $1^{\infty}$ and $0^{\infty}$.

Furthermore, $r_Q>1/q_1$ and $1/(q_1-1)>\ell_Q$ by a direct computation; this implies by the definition of the quasi-greedy and quasi-lazy algorithms (we recall them at the beginning of Section \ref{s2}) that $\alpha(Q)$  starts with 1, and $\mu(Q)$  starts  with 0.

Therefore we have
\begin{equation*}
0^{\infty}\prec\mu(Q)\prec\alpha(Q)\prec 1^{\infty}.
\end{equation*}

\item A direct computation shows that
\begin{equation*}
\pi_Q\bj{0\alpha(Q)}=\frac{1}{q_1}
=\pi_Q(10^\infty)\qtq{and}
\pi_Q\bj{1\mu(Q)}=\frac{1}{q_0(q_1-1)}
=\pi_Q(01^\infty).
\end{equation*}

\item We show in Remark \ref{r22} (i)--(ii) below that 
\begin{equation*}
\mu\preceq\sigma^i(\mu)\qtq{and}\sigma^j(\alpha)\preceq\alpha\qtq{for all}i,j\in\N_0.
\end{equation*}
\end{enumerate}
\end{remark}

In this paper, $\N$ and $\N_0$ denote the sets of positive and nonnegative integers, respectively. 
In \cite{HuBarZou2024} the sets $\U, \overline{\U}, \V$ have been extended to the framework $Q=(q_0,q_1)\in\A$ as follows:

\begin{align*}
&\U:=\set{Q\in \A: \ell_Q\text{ and } r_Q \text{ have  unique expansions}},\\
&\overline{\U}\text{ is the topological closure of }\U,\\
&\V:=\set{Q\in \A: \sigma^i(\mu(Q)) \preceq \alpha(Q) \text{ and } \sigma^j(\alpha(Q)) \succeq \mu(Q) \text{  for all }i, j\in\N},\\
\end{align*}
It was also shown that $\V$ is closed, and $\U\subsetneqq\overline{\U}\subsetneqq\V$.

The above asymmetry between the definitions of $\U$ and $\V$ is only apparent:

\begin{proposition}\label{p15qqq}\
Let $Q\in \A$.
\begin{enumerate}[\upshape (i)]
\item $Q$ belongs to $\U$ if and only if
\begin{equation*}
\sigma^i(\mu(Q)) \prec \alpha(Q) \text{ and } \sigma^j(\alpha(Q)) \succ \mu(Q) \text{  for all }i, j\in\N.
\end{equation*}

\item $Q$ belongs to $\V$ if and only if $\ell_Q$ and $r_Q$ have  unique doubly infinite expansions.
\end{enumerate}
\end{proposition}

Now we extend the definition of the sets $\U_q, \overline{\U}_q$ and $\V_q$ to  all double-bases $Q\in \A$ as  follows:
\begin{align*}
&\U_Q \text{ is the set of numbers }x\in J_Q\text{  with an expansion }(x_i)\text{ satisfying } \\ 
&\qquad\sigma^j((x_i)) \prec \alpha(Q) \quad \text{whenever } x_j = 0, \text{ and}\\
&\qquad\sigma^j((x_i)) \succ \mu(Q) \quad \text{whenever } x_j = 1.\\
&\overline{\U}_Q\text{ is the topological closure of }\U_Q,\\
&\V_Q \text{ is the set of numbers }x\in J_Q\text{ satisfying } \\ 
&\qquad\sigma^j(m(x))\preceq\alpha(Q) \quad \text{ whenever }m_j(x)=0, \text{ and}\\
&\qquad\sigma^j(a(x))\succeq\mu(Q) \quad \text{ whenever }a_j(x)=1.
\end{align*}

\begin{remark}\label{r16}
It follows from the definitions that $\U_Q\subseteqq\V_Q$, and $\V_Q\setminus\U_Q$ is a countable set.
\end{remark}

The following alternative descriptions hold:

\begin{proposition}\label{p17qqq}
Let $Q=(q_0,q_1) \in \A$.

\begin{enumerate}[\upshape (i)]
\item $\U_Q=\set{x\in J_Q\ :\ x\text{ has a unique expansion}}$.

\item $\V_Q=\set{x\in J_Q\ :\ x\text{ has at most  one doubly infinite expansion}}$.

\item $\V_Q=J_Q$\quad if\quad $Q\in \C$.

\item  $\V_Q=\set{x\in J_Q\ :\ x\text{ has a unique doubly infinite expansion }}$\quad if\quad $Q \in \A\setminus\C$.
\end{enumerate}
\end{proposition}

In order to extend Theorems \ref{t:results1 of DeVKomLor2022} and \ref{t:results2 of DeVKomLor2022} to double-base expansions, we need to distinguish twelve classes of double-bases in $\A$.
In the statement of the following lemma \emph{we use an exceptional convention:} when we write
\begin{equation*}
\mu\preceq\sigma^{i}(\mu)\qtq{for all}i\in\N,
\end{equation*}
then we assume not only that these weak inequalities hold, but also that equality holds for at least one $i\in\N$.
Similar conventions are adopted when we write
\begin{equation*}
\sigma^{i}(\mu)\preceq\alpha,\quad
\mu\preceq\sigma^{j}(\alpha)\qtq{and}
\sigma^{j}(\alpha)\preceq\alpha.
\end{equation*}
Using this convention the twelve cases of the following lemma are disjoint:

\begin{lemma}\label{l:HuBarZou}\label{l18qqq}
Let $Q\in \A$, and write $\mu=(\mu_i):=\mu(Q)$ and $\alpha=(\alpha_i):=\alpha(Q)$ for brevity.
Consider the following conditions:
\begin{enumerate}[\upshape (i)]
\item $\mu\prec\sigma^{i}(\mu)\prec\alpha\text{ and }\mu\prec\sigma^{j}(\alpha)\prec\alpha\qtq{for all}i,j\in\N;\label{21}$
\item $\mu\prec\sigma^{i}(\mu)\prec\alpha\text{ and }\mu\preceq \sigma^{j}(\alpha)\prec\alpha\qtq{for all}i,j\in\N;
\label{22}$
\item $\mu\prec\sigma^{i}(\mu)\preceq\alpha\text{ and }\mu\prec\sigma^{j}(\alpha)\prec\alpha\qtq{for all}i,j\in\N;\label{25}$
\item $\mu\prec\sigma^{i}(\mu)\prec\alpha\text{ and }\mu\prec\sigma^{j}(\alpha)\preceq\alpha\qtq{for all}i,j\in\N;\label{23}$
\item $\mu\preceq\sigma^{i}(\mu)\prec\alpha\text{ and }\mu\prec\sigma^{j}(\alpha)\prec\alpha\qtq{for all}i,j\in\N;\label{26}$
\item $\mu\prec\sigma^{i}(\mu)\preceq \alpha\text{ and }\mu\prec\sigma^{j}(\alpha)\preceq\alpha\qtq{for all}i,j\in\N;\label{24}$
\item $\mu\preceq\sigma^{i}(\mu)\prec\alpha\text{ and }\mu\preceq\sigma^{j}(\alpha)\prec\alpha\qtq{for all}i,j\in\N;\label{27}$
\item $\mu\preceq\sigma^{i}(\mu)\prec\alpha\text{ and }\mu\prec\sigma^{j}(\alpha)\preceq\alpha\qtq{for all}i,j\in\N;\label{28}$
\item $\mu\preceq\sigma^{i}(\mu)\preceq\alpha\text{ and }\mu\preceq\sigma^{j}(\alpha)\preceq \alpha\qtq{for all}i,j\in\N;\label{29}$
\item $\mu_i=0$ \text{ and } $\sigma^i(\mu)\succ\alpha\text{ for at least one } i\in \N, \text{ and } \mu\prec\sigma^{j}(\alpha)\qtq{for all}j\in\N;\label{30}$
\item $\sigma^{i}(\mu)\prec\alpha\qtq{for all}i\in\N, \text{ and }
\alpha_j=1 \text{ and }\mu\succ\sigma^j(\alpha)\qtq{for at least one}j\in\N;\label{31}$
\item \text{There exist }$i,j\in\N$\text{ such that }$\mu_i=0$, $\sigma^i(\mu)\succ\alpha$, $\alpha_j=1$ and $\mu\succ\sigma^j(\alpha)$.\label{32}
\end{enumerate}
Then
\begin{align*}
Q\in \C&\Longrightarrow (\mu,\alpha)\text{ satisfies \eqref{28}},\\
Q\in \U \setminus\C&\Longleftrightarrow (\mu,\alpha)\text{ satisfies \eqref{21}},\\
Q\in \overline{\U}&\Longleftrightarrow (\mu,\alpha)\text{ satisfies \eqref{21}--\eqref{28}},\\
Q\in \V&\Longleftrightarrow (\mu,\alpha)\text{ satisfies \eqref{21}--\eqref{29}},\\
Q\in \A \setminus\V&\Longleftrightarrow (\mu,\alpha)\text{ satisfies \eqref{30}--\eqref{32}}.
\end{align*}
\end{lemma} 

Lemma \ref{l:HuBarZou} extends \cite[Proposition 3.3 and Lemmas 3.4,  5.4, 5.6]{HuBarZou2024} where $\V$ was partitioned into the sets satisfying the conditions (i)--(ix).
The remaining part of Lemma \ref{l:HuBarZou} on the partition of $\A\setminus\V$ into the sets satisfying the conditions (x)--(xii) will be proved in Lemma \ref{l:outV1}, in the last section of the paper, and will only be used there.

We show in Example \ref{e18} that all cases of Lemma \ref{l:HuBarZou} may occur.

\begin{remark}\label{r19}
Since there are only countable many periodic sequences, the sets of double-bases satisfying the condition (viii) or (ix) are countable.

The sets of double-bases satisfying condition (vi) or (vii) are also countable.
By symmetry we prove this for the condition (vi).
Since $\alpha$ is periodic by assumption, there are only countably many choices for $\alpha$.
Furthermore, for each fixed $\alpha$ there are only countable many choices for $\mu$ because $\mu$ ends with $\alpha$.

We show in Example \ref{e72} that the remaining eight sets are uncountable.
We recall from \cite{HuBarZou2024} that the Hausdorff dimension of $\overline{\U}\setminus\U$ is at least one.
\end{remark}

\begin{remark}\label{r110qqq}
In the equal-base case $q_0=q_1$ where $\mu(Q)$ is the reflection of $\alpha(Q)$, only the four cases (i), (viii), (ix) and (xii) of Lemma \ref{l:HuBarZou} may occur, corresponding to the cases $q\in\U$, $q\in\overline{\U}\setminus\U$, $q\in\V\setminus\overline{\U}$ and $(1,2]\setminus\V$, respectively, while $q\in\C$ corresponds to the case $q=2$.

The results of this paper show that various new phenomena occur in the remaining eight cases with respect to the classical case developed in \cite{DeVKom2009} and \cite{DeVKomLor2022}.
\end{remark}

Finally, we generalize the sets $A_q$ and $B_q$ to all $Q \in \A$:
\begin{align*}
&A_Q:=\{x\in \V_Q:\sigma^j(a(x))=\alpha(Q) \text{ for at least one digit }a_j(x)=0\},\\
&B_Q:=\{x\in \V_Q:\sigma^j(m(x))=\mu(Q) \text{ for at least one digit }m_j(x)=1\}.
\end{align*}
It follows the lexicographic characterizations of $\U_Q$ and $\V_Q$ that
\begin{equation*}
A_Q\cup B_Q=\V_Q\setminus\U_Q.
\end{equation*}

An alternative description is the following:

\begin{proposition}\label{p111qqq}
Let $Q \in \A$.
Then
\begin{align*}
&A_Q:=\set{x\in \V_Q: \text{ it's greedy expansion is finite}},\\
&B_Q:=\set{x\in \V_Q: \text{ it's lazy expansion is co-finite}}.
\end{align*}
\end{proposition}

\begin{remark}\label{r112}
It follows from Proposition \ref{p111qqq} that our new definition reduces to the old one in the equal-base case if $q\in(1,2)$.
For $q=2$ the two definitions are different: while $A_{2,2}=A_2$ is a countably infinite set, $B_{2,2}=A_{2,2}$, and  $B_2=\emptyset$.

While in the equal-base case $A_q$ and $B_q$ form a disjoint partition of $\V_q\setminus\U_q$, now $A_Q$ and $B_Q$  cover $\V_Q\setminus\U_Q$ with a possible overlap; see Tables \ref{t:one-base-A-B} and \ref{t:two-base-A-B}, and Examples \ref{e:ABnondisjoint} below.
\end{remark}

Now we are ready to state our  main results.
In the following theorems we refer to the conditions \eqref{21}--\eqref{31} of Lemma \ref{l:HuBarZou}, and write $(\mu,\alpha):=(\mu(Q),\alpha(Q))$ for brevity.

\begin{theorem}\label{T:expansion}\label{t113}\
\begin{enumerate}[\upshape (i)]
\item If $Q\in\U $, i.e., if $q\in\C$ or $(\mu,\alpha)$ satisfies \eqref{21}, then every $x\in\V_Q\setminus\U_Q$ has exactly two expansions.

\item  Let $Q\in\overline{\U }\setminus{\U }$.

\begin{enumerate}[\upshape (a)]
\item If $(\mu,\alpha)$ satisfies \eqref{22} or \eqref{25}, then every  $x\in\V_Q\setminus\U_Q$ has two or three  expansions.

\item If $(\mu,\alpha)$ satisfies \eqref{23} or \eqref{26}, then every  $x\in\V_Q\setminus\U_Q$ has two or $\aleph_0$ expansions.

\item If $(\mu,\alpha)$ satisfies \eqref{24} or \eqref{27} or \eqref{28}, then every $x\in\V_Q\setminus\U_Q$ has exactly $\aleph_0$ expansions.
\end{enumerate} 

\item If $(\mu,\alpha)$ satisfies \eqref{29}, i.e., if $Q\in{\V \setminus\overline{\U }}$, then every  $x\in\V_Q\setminus\U_Q$ has exactly $\aleph_0$ expansions.

\item If $(\mu,\alpha)$ satisfies \eqref{30} or \eqref{31}, then every  $x\in\V_Q\setminus\U_Q$ has two or $\aleph_0$ expansions.
\end{enumerate} 
\end{theorem}

\begin{remark}\label{r114}
More precise results will be given in Proposition \ref{p38} and Lemma \ref{l62} for the cases (ii-a), (ii-b) and (iv).
We recall that these cases do not occur in the classical case $q_0=q_1$ where $\mu(Q)$ is the reflection of $\alpha(Q)$.

The case \eqref{32} is absent from Theorem \ref{T:expansion}: in fact, we have $\U_Q=\V_Q$ in this case by Theorem \ref{T:topology1} (viii).
\end{remark}

The following theorem gives the relevant topological properties of sets $\U_Q$ and $\V_Q$. 
We write $(\mu,\alpha)$ instead of $(\mu(Q),\alpha(Q))$ for brevity.
 
\begin{theorem}\label{T:topology1}\label{t115} 
Let $Q\in \A.$
	
\begin{enumerate}[\upshape (i)]
\item $\V_Q$ is closed, and $\U_Q\subseteq\overline{\U}_Q\subseteq\V_Q$. 

\item If $(\mu,\alpha)$ satisfies one of the conditions \eqref{21}--\eqref{31}, then $|\V_Q\setminus\U_Q|=\aleph_0$, and $\V_Q\setminus\U_Q$ is dense in $\V_Q$.
		
\item If $(\mu,\alpha)$ satisfies \eqref{29}, i.e., if $Q\in{\V \setminus\overline{\U }}$, then  $\U_Q$ is closed, and
$\V_Q\setminus\U_Q$ is a discrete set, and $\V_Q$ is not a Cantor set.

\item If $Q \in \C$, then $\overline\U_Q=\V_Q=J_Q$.

\item If $(\mu,\alpha)$ satisfies \eqref{21} or \eqref{23} or \eqref{26} or \eqref{28}$\setminus\C$, then $\U_Q\subsetneqq\overline{\U}_Q=\V_Q$, and $\V_Q$ is a Cantor set.

\item If $(\mu,\alpha)$ satisfies \eqref{22} or \eqref{25} or \eqref{24} or \eqref{27}, then $\overline{\U}_Q\subsetneqq\V_Q$, and $\V_Q$ is not a Cantor set.

Furthermore,
\begin{equation*}
\V_Q\setminus\overline\U_Q\text{ is discrete }\Longleftrightarrow
\U_Q\text{ is closed }\Longleftrightarrow
\begin{cases}
1/(q_0(q_1-1))\notin\overline\U_Q&\text{in cases \eqref{22}  and \eqref{27}},\\
1/q_1\notin\overline\U_Q&\text{in cases \eqref{25} and \eqref{24}.}
\end{cases}
\end{equation*}

\item If $(\mu,\alpha)$ satisfies  \eqref{30} or \eqref{31}, then $\U_Q\subsetneqq\V_Q$, 
and 
\begin{equation*}
\U_Q\text{ is closed }\Longleftrightarrow
\begin{cases}
1/(q_0(q_1-1))\notin\overline\U_Q&\text{in case \eqref{31}},\\
1/q_1\notin\overline\U_Q&\text{in case \eqref{30}.}
\end{cases}
\end{equation*}
Furthermore, $\V_Q\setminus\U_Q$ is a non-empty discrete set if $\U_Q$ is closed, and $\overline\U_Q=\V_Q$ otherwise.

\item If $(\mu,\alpha)$ satisfies \eqref{32},  then $\U_Q=\overline\U_Q=\V_Q$. 
\end{enumerate}
\end{theorem}

Table \ref{t:two-base-A-B} gives an overview of the main topological properties of $\U_Q$, $\overline\U_Q$ and $\V_Q$ in the double-base case, proved in Theorems \ref{T:expansion} and \ref{T:topology1}, with some further information proved in Sections \ref{s:3}--\ref{s7} below.

In Table \ref{t:two-base-A-B} $|A'_x|$ and $|B'_x|$ denote the number of expansions of each $x\in A_Q$ and $x\in B_Q$, respectively.

Comparing to Table \ref{t:one-base-A-B} we see that the double-base case is much more complex.
For example, contrary to the equal-base case,
\begin{itemize}
\item $\U_Q$ may be closed even if $Q\in\overline\U$;

\item $\U_Q$ may be not closed even if $Q\in\A\setminus\overline\U$;

\item there exist double-bases for which the three sets $\U_Q$, $\overline\U_Q$ and $\V_Q$ are different;

\item there exist double-bases for which $V_Q\setminus\overline{\U}_Q$ is nonempty and non-discrete;

\item $A_Q$ and $B_Q$ are nonempty for all $Q\in\V$;

\item $A_Q$ and $B_Q$ may cover $V_Q\setminus\U_Q$ with an overlap.
\end{itemize}
\
\begin{table*}[h]
\begin{center}
\begin{tabular}{|c|c|c|c|c|c|}
\hline
Case
& $Q\in $
& Inclusions 
& $A_Q$ and $B_Q$ 
& $|A'_x|$ 
&$|B'_x|$\\ 
			
\hline
$\C$ 
& $\C$ 
& $\U_Q\subsetneqq\overline\U_Q=\V_Q$ 
& $A_Q=B_Q$ 
& 2 
& 2\\
			
\eqref{21} 
& $\U\setminus\C$ & $\U_Q\subsetneqq\overline\U_Q=\V_Q$ 
& $A_Q\cap B_Q=\emptyset$ 
& 2 
& 2\\
			
\eqref{22} 
& $\overline{\U}\setminus\U$ 
& $\U_Q=\overline{\U}_Q\subsetneqq\V_Q$ or $\U_Q\subsetneqq\overline{\U}_Q\subsetneqq\V_Q$ 
& $A_Q\subsetneqq B_Q$ 
& 3 
& $2$ or $3$\\

\eqref{25} 
& $\overline{\U}\setminus\U$ 
& $\U_Q=\overline{\U}_Q\subsetneqq\V_Q$ or $\U_Q\subsetneqq\overline{\U}_Q\subsetneqq\V_Q$ 
& $B_Q\subsetneqq A_Q$ 
& $2$ or $3$ 
& 3\\

\eqref{23} 
& $\overline{\U}\setminus\U$ 
& $\U_Q\subsetneqq\overline\U_Q=\V_Q$ 
& $A_Q\cap B_Q=\emptyset$ 
& $\aleph_0$ 
& 2\\
			
\eqref{26} & $\overline{\U}\setminus\U$ &   $\U_Q\subsetneqq\overline\U_Q=\V_Q$ & $A_Q\cap B_Q=\emptyset$ & 2 & $\aleph_0$\\
			
\eqref{24}
& $\overline{\U}\setminus\U$ 
& $\U_Q=\overline{\U}_Q\subsetneqq\V_Q$ or $\U_Q\subsetneqq\overline{\U}_Q\subsetneqq\V_Q$ 
& $B_Q\subsetneqq A_Q$ 
& $\aleph_0$ 
& $\aleph_0$\\
			
\eqref{27} 
& $\overline{\U}\setminus\U$ 
& $\U_Q=\overline{\U}_Q\subsetneqq\V_Q$ or $\U_Q\subsetneqq\overline{\U}_Q\subsetneqq\V_Q$
& $A_Q\subsetneqq B_Q$ 
& $\aleph_0$ 
& $\aleph_0$ \\
			
\eqref{28}$\setminus\C$ 
& $\overline{\U}\setminus\U$ 
& $\U_Q\subsetneqq\overline\U_Q=\V_Q$ 
& $A_Q\cap B_Q=\emptyset$ 
& $\aleph_0$ 
& $\aleph_0$ \\
			
\eqref{29}
& $\V\setminus\overline{\U}$
& $\U_Q=\overline{\U}_Q\subsetneqq \V_Q$ 
& $A_Q=B_Q$ 
& $\aleph_0$ 
& $\aleph_0$\\
			
\eqref{30}
& $\A\setminus\V$ 
& $\U_q\subsetneqq\overline\U_{q}=\V_{q}$ or $\U_q=\overline\U_{q}\subsetneqq\V_{q}$ 
& $B_Q\subsetneqq A_Q$ 
& $\aleph_0$ or $2$ 
& $B_q=\emptyset$\\
			
\eqref{31}
& $\A\setminus\V$ 
& $\U_q\subsetneqq\overline\U_{q}=\V_{q}$ or $\U_q=\overline\U_{q}\subsetneqq\V_{q}$ 
& $A_Q\subsetneqq B_Q$ 
& $A_q=\emptyset$ 
& $2$ or $\aleph_0$\\
			
\eqref{32}
& $\A\setminus\V$ 
& $\U_Q=\overline{\U}_Q=\V_Q$ 
& $A_Q=B_Q=\emptyset$ 
& $A_q=\emptyset$ 
& $B_q=\emptyset$\\
			
\hline
\end{tabular}	
\end{center}
\caption{Overview of the double-base case
}
\label{t:two-base-A-B}
\end{table*}

\begin{corollary}\label{c116qqq}
Let $Q\in\A$. 
The following relations hold:
\begin{align*}
&Q\in \U \Longleftrightarrow \ell_Q\text{ and  } r_Q\in \U_Q,\\
&Q\in \V \Longleftrightarrow \ell_Q \text{ and }r_Q\in \V_Q,\\
&\ell_Q \text{ and }r_Q\in \overline{\U}_Q\Longrightarrow Q\in\overline{\U}.
\end{align*}
\end{corollary}	 
  
\begin{example}\label{e117}
The last implication cannot be reversed in general.
For example, if $\mu(Q)=(01)^\infty$ and $\alpha(Q)=11(01)^\infty$,\footnote{Case \eqref{27} of Lemma \ref{l:HuBarZou}.} then $Q\in\overline{\U }$, but none of $\ell_Q$ and $r_Q$ belongs to $\U_Q=\{0, 1/(q_1-1)\}$.

We recall from \cite[Corollary 1.8 ]{DeVKomLor2022} that  the reverse implication holds if $q_0=q_1$.
\end{example}

Finally we describe the finer structure of $\V_Q$ and $\U_Q$ for $Q \in \V\setminus \C$ and $Q\in{\V \setminus\overline{\U }}$, respectively.

\begin{theorem}\label{T:topology2}\label{t118}\
Let $Q \in \V\setminus \C$.

\begin{enumerate}[\upshape (i)]
\item $J_Q\setminus \V_Q$ is a  union of $\aleph_0$  disjoint open sets $(x_L,x_R)$, where $x_L$ and $x_R$ run over $A_Q$ and $B_Q$, respectively. 
Furthermore, 
\begin{equation*}
b(x_L)=b_1\cdots b_{n-1}10^\infty\Longleftrightarrow
l(x_R)=b_1\cdots b_{n-1}01^\infty.
\end{equation*}

\item If $(\mu,\alpha)$ satisfies the condition \eqref{29}, i.e., if $Q\in{\V \setminus\overline{\U }}$, then $J_Q\setminus \U_Q$ is an open set. 
Furthermore, each connected component $(x_L,x_R)$ of $J_Q\setminus \U_Q$ contains infinitely many elements of $\V_Q$, forming an increasing sequence $(x_k)^\infty_{k=-\infty}$ satisfying 
\begin{equation*}
x_k\rightarrow x_L \text{ as } k\rightarrow-\infty,\qtq{and}x_k\rightarrow x_R \text{ as } k\rightarrow\infty.
\end{equation*}
Moreover, each $x_k$ has a finite greedy expansion
\begin{equation*}
b(x_k)=b_1\cdots b_n0^\infty\text{ with } b_n=1,\qtq{and then}
a(x_{k+1},q)=b_1\cdots b_n\mu(Q).
\end{equation*}	
\end{enumerate}
\end{theorem}

The rest of the paper is organized as follows.
In Section 2 we recall some relevant results on double-base expansions, and  we prove Propositions \ref{p15qqq} and \ref{p17qqq}.
In Section 3 we prove Proposition \ref{p111qqq}, and Theorem \ref{T:expansion} (i)--(iii).
Theorem \ref{T:topology1} (i)--(iv) (except (ii) for $Q\in\A\setminus\V$) and Corollary \ref{c116qqq} are proved in  Section 4.
Theorems \ref{T:topology1} (v)--(vi) and \ref{T:topology2} are proved in  Section 5, and the remaining parts of Theorems \ref{T:expansion} and \ref{T:topology1} are proved in  Section 6; the section titles give more precision.
Finally, in Section 7 we illustrate our theorems by many examples.

The results of this paper show that many important theorems of the classical theory may be generalized to double-bases. 
There remains a lot of other results on equal-base expansions that could similarly be extended to the more general framework.
	
\section{Proof of Propositions \ref{p15qqq} and \ref{p17qqq}}\label{s2}

For the convenience of the reader we recall from \cite{KomLuZou2022} some results  concerning the  greedy, quasi-greedy, lazy and quasi-lazy expansions.
In this section we fix an arbitrary $Q=(q_0,q_1)\in \A$, and we write  
\begin{equation*}
b(x),\quad
a(x),\quad
m(x),\quad
l(x),\quad
\alpha\qtq{and}
\mu
\end{equation*}
instead of
\begin{equation*}
b(x,Q),\quad
a(x,Q),\quad
m(x,Q),\quad
l(x,Q),\quad
\alpha(Q)\qtq{and}
\mu(Q).
\end{equation*}
We recall from the introduction that
\begin{equation*}
\alpha=a(r_Q)=a\bj{\frac{q_0}{q_1}}
\qtq{and}
\mu=m(\ell_Q)=m\bj{\frac{q_1}{q_0(q_1-1)}-1}.
\end{equation*}

The greedy expansion $b(x)=(b_i)$ of every $x\in J_Q$ is obtained by the following \emph{algorithm}: if the digits $b_1,\cdots,b_{N-1}$ have been already defined for some positive integer $N$ (no assumption if $N=1$), then let $b_N$ be the largest digit in $\{0,1\}$ such that  
\begin{equation}\label{e:21}
\sum_{i=1}^{N}\frac{b_i}{q_{b_1}\cdots q_{b_{i}}} \leq x.
\end{equation} 
If we change $b_i$ to $a_i$, and we write a strict inequality in \eqref{e:21}, then we obtain the quasi-greedy expansion $a(x)=(a_i)$ of every $x\in J_Q\setminus\set{0}$.
Furthermore, $a(0)=1^{\infty}$.

Similarly, the lazy expansion $l(x)=(l_i)$ of every $x\in J_Q$ is obtained by the following algorithm: if the digits $l_1,\cdots,l_{N-1}$ have been already defined for some positive integer $N$ (no assumption if $N=1$), then let $l_N$ be the smallest digit in $\{0,1\}$ such that 
\begin{equation}\label{e:l26}
\sum_{i=1}^{N}\frac{l_i}{q_{l_1}\cdots q_{l_{i}}}+\frac{1}{q_{l_1}\cdots q_{l_{N}}(q_1-1)} \geq x .
\end{equation}
If we change $l_i$ to $m_i$, and we write a strict inequality in \eqref{e:l26}, then we obtain the quasi-lazy expansion $m(x)=(m_i)$ of every $x\in J_Q\setminus\set{1/(q_1-1)}$.
Furthermore, $m(1/(q_1-1))=0^{\infty}$.

It follows from the definitions of these expansions that
\begin{equation*}
l(x)\preceq m(x)\preceq a(x)\preceq b(x)\qtq{for every}x\in  J_Q.
\end{equation*}

\begin{lemma}\cite[Theorem 2]{KomLuZou2022}\label{l:monotonicity}
\label{l21qqq}
Fix $Q\in \A$.
\begin{enumerate}[\upshape (i)]
\item The \emph{greedy map} $x\mapsto b(x)$ is a strictly increasing bijection from $J_Q$ onto the set of all sequences $(j_i)$ satisfying
\begin{equation*}
\sigma^n((j_{i}))\prec\alpha\qtq{whenever} j_n=0.
\end{equation*}
		
\item The \emph{quasi-greedy map} $x\mapsto a(x)$ is a strictly increasing bijection from $J_Q$ onto the set of all infinite sequences $(j_i)$ satisfying
\begin{equation*}
\sigma^n((j_{i}))\preceq\alpha\qtq{whenever} j_n=0.
\end{equation*}
		
\item The \emph{lazy map} $x\mapsto l(x)$ is a strictly increasing bijection from $J_Q$ onto the set of all sequences $(j_i)$ satisfying
\begin{equation*}
\sigma^n((j_{i}))\succ\mu \qtq{whenever} j_n=1.
\end{equation*}
		
\item The \emph{quasi-lazy map} $x\mapsto m(x)$ is a strictly increasing bijection from $J_Q$ onto the set of all co-infinite sequences $(j_i)$ satisfying
\begin{equation*}
\sigma^n((j_{i}))\succeq \mu \qtq{whenever} j_n=1.
\end{equation*}
\end{enumerate}
\end{lemma}

\begin{remark}\label{r22}
Sometimes the inequalities of Lemma \ref{l21qqq} are satisfied for \emph{all} $n\ge 1$.
Two important examples are $\mu=(\mu_i):=\mu(Q)$ and $\alpha=(\alpha_i):=\alpha(Q)$ for $Q\in\A$.
\begin{enumerate}[\upshape (i)]
\item We have
\begin{equation*}
\sigma^n(\mu)\succeq\mu\qtq{for all}n\ge 0.
\end{equation*}
For the proof first we observe that if this $\sigma^k(\mu)\ge\mu$ for some $k\ge 0$, and $\mu_{k+1}=\cdots\mu_n=0$ for some $n>k$, then the inequalities trivially also holds for $n$ in place of $k$.
The case $k=0$ being obvious, it remains to observe that for any $n\ge 1$ with $\mu_n=0$ we have either $\mu_1=\cdots\mu_n=0$, or there exists a $k<n$ such that $\mu_k=1$, and $\mu_{k+1}=\cdots\mu_n=0$.

\item By reflection, we obtain from (i) that
\begin{equation*}
\sigma^n(\alpha)\preceq\alpha\qtq{for all}n\ge 0.
\end{equation*}

\item Since $\mu\prec\alpha$ by Remark \ref{r14} (v) we obtain similarly that if
\begin{align*}
\sigma^i(\mu)\prec\alpha\qtq{whenever}\mu_i=0,
\intertext{and}
\sigma^j(\alpha)\succ\mu\qtq{whenever}\alpha_j=1,
\end{align*}
then in fact both inequalities hold for all $i,j\ge 0$.

\item Similarly, if
\begin{align*}
\sigma^i(\mu)\preceq\alpha\qtq{whenever}\mu_i=0,
\intertext{and}
\sigma^j(\alpha)\succeq\mu\qtq{whenever}\alpha_j=1,
\end{align*}
then in fact both inequalities hold for all $i,j\ge 0$.
\end{enumerate}
\end{remark}

\begin{lemma}\cite[Proposition 13]{KomLuZou2022}\label{l:rela-quasi-gl}\label{l22qqq}
Let $x\in J_Q$.

\begin{enumerate}[\upshape (i)]
\item If $b(x)$ is infinite, then $a(x)=b(x)$.
If $b(x)=(b_i)$ has a last nonzero element $b_k=1$, then
\begin{equation*}
a(x)=b_1\cdots b_{k-1}0\alpha(Q).
\end{equation*} 

\item If $l(x)$ is co-infinite, then $m(x)=l(x)$. 
If $l(x)=(l_i)$ has a last zero element $l_k=0$, then
\begin{equation*}
m(x)=l_1\cdots l_{k-1}1\mu(Q).
\end{equation*}
\end{enumerate}
\end{lemma}

Let us consider a special case:
\begin{lemma}\label{l23}
Let $Q\in\C$. 
\begin{enumerate}[\upshape (i)]
\item For any $x\in J_Q$, then there are two possibilities:
\begin{enumerate}[\upshape (a)]
\item $x$ has a unique expansion, and it is doubly infinite.
\item $x$ has exactly two expansions:  $b(x)=m(x)$ and  $a(x)=l(x)$, and none of them is doubly infinite.
\end{enumerate}

\item $A_Q=B_Q=\V_Q\setminus\U_Q$.

\item The following relations hold:
\begin{equation*}
\U_Q\subsetneqq\overline\U_Q=\V_Q=J_Q, \qtq{and}
\abs{J_Q\setminus\U_Q}=\aleph_0.
\end{equation*} 
\end{enumerate}
\end{lemma}

\begin{proof}
(i) Since $\mu=0^\infty$ in this case, every infinite expansion is lazy by Lemma \ref{l:monotonicity} (iii).
In particular, $a(x)=l(x)$.

Similarly, since $\alpha=1^\infty$, every co-infinite expansion is greedy by Lemma \ref{l:monotonicity} (iii).
In particular, $m(x)=b(x)$.

It follows that if $x$ has a doubly infinite expansion, then it is necessarily equal to both $l(x)$ and $b(x)$, whence $x$ has a unique expansion.

If $b(x)$ is infinite, then $b(x)=a(x)$, and hence $b(x)=l(x)$, so that $x$ has a unique expansion.
It is doubly infinite because it is also equal to $a(x)$ and $m(x)$ by uniqueness, and therefore it is both infinite and co-infinite.

If $b(x)$ is finite, then it has the form $b(x)=b_1\cdots b_k10^{\infty}$ for some integer $k$, and then $a(x)=b_1\cdots b_k01^{\infty}$ by Lemma \ref{l:rela-quasi-gl} (i).
Since there is no sequence between $10^{\infty}$ and $01^{\infty}$, there is no expansion of $x$ between $b(x)$ and $a(x)=l(x)$.
Hence $x$ has exactly two expansions:  $b(x)=m(x)$ and  $a(x)=l(x)$, and none of them is doubly infinite by a preceding observation.
\medskip

(ii) If $x\in \V_Q\setminus\U_Q$, then the proof of (i) shows that $a(x)$ ends with $01^{\infty}$ and $m(x)$ ends with $10^{\infty}$.
$\mu(Q)=0^\infty$ and $\alpha(Q)=1^\infty$, hence $x\in A_Q$ and $x\in B_Q$ by the definition of these sets.
\medskip

(iii) Since $\mu=0^\infty$ and $\alpha=1^\infty$, $\V_Q=J_Q$ by the definition of $\V_Q$.

If $x\in J_Q\setminus\U_Q$, then $a(x)\ne b(x)$ by (i), and the set of such numbers is countable by Remark \ref{r11}.
Therefore $J_Q\setminus\U_Q$ is countable, and this implies the relation $\overline\U_Q=J_Q$.
Finally, $J_Q\setminus\U_Q$ is infinite because  $1/q_1^n$ has two expansions for every $n\in\N$: $0^{n-1}10^{\infty}$ and $0^n1^{\infty}$.
\end{proof}

Now we consider the case $Q\in \A\setminus \C$.

\begin{lemma}\label{l:doublyinfinite}\label{l24}
If $Q\in \A\setminus \C$ and $x\in J_Q$, then both expansions $a(x)$ and $m(x)$ are doubly infinite.
\end{lemma}

\begin{proof}
The numbers $x=0$ and $x=1/(q_1-1)$ have the unique expansions $0^\infty$ and $1^\infty$, respectively, and both are doubly infinite.
	
If $x\in (0, 1/(q_1-1))$,  then the expansion $a(x)\neq 1^\infty$ is infinite by definition, and it remains to show that it cannot end with $01^\infty$.
This follows from Lemma \ref{l:monotonicity} and Remark \ref{r14} because $\alpha<1^{\infty}$ if $Q\in \A\setminus \C$.

The proof for $m(x)$ is analogous.
\end{proof}

For our next lemma we recall that for any given $Q\in\A$, $\V_Q$ is the set of numbers $x\in J_Q$ satisfying the following two conditions:
\begin{align}
&\sigma^j(m(x))\preceq\alpha \quad \text{ whenever }m_j(x)=0,\label{e:24}\\
&\sigma^j(a(x))\succeq\mu \quad \text{ whenever }a_j(x)=1.\label{e:25}
\end{align}

\begin{lemma}\label{l:quasi-gr-la}\label{l25}
If $Q\in \A\setminus \C$ and $x\in J_Q$, then the following properties are equivalent:
\begin{enumerate}[\upshape (i)]
\item $x\in\V_Q$;
\item $a(x)=m(x)$;
\item $x$ has a unique doubly infinite expansion.
\end{enumerate}
\end{lemma}

\begin{proof}
(i) $\Longrightarrow$ (ii) 
If $x\in \V_Q$, then $a(x)$ is co-infinite by Lemma \ref{l:doublyinfinite}, and hence $a(x)=m(x)$ by  \eqref{e:25} and
Lemma \ref{l:monotonicity} (ii), (iv).
\medskip

(ii) $\Longrightarrow$ (iii) 
Since $a(x)=m(x)$ is doubly infinite, it remains to show that  no other expansion $c(x)$ of $x$ is doubly infinite.
This follows by recalling that every expansion $c(x)>a(x)$ of $x$ is finite because $a(x)$ is the largest infinite expansion of $x$, and every expansion $c(x)<m(x)$ of $x$ is co-finite because $m(x)$ is the smallest co-infinite expansion of $x$.
\medskip

(iii) $\Longrightarrow$ (ii) 
If $x$ has a unique doubly infinite expansion, then $a(x)=m(x)$ by Lemma \ref{l:doublyinfinite}.
\medskip

(ii) $\Longrightarrow$ (i) 
If $a(x)=m(x)$, then Lemma \ref{l:monotonicity} (ii), (iv) imply \eqref{e:24} and \eqref{e:25}.
\end{proof}

\begin{proof}[Proof of Proposition \ref{p15qqq}]
(i) follows from Lemma \ref{l:monotonicity} and Remark \ref{r22}.

(ii) If $Q\in \C$, then $\mu=0^\infty$ and $\alpha=1^\infty$ by Remark \ref{r14}.
Hence $\ell_Q$ and $r_Q$ have unique doubly infinite double-base expansions, and the definition of $Q\in \V$ is also trivially satisfied.

Henceforth we assume that $Q\in \A \setminus\C$.
It follows from Lemma \ref{l:monotonicity} and the definition of $\V $ that $Q\in \V$ if and only if $m(\ell_Q)=a(\ell_Q)$ and $m(r_Q)=a(r_Q)$.
By Lemma \ref{l:quasi-gr-la} this is equivalent to the property that $\ell_Q$ and $r_Q$ have unique doubly infinite expansions.
\end{proof}

\begin{proof}[Proof of  Proposition \ref{p17qqq}]
(i) follows from Lemma \ref{l:monotonicity} (i) and (iii).

(iii) If $Q\in \C$, then $\mu=0^\infty$ and $\alpha=1^\infty$, and hence the definition of $\V_Q$ is trivially satisfied for every $x\in J_Q$.

(iv) It is contained in Lemma \ref{l:quasi-gr-la}.

(ii) This follows from (iv) if $Q \in \A\setminus\C$, and from (iii) and Lemma \ref{l23} if $Q\in \C$.
\end{proof}

\section{Proof of Proposition \ref{p111qqq} and Theorem \ref{T:expansion} (i)--(iii)}\label{s:3}\label{s3}

In this section we determine the number of expansions of every $x\in\V_Q\setminus\U_Q$ when $Q\in \V $. 
The situation being rather complex, we summarize the results to be proved in Table \ref{t:two-base-A-B}; see also Lemma \ref{l:A-B-set} and Proposition \ref{p38}. 
Where we write $\mu$ and $\alpha$ instead of $\mu(Q)$ and $\alpha(Q)$, and we use the notations 
\begin{align*}
A_x'&:=\{c:\pi_Q(c)=x\}\text{ if } x\in A_Q,\\
B_x'&:=\{c:\pi_Q(c)=x\}\text{ if } x\in B_Q.
\end{align*}

We recall that
\begin{align*}
&A_Q:=\{x\in \V_Q\setminus\U_Q:\sigma^j(a(x))=\alpha(Q) \text{ for at least one digit }a_j(x)=0\},\\
&B_Q:=\{x\in \V_Q\setminus\U_Q:\sigma^j(m(x))=\mu(Q) \text{ for at least one digit }m_j(x)=1\}.
\end{align*}
Furthermore, we recall the relations
\begin{align}
\sigma^j(\alpha(Q))&\preceq \alpha(Q)\text{ for all } j\ge 0\label{e:3101}
\intertext{and}
\mu(Q)&\preceq\sigma^i(\mu(Q))\text{ for all } i\ge 0.\label{e:3102} 
\end{align}

In this section we often write
\begin{equation*}
\mu(Q)=\mu=(\mu_i),\quad
m(x)=(m_i),\quad
\alpha(Q)=\alpha=(\alpha_i)\qtq{and}
a(x)=(a_i)
\end{equation*}
for brevity, when $Q\in\A$ and $x\in J_Q$ are given.

In the following lemma we refer to the conditions of Lemma \ref{l:HuBarZou}:

\begin{lemma} \label{l31}\
\begin{enumerate}[\upshape (i)]
\item If $(\mu,\alpha)$ satisfies one of the conditions \eqref{21}--\eqref{30}, then $1/q_1^k\in A_Q$ for every $k\in\N$.

\item If $(\mu,\alpha)$ satisfies one of the conditions \eqref{21}--\eqref{29} and \eqref{31}, then $1/(q_0^k(q_1-1))\in B_Q$ for every $k\in\N$.

\item If $Q\in \V$, then $A_Q\neq \emptyset$ and $B_Q\neq \emptyset$.
\end{enumerate}
\end{lemma}

\begin{proof}
(i) Fix an arbitrary $k\ge 0$, and set $x_k:=1/q_0^kq_1$.
It follows from Lemma \ref{l:rela-quasi-gl} that 
\begin{equation*}
b(x_k)=0^k10^{\infty}\qtq{and}a(x_k)=0^{k+1}\alpha(Q).
\end{equation*}
In view of the definition of $A_Q$ it only remains to prove that $x_k\in \V_Q$.
This is true for $Q\in \C$ because then $\V_Q=J_Q$ by Lemma \ref{l23}.

Otherwise we have $0^\infty\prec\mu(Q)\preceq\sigma^j(\alpha(Q))$ for all $j\ge 1$  by Remark \ref{r14} and Lemma \ref{l:HuBarZou}.
Hence $a(x_k)=0^{k+1}\alpha(Q)$ is co-infinite, and therefore $m(x_k)=a(x_k)$ by Lemma \ref{l:monotonicity} (iv).
Applying Lemma \ref{l:quasi-gr-la} we conclude that $x_k\in \V_Q$.
\medskip

(ii) The proof is similar to that of (i).
\medskip

(iii) follows from (i) and (ii).
\end{proof}
\color{black}

\begin{lemma}\label{l:lA}\label{l33}
Let $Q\in\A$ and  $x\in A_Q$.
\begin{enumerate}[\upshape (i)]
\item There exists a  positive integer $n$ such that
\begin{equation*}
b(x)=a_1a_2\cdots a_{n-1}10^\infty\qtq{and}
a(x)=a_1a_2\cdots a_{n-1}0\alpha_1\alpha_2\cdots .
\end{equation*}

\item If $\alpha(Q)=1^\infty$ or if the inequalities in \eqref{e:3101} are strict,  then there is no expansion between $a(x)$ and $b(x)$.

\item If $\alpha(Q)\neq 1^\infty$, and equality holds in \eqref{e:3101} for a smallest positive integer $k$, then $k\ge 2$, $\alpha_k=0$, and all expansions between $a(x)$ and $b(x)$ are given by the sequences
\begin{equation*}
c^N:=a_1\cdots a_{n-1}(0\alpha_1\cdots \alpha_{k-1})^N10^\infty, \quad N=1,2,\ldots,
\end{equation*}
with $n$ as in (i).
\color{black}
\end{enumerate}
\end{lemma}	
	
\begin{proof}
(i)  By the definition of $A_Q$, $a(x)$ ends with $0\alpha(Q)$.
Since $\pi_Q(0\alpha(Q))=\pi_Q(10^{\infty})$, this implies that $a(x)$ is not the largest expansion of $x$.
Therefore $x$ has a finite greedy expansion, and we conclude by applying Lemma \ref{l:rela-quasi-gl}.
\medskip

(ii) If $\alpha(Q)=1^\infty$, then using (i) we get
\begin{equation*}
b(x)=a_1a_2\cdots a_{n-1}10^\infty\text{ and }
a(x)=a_1a_2\cdots a_{n-1}01^\infty
\end{equation*}
for some positive integer $n$.
This implies our claim because there is no sequence between $01^\infty$ and $10^\infty$.

Now assume that all  inequalities in \eqref{e:3101} are strict, and assume on the contrary that $x$ has an expansion $(x_i)$ satisfying the inequalities
\begin{equation*}
a(x)=a_1a_2\cdots a_{n-1}0\alpha_1\alpha_2\cdots 
\prec(x_i)
\prec a_1a_2\cdots a_{n-1}10^\infty=b(x).
\end{equation*}
Since $(x_i)\succ a(x)$, and $a(x)$ is the largest infinite expansion of $x$, and since $\alpha_1(Q)=1$ for every $Q\in\A$, there exists a positive integer $k$ such that $\alpha_k=0$, and 
\begin{equation*}
(x_i)=a_1a_2\cdots a_{n-1}0\alpha_1\cdots\alpha_{k-1}10^\infty.
\end{equation*}
Then
\begin{equation*}
(y_i):=a_1a_2\cdots a_{n-1}0\alpha_1\cdots\alpha_{k-1}0\alpha_1\alpha_2\cdots
=a_1a_2\cdots a_{n-1}0\alpha_1\cdots\alpha_{k-1}\alpha_k\alpha_1\alpha_2\cdots
\end{equation*}
is an infinite expansion of $x$, and therefore $(y_i)\preceq a(x)$.
This implies the inequality
\begin{equation*}
\alpha_1\alpha_2\cdots\preceq 
\alpha_{k+1}\alpha_{k+1}\cdots,
\end{equation*}
contradicting our assumption that the inequalities in \eqref{e:3101}  are strict.
\medskip
	
(iii) By our assumption we have $\alpha(Q)=(\alpha_1\cdots\alpha_k)^{\infty}$.

Furthermore, we have $k\ge 2$ and $\alpha_k=1$.
Indeed, in case $k=1$ we would obtain $\alpha(Q)=1^{\infty}$, which is excluded.
(Note that $\alpha_1=1$ for all $Q\in\A$.)
Next, in case $\alpha_k=1$ we would infer from the inequality
\begin{equation*}
\sigma^{k-1}(\alpha(Q))=1\alpha(Q)\preceq\alpha(Q),
\end{equation*}
the excluded case $\alpha(Q)=1^{\infty}$.

Using these relations we infer from (i) that
\begin{equation*}\label{e:312}
a(x)=a_1\cdots a_{n-1}(0\alpha_1\cdots\alpha_{k-1})^{\infty} \text{ and } b(x)=a_1\cdots a_{n-1}10^\infty.
\end{equation*}
It follows that the sequences $c^N$ are expansions of $x$.
Indeed, using again the relation $\alpha_k=0$, we have
\begin{align*}
\pi_Q(c^N)
&=\pi_Q\bj{a_1\cdots a_{n-1}(0\alpha_1\cdots \alpha_{k-1})^N10^\infty}\\
&=\pi_Q\bj{a_1\cdots a_{n-1}(0\alpha_1\cdots \alpha_{k-1})^N0(\alpha_1\cdots \alpha_k)^\infty}\\
&=\pi_Q\bj{a_1\cdots a_{n-1}0(\alpha_1\cdots \alpha_k)^\infty}\\
&=\pi_Q\bj{a(x)}=x.
\end{align*}
In the last step we used (i).

To complete the proof we assume on the contrary that there exists an expansion $(x_i)$ of $x$ and a positive integer $N$ such that $c^{N+1}\prec(x_i)\prec c^N$.
Hence we obtain that 
\begin{align*}
&(x_i)\text{ starts with }a_1\cdots a_{n-1}(0\alpha_1\cdots\alpha_{k-1})^N0,\\
&x_{n+kN+1}\cdots x_{n+kN+k}\succ \alpha_1\cdots\alpha_{k-1}1,\\
&x_{n+kN+1}\cdots x_{n+kN+k-1}\succ \alpha_1\cdots\alpha_{k-1},
\intertext{and}
&(x_i)\text{ and ends with }10^{\infty}.
\end{align*}
If the last nonzero digit of $(x_i)$ is $x_{\ell}=1$ with $\ell\ge n+k(N+1)+1$, then replacing $10^{\infty}$ with $0\alpha_1\alpha_2\cdots$ 
we obtain from $(x_i)$ an infinite expansion $(y_i)$ starting with
\begin{equation*}
x_1\cdots x_{n+kN+k}\succ 
a_1\cdots a_{n-1}(0\alpha_1\cdots\alpha_{k-1})^{N+1}1.
\end{equation*}
This is impossible,  because $\alpha_k=0$, and therefore $(y_i)\succ a(x)$.

It remains to consider the cases where say $\ell=n+kN+j$ with some $1\le j\le k$.
In fact, we cannot have $j=k$, because then $(x_i)=c^{N+1}$.
Thus we have $1\le j\le k-1$.

Observe that
\begin{equation*}
(y_i):=a_1\cdots a_{n-1}(0\alpha_1\cdots\alpha_{k-1})^N0x_{n+kN+1}\cdots x_{n+kN+j-1}0\alpha_1\alpha_2\cdots
\end{equation*}
is an infinite expansion of $x$, and 
\begin{equation*}
x_{n+kN+1}\cdots x_{n+kN+j-1}x_{n+kN+j}
=x_{n+kN+1}\cdots x_{n+kN+j-1}1\succ\alpha_1\cdots\alpha_j.
\end{equation*}
We distinguish two cases.
If 
\begin{equation*}
x_{n+kN+1}\cdots x_{n+kN+j-1}\succ \alpha_1\cdots\alpha_{j-1},
\end{equation*}
then
\begin{align*}
a_1\cdots a_{n-1}(0\alpha_1\cdots\alpha_{k-1})^N0x_{n+kN+1}\cdots x_{n+kN+j-1}
\succ
a_1\cdots a_{n-1}(0\alpha_1\cdots\alpha_{k-1})^N0\alpha_1\cdots\alpha_{j-1}.
\end{align*}
This implies that $(y_i)\succ a(x)$, which is impossible because $a(x)$ is the largest infinite expansion of $x$.

If
\begin{equation*}
x_{n+kN+1}\cdots x_{n+kN+j-1}=\alpha_1\cdots\alpha_{j-1},
\end{equation*}
then we have necessarily $\alpha_j=0$, and
\begin{align*}
(y_i)&=a_1\cdots a_{n-1}(0\alpha_1\cdots\alpha_{k-1})^N0x_{n+kN+1}\cdots x_{n+kN+j-1}0\alpha_1\alpha_2\cdots\\
&=a_1\cdots a_{n-1}(0\alpha_1\cdots\alpha_{k-1})^N0\alpha_1\cdots\alpha_{j-1}\alpha_j\alpha_1\alpha_2\cdots
\end{align*}
Since $1\le j<k$, using the minimality of $k$ we obtain that
\begin{equation*}
\alpha_1\cdots\alpha_{j-1}\alpha_j\alpha_1\alpha_2\cdots
\succ
\alpha_1\cdots\alpha_{j-1}\alpha_j\alpha_{j+1}\alpha_{j+1}\cdots,
\end{equation*}
whence $(y_i)\succ a(x)$ again, a contradiction.
\end{proof}

We obtain the following lemma by symmetry.

\begin{lemma}\label{l:lB} \label{l34}
Let $Q\in\A$ and $x\in B _Q$.

\begin{enumerate}[\upshape (i)]
\item There exists a  positive integer $n$ such that
\begin{equation*}
l(x)=m_1m_2\cdots m_{n-1}01^\infty\qtq{and}
m(x)=m_1m_2\cdots m_{n-1}1\mu_1\mu_2\cdots .
\end{equation*}

\item If $\mu(Q)=0^\infty$, or if the inequalities in \eqref{e:3102} are strict, then there is no expansion between $m(x)$ and $l(x)$.\label{371-2}

\item If $\mu(Q)\neq 0^\infty$, and equality holds in \eqref{e:3102} for a smallest positive integer $k$, then $k\ge 2$, $\mu_k=1$, and all expansions between $m(x)$ and $l(x)$ are given by the sequences
\begin{equation*}\label{371-3}
m_1\cdots m_{n-1}(1\mu_1\cdots \mu_{k-1})^N01^\infty, \quad N=1,2,\ldots,
\end{equation*}
with $k$ as in (i).
\end{enumerate}
\end{lemma}	

\begin{proof}[Proof of Proposition \ref{p111qqq}]
Combine Lemma \ref{l:lA} (i) and Lemma \ref{l:lB} (i).
\end{proof}	

\begin{lemma}\label{l:condition}\label{l35qqq}
Fix $Q\in\V\setminus \C$.
\begin{enumerate}[\upshape (i)]
\item If $\sigma^j(\mu(Q))=\alpha(Q)$ for some $j\ge 1$ and $x\in B_Q$, then $a(x)=m(x)\prec b(x)$ and $x\in A_Q$.\footnote{Cases (iii), (vi), (ix) of Lemma \ref{l:HuBarZou}.} \label{c2} 

\item If $\mu(Q)=\sigma^j(\alpha(Q))$ for some $j\ge 1$ and $x\in A_Q$, then $a(x)=m(x)\succ l(x)$ and $x\in B_Q$.\footnote{Cases (ii), (vii), (ix) of Lemma \ref{l:HuBarZou}.} \label{c3}

\item If $\sigma^i(\mu(Q))\prec\alpha(Q)$ for all $j\ge 1$ and $x\in B_Q\setminus A_Q$, then $m(x)=a(x)=b(x)$.\footnote{Cases (i), (ii), (iv), (v), (vii), (viii) of Lemma \ref{l:HuBarZou}.}\label{c4}
		
\item If $\mu(Q)\prec\sigma^i(\alpha(Q))$ for all $j\ge 1$ and $x\in A_Q\setminus B_Q$, then $m(x)=a(x)=l(x)$.\footnote{Cases (i), (iii), (iv), (v), (vi), (viii) of Lemma \ref{l:HuBarZou}.} \label{c5}
\end{enumerate}		
\end{lemma}

\begin{proof}
(i) By our assumption there exists a smallest positive integer $k$ such that 
\begin{equation}\label{e:317bis}
\mu(Q)=\mu_1\cdots\mu_k\alpha(Q).
\end{equation}
Furthermore, we must have
\begin{equation}\label{e:318bis}
\mu_k=0.
\end{equation}
For otherwise we would have $\mu_k=1$, and hence
\begin{equation*}
\alpha(Q)\succeq\sigma^{k-1}(\mu(Q))=1\alpha(Q),
\end{equation*}
implying  $\alpha(Q)=1^\infty$, contradicting our assumption $Q\notin\C$.
Here the inequality $\alpha(Q)\succeq\sigma^{k-1}(\mu(Q))$  follows from the minimality of $k$ if $k\ge 2$.
For $k=1$ it follows from the fact that $\alpha_1=1$ and $\mu_1=0$ for all $Q\in\A$; this follows from \cite[Theorem 1]{HuBarZou2024} for $Q\in\A\setminus\C$, and from Remark \ref{r14} (ii) for $Q\in\C$.

Since $x\in B_Q\subseteqq \V_Q$, by Lemmas \ref{l:quasi-gr-la} and \ref{l:lB} we have
\begin{equation}\label{e:319bis}
a(x)=m(x)=m_1m_2\cdots m_{n-1}1\mu(Q)\qtq{and} l(x)=m_1m_2\cdots m_{n-1}01^\infty,
\end{equation}
and \eqref{e:317bis}--\eqref{e:319bis} imply that $x\in A_Q$.
Finally, applying Lemma \ref{l:lA} (ii) we get
\begin{equation*}
b(x)=m_1m_2\cdots m_{n-1}1\mu_1\cdots\mu_{k-1}10^\infty,
\end{equation*}
so that  $a(x)=m(x)\prec b(x)$.  
\medskip 

(ii) follows from (i) by symmetry. 
\medskip

(iii)  Let $x\in B_Q\setminus A_Q$.
Then $m(x)=a(x)$ by Lemma \ref{l:quasi-gr-la}, and
Furthermore, 
\begin{equation*}
m(x)=m_1m_2\cdots m_{k-1}1\mu(Q)
\end{equation*}
for some $k\ge 1$ by Lemma \ref{l:lB} (i).

It remains to show that $b(x)=m(x)$, i.e., that $m(x)$ satisfies the lexicographic condition of Lemma \ref{l:monotonicity} (i).
Thanks to our assumption on $(\mu,\alpha)$ this is satisfied for every digit $m_j=0$ with $j>k$.
It remains to show that
\begin{equation}\label{e:320bis} 
m_{j+1}\cdots m_{k-1}1\mu(Q)\prec\alpha(Q) \text{ whenever }  1\leq j\le k-1 \text{ and } m_j=0.
\end{equation}

Since $a(x)=m(x)$, by Lemma \ref{l:monotonicity} (ii) we have
\begin{equation*}
m_{j+1}\cdots m_{k-1}1\preceq\alpha_1\cdots \alpha_{k-j}.
\end{equation*}
If this inequality is strict, then \eqref{e:320bis} obviously holds.
If this is an equality, and $\mu(Q)\prec \sigma^{k-j}(\alpha(Q))$, then \eqref{e:320bis} holds again.
Since $Q\in\V\setminus \C$ by our assumption, Lemma \ref{l:HuBarZou} implies the weak inequality $\mu(Q)\preceq \sigma^{k-j}(\alpha(Q))$, so that the only remaining case is where
\begin{equation*}
m_{j+1}\cdots m_{k-1}1\mu(Q)=\alpha(Q).
\end{equation*}
Then the inequality \eqref{e:320bis} fails, but this case is excluded by our assumption $x\notin A_Q$ because the properties
\begin{equation*}
a_j=m_j=0\text{ and }m_{j+1}\cdots m_{k-1}1\mu(Q)=\sigma^{j}(a(x))=\alpha(Q)
\end{equation*}
imply $x\in A_Q$ by definition.
\medskip

(iv)  follows from (iii) by symmetry. 
\end{proof}

In the following two lemmas we clarify the inclusion relations $A_Q$ and $B_Q$.

\begin{lemma}\label{l36qqq}
If $Q\in\C$, then $A_Q=B_Q=\V_Q\setminus\U_Q$, and every $x\in\V_Q\setminus\U_Q$ has exactly $2$ expansions.
\end{lemma}

\begin{proof}
We recall from Remark \ref{r14} that $\mu(Q)=0^\infty$ and $\alpha(Q)=1^\infty$. 
If $x\in A_Q$, then it follows from Lemma \ref{l:lA} (i)--(ii) that
\begin{equation*}
b(x)=a_1a_2\cdots a_{n-1}10^\infty\qtq{and}
a(x)=a_1a_2\cdots a_{n-1}01^\infty;
\end{equation*}
it is clear that there is no expansion between $b(x)$ and $a(x)$.
Since $a(x)=l(x)$ by Lemma \ref{l23}, $x$ has exactly two expansions.
Furthermore, $m(x)=b(x)$ by Lemma \ref{l23}, and hence $m(x)$ ends with $10^\infty=1\mu(Q)$, whence $x\in B_Q$.

Similarly, if $x\in B_Q$, then it follows from Lemma \ref{l:lB} (i)--(ii) that
\begin{equation*}
m(x)=m_1m_2\cdots m_{n-1}10^\infty 
\qtq{and} 
l(x)=m_1m_2\cdots m_{n-1}01^\infty;
\end{equation*}
it is clear that there is no expansion between $m(x)$ and $l(x)$. 
Since $m(x)=b(x)$ by Lemma \ref{l23}, $x$ has exactly two expansions.
Furthermore, $a(x)=l(x)$ by Lemma \ref{l23}, and hence $a(x)$ ends with $01^\infty=1\alpha(Q)$, whence $x\in B_Q$.

Finally, since $A_Q\subseteqq B_Q$ and $B_Q\subseteqq A_Q$, we conclude that
\begin{equation*}
A_Q=B_Q=A_Q\cup B_Q=\V_Q\setminus\U_Q.\qedhere
\end{equation*}
\end{proof}

In the following two results we refer again to the cases \eqref{21}--\eqref{29} of Lemma  \ref{l:HuBarZou}.

\begin{lemma}\label{l:A-B-set}\label{l37}\mbox{}
Let $Q\in\V\setminus\C$.

\begin{enumerate}[\upshape (i)]
\item If $Q$ satisfies one of the conditions \eqref{21}, \eqref{23}, \eqref{26} and \eqref{28}, then $A_Q\cap B_Q=\emptyset$.

\item If $Q$ satisfies one of the conditions \eqref{21}, \eqref{22}, \eqref{23}, \eqref{26}, \eqref{27} and \eqref{28}, then $1/(q_0(q_1-1))\in B_Q\setminus A_Q$.

\item If $Q$ satisfies one of the conditions \eqref{21}, \eqref{25},  \eqref{23}, \eqref{26}, \eqref{24} or \eqref{28}, then $1/q_1\in A_Q\setminus B_Q$.

\item If $Q$ satisfies the condition \eqref{29}, then $A_Q=B_Q$.
\end{enumerate}
\end{lemma}

\begin{proof}
(i) By the definitions of $A_Q$ and $B_Q$, if $x\in A_Q\cap B_Q$, then $a(x)$ ends with $0\alpha(Q)$ and $m(x)$ ends with $1\mu(Q)$.
Since $a(x)=m(x)$ by Lemma \ref{l:quasi-gr-la}, hence either 
$\sigma^{i}(\mu(Q))=\alpha(Q)$ for some $i\ge 1$, or 
$\mu(Q)=\sigma^{j}(\alpha(Q))$ for some $j\ge 1$.
But this is impossible because in the four cases of Lemma \ref{l:HuBarZou} considered here we have
\begin{equation*}
\sigma^{i}(\mu(Q))\prec\alpha(Q) \qtq{and} 
\mu(Q)\prec\sigma^{j}(\alpha(Q))\text{ for all } i, j\in \N.
\end{equation*}
\medskip

(ii) We already know  from Lemma \ref{l31} that $x:=1/(q_0(q_1-1))\in B_Q$ and $a(x)=m(x)=1\mu(Q)$. 
It remains to prove that $x\notin A_Q$.

Assume on the contrary that $x\in A_Q$, i.e., $a(x)=1\mu(Q)$ ends with $0\alpha(Q)$.
Then there exists an integer $i\ge 1$ such that $\sigma^i(\mu(Q))=\alpha(Q)$.
But this is impossible because in the six cases of Lemma \ref{l:HuBarZou} considered here we have
$\sigma^i(\mu(Q))\prec \alpha(Q)$ for all $i\ge 1$.
\medskip

(iii) Similarly to (ii), we already know  from Lemma \ref{l31} that $x:=1/q_1\in A_Q$ and $a(x)=m(x)=0\alpha(Q)$.  
It remains to prove that $x\notin B_Q$.

Assume on the contrary that $x\in B_Q$, i.e., $m(x)=0\alpha(Q)$ ends with $1\mu(Q)$.
Then there exists an integer $j\ge 1$ such that $\sigma^j(\alpha(Q))=\mu(Q)$.
But this is impossible because in the six cases of Lemma \ref{l:HuBarZou} considered here we have
$\mu(Q)\prec \sigma^j(\alpha(Q))$ for all $i\ge 1$.
\medskip

(iv) In this case the hypotheses of Lemma \ref{l:condition} (i) and (ii) are satisfied, so that $B_Q\subseteq A_Q$ and $A_Q\subseteq B_Q$.
\end{proof}

Now we determine for each $Q\in\V$ the number of expansions of every $x\in \V_Q\setminus\U_Q=A_Q\cup B_Q$.

\begin{proposition}\label{p38}\ 
Let $Q\in\V$.

\begin{enumerate}[\upshape (i)]
\item If $Q\in\U$,  then every $x\in\V_Q\setminus\U_Q$ has exactly $2$ expansions.

\item If $Q$ satisfies the condition \eqref{22}, then
\begin{enumerate}[\upshape (a)]
\item every $x\in A_Q$ has exactly $3$ expansions;
\item every $x\in B_Q\setminus A_Q$ has exactly $2$ expansions.
\end{enumerate}

\item If $Q$ satisfies the condition \eqref{25}, then
\begin{enumerate}[\upshape (a)]
\item every  $x\in B_Q$ has exactly $3$ expansions;
\item every  $x\in A_Q\setminus B_Q$ has exactly $2$ expansions.
\end{enumerate}

\item If $Q$ satisfies the condition \eqref{23}, then

\begin{enumerate}[\upshape (a)]
\item every $x\in A_Q$ has exactly $\aleph_0$ expansions;
\item every $x\in B_Q$ has exactly $2$ expansions.		
\end{enumerate}

\item If $Q$ satisfies the condition \eqref{26}, then
\begin{enumerate}[\upshape (a)]
\item every $x\in A_Q$ has exactly $2$ expansions;
\item every  $x\in B_Q$ has exactly $\aleph_0$ expansions.
\end{enumerate}

\item If $Q$ satisfies the condition \eqref{24}, then every $x\in \V_Q\setminus \U_Q$ has exactly $\aleph_0$ expansions. 

\item If $Q$ satisfies the condition \eqref{27}, then every $x\in \V_Q\setminus \U_Q$ has exactly $\aleph_0$ expansions. 

\item If $Q$ satisfies the condition \eqref{28}, then every $x\in \V_Q\setminus \U_Q$ has exactly $\aleph_0$ expansions. 

\item If $Q$ satisfies the condition \eqref{29}, 
i.e., if $Q\in\V \setminus\overline{\U }$, \label{key76}
then  every $x\in\V_Q\setminus\U_Q$ has exactly $\aleph_0$ expansions. 
\end{enumerate}
\end{proposition}

\begin{proof}
(i) For $Q\in\C$ this was proved in Lemma \ref{l36qqq}. 
Henceforth we assume that $Q\in\U \setminus \C$.

We know from Lemma \ref{l:A-B-set} (i) that $A_Q\cap B_Q=\emptyset$. 
If $x\in A_Q$,  then  $a(x)=m(x)=l(x)$ by Lemma \ref{l:condition} \eqref{c5}. 
Since $a(x)=m(x)=l(x)$ Lemma \ref{l:lA} (i)--(ii) implies that $x\in A_Q$ has exactly 2 expansions, namely $a(x)=m(x)=l(x)$ and $b(x)$. 
	
Similarly, 	if $x\in B_Q$, then Lemma \ref{l:condition} \eqref{c4} and Lemma \ref{l:lB}  (i)--(ii) imply that $x\in B_Q$ has exactly 2 expansions: $a(x)=m(x)=b(x)$ and $l(x)$. 
\medskip

(iia) If $x\in A_Q$, then it follows from  Lemma \ref{l:condition}  \eqref{c3} that $x\in B_Q$ and $ a(x)=m(x)\succ l(x)$. 
Now applying Lemmas \ref{l:lA} (i)--(ii) and  \ref{l:lB} we obtain that $x$ has exactly 3 expansions: $b(x)$, $a(x)=m(x)$ and $l(x)$. 
\medskip

(iib) If $x\in B_Q\setminus A_Q$, then  Lemma \ref{l:lB} (i) and  Lemma \ref{l:condition} \eqref{c4} imply that
\begin{equation*}
l(x)=m_1m_2\cdots m_{n-1}01^\infty\qtq{and}
b(x)=a(x)=m(x)=m_1m_2\cdots m_{n-1}1\mu(Q).
\end{equation*}
Applying Lemma \ref{l:lB} (ii) hence we conclude that every $x\in B_Q\setminus A_Q$ has exactly 2 expansions.
\medskip
	
(iii) This follows from (ii) by symmetry.
\medskip

(iva) If $x\in A_Q$, then it follows from Lemmas \ref{l:condition}  \eqref{c5} and  \ref{l:lA} (i), (iii) that  $a(x)=m(x)=l(x)\prec b(x)$, and there are exactly $\aleph_0$ expansions between $a(x)$ and $b(x)$.
This implies our result.
\medskip

(ivb) If $x\in B_Q$, then we infer from Lemmas \ref{l:condition} \eqref{c4} and   \ref{l:lB} (i)--(ii), we obtain that $a(x)=m(x)=b(x)\succ l(x)$, and there are no expansions between $m(x)$ and $l(x)$.
\medskip

(v) This follows from (iii) by symmetry.
\medskip

(vi) We have $B_Q\subseteq A_Q$ by Lemma \ref{l:condition} (i).
For each $x\in B_Q\subseteq A_Q$,  using Lemmas \ref{l:condition} \eqref{c2}, \ref{l:lB} (ii) and \ref{l:lA}  (iii) we obtain that $a(x)=m(x)\prec b(x)$, there are no expansion of $x$ between $m(x)$ and $l(x)$, and there are exactly $\aleph_0$ expansions between $a(x)$ and $b(x)$.
\medskip
		
If $x\in A_Q\setminus B_Q$, then by Lemmas \ref{l:lA} (i) and  \ref{l:condition} \eqref{c5} we know that
\begin{equation*}
b(x)=a_1a_2\cdots a_{k-1}10^\infty
\end{equation*}
and
\begin{equation*}
a(x)=m(x)=l(x)=a_1a_2\cdots a_{k-1}0\alpha(Q).
\end{equation*} 
Therefore, applying Lemma \ref{l:lA} (iii) again, we conclude that every $x\in A_Q\setminus B_Q$ has exactly $\aleph_0$ expansions.
\medskip

(vii) follows from (vi) by symmetry.
\medskip        
			
(viii) Since $\alpha(Q)$ is periodic, using Lemmas  \ref{l:lA} (i), (iii) and   \ref{l:condition} \eqref{c5}  we obtain that every $x\in A_Q$ has exactly $\aleph_0$ expansions.
\medskip

Similarly, since $\mu(Q)$ is also periodic, by using  Lemmas  \ref{l:lB} (i), (iii) and  \ref{l:condition} \eqref{c4} we obtain that every  $x\in B_Q$ has exactly $\aleph_0$ expansions.
\medskip

(ix) Applying Lemma \ref{l:condition} (i) and (ii), we have $A_Q=B_Q$, and  $l(x)\prec m(x)=a(x)\prec b(x)$ for every $x\in A_Q=B_Q$. 
Furthermore, Lemma  \ref{l:lA}  (iii) implies that there are $\aleph_0$ expansions between $a(x)$ and $b(x)$, and  Lemma \ref{l:lB}  (iii) implies that there are $\aleph_0$ expansions between $l(x)$ and $m(x)$. 
Hence our claim follows.
\end{proof}

We illustrate Proposition \ref{p38} with two examples in Examples \ref{e37}.

\begin{proof}[Proof of Theorem \ref{T:expansion}]
The theorem follows from Lemma \ref{l:HuBarZou} and Proposition \ref{p38}.
\end{proof}	

\section{Proof of Theorem \ref{T:topology1} (i)--(iv), except (ii) for $Q\in\A\setminus\V$, and Corollary \ref{c116qqq}}\label{s4}

First we prove some preparatory results.
Lemmas \ref{l:leftrightcontinu} and \ref{l:neighbor} are  generalizations of  \cite [Lemmas 2.2, 2.8 and 4.7]{DeVKomLor2022}.
	
\begin{lemma}\label{l:leftrightcontinu}\label{l41}
Let $x,  y_n\in J_Q$ for $n\in \mathbb{N}$. 
Then:
\begin{enumerate}[\upshape (i)]
\item If $y_n\searrow x$, then $b(y_n)\rightarrow b(x)$  and $m(y_n)\rightarrow m(x)$. \label{10}

\item	If $y_n\nearrow x$, then $l(y_n)\rightarrow l(x)$ and $a(y_n)\rightarrow a(x)$ . 
				
\item\label{key41} Let $(d_i)\neq 1^\infty$ be a greedy sequence. Then for every positive integer $N$, there exists a greedy sequence $(c_i)\succ(d_i)$ such that $$d_1\cdots d_N=c_1\cdots c_N.$$

\item\label{key411}  Let $(d_i)\neq 0^\infty$ be a lazy sequence. Then for every positive integer $N$, there exists a lazy sequence $(c_i)\prec(d_i)$ such that $$d_1\cdots d_N=c_1\cdots c_N.$$
\end{enumerate}
\end{lemma}
		
\begin{proof}
We only prove (i); (ii) can be proved similarly, and  (iii) and (iv) follow from (i) and (ii), respectively.
				
Write $b(y_n):=(b_i(y_n))$ and $b(x):=(b_i(x))$. 
We have to prove that for every positive integer $N$ there exists a number $n_N$ such that
\begin{equation*}
b_1(y_n)b_2(y_n)\cdots b_N(y_n)=b_1(x)b_2(x)\cdots b_N(x)
\qtq{and}
m_1(y_n)\cdots m_N(y_n)=m_1(x)\cdots m_N(x)
\end{equation*}
for all $n\ge n_N$.

First we consider the greedy expansions.
We proceed by induction on $N$.
Let $N\ge 1$, and assume that there exists a number $n_{N-1}$ such that 
\begin{equation*}
b_1(y_n)b_2(y_n)\cdots b_{N-1}(y_n)=b_1(x)b_2(x)\cdots b_{N-1}(x)
\end{equation*}
for all $n\ge n_{N-1}$;
for $N=1$ we may simply take $n_0=1$.
In the rest of the proof we consider only indices $n\ge n_{N-1}$.

If $b_N(x)=1$, then
\begin{equation*}
\sum_{i=1}^{N-1}\frac{b_i(x)}{q_{b_1(x)}\cdots q_{b_{i}(x)}}+\frac{1}{q_{b_1(x)}\cdots q_{b_{N-1}(x)}q_1}\le x
\end{equation*}
by definition (see Section 2).
Since $y_n\ge x$ for every $n\ge 1$, this inequality remains valid if we change $x$ to $y_n$.
Using the definition again, it follows that $b_N(y_n)=1=b_N(x)$ for all $n\ge 1$.

If $b_N(x)=0$, then
\begin{equation*}
\sum_{i=1}^{N-1}\frac{b_i(x)}{q_{b_1(x)}\cdots q_{b_{i}(x)}}+\frac{1}{q_{b_1(x)}\cdots q_{b_{N-1}(x)}q_1}>x
\end{equation*}
by definition.
Thanks to the induction hypothesis and the assumption $y_n\to x$, there exists a number $n_N\ge n_{N-1}$ such that this inequality remains valid if we change $x$ to $y_n$ for any $n\ge n_N$.
Using the definition again, it follows that $b_N(y_n)=0=b_N(x)$ for all $n\ge n_N$.

The proof for the quasi-lazy expansions is analogous, we only have to replace the above inequalities to
\begin{equation*}
\sum_{i=1}^{N-1}\frac{m_i(x)}{q_{m_1(x)}\cdots q_{m_{i}(x)}} +\frac{1}{q_{m_1(x)}\cdots q_{m_{N-1}(x)}q_0(q_1-1)} \le  x
\end{equation*}				
if $m_N(x)=1$, and to 
\begin{equation*}
\sum_{i=1}^{N-1}\frac{m_i(x)}{q_{m_1(x)}\cdots q_{m_{i}(x)}} +\frac{1}{q_{m_1(x)}\cdots q_{m_{N-1}(x)}q_0(q_1-1)}> x
\end{equation*}	
if $m_N(x)=0$, respectively.
\end{proof}
		
The following Lemma directly follows from Lemma \ref{l:monotonicity}:
	
\begin{lemma}\label{l:trunca}\label{l42}\quad	
\begin{enumerate}[\upshape (i)]
\item  If $(d_i)=d_1d_2\cdots$ is a greedy or quasi-greedy sequence, then the sequence $d_1\cdots d_k0^\infty$ is greedy for every $k \geq1$.\label{key45-1}
			
\item If $(d_i)=d_1d_2\cdots $ is a lazy or quasi-lazy sequence, then the sequence $d_1\cdots d_k1^\infty$ is lazy for every $k \geq1$.
\end{enumerate}	
\end{lemma}

\begin{lemma}\label{l:neighbor}\label{l43}
Let $Q\in\A\setminus\C$, $x\in J_Q$, and consider the greedy and lazy expansions $(b_i)$ and $(l_i)$ of $x$.

\begin{enumerate}[\upshape (i)]
\item  Assume that
\begin{equation*}
b_n=1 \qtq{and} b_{n+1}b_{n+2}\cdots\prec\mu(Q)
\qtq{for some}n\geq 1.
\end{equation*} 

\begin{enumerate}[\upshape (a)]
\item There exists a number $z>x$ such that $[x,z]\cap \U_Q=\emptyset$ and $(x,z]\cap \V_Q=\emptyset$.

\item If $b_j=1$ for  some $j>n$, there exists a number $y<x$ such that $[y,x]\cap \U_Q=\emptyset$.
\end{enumerate}

\item Assume that
\begin{equation*}
l_n=0 \qtq{and} l_{n+1}l_{n+2}\cdots\succ\alpha(Q)
\qtq{for some}n\geq 1.
\end{equation*}		

\begin{enumerate}[\upshape (a)]
\item There exists a number $z<x$ such that $[z,x]\cap \U_Q=\emptyset$ and $[z,x)\cap \V_Q=\emptyset$.

\item If  $l_j=0$ for some $j>n$, there exists a number $y>x$ such that $[x,y]\cap \U_Q=\emptyset$.
\end{enumerate}
\end{enumerate}
\end{lemma}

\begin{proof}
We only prove (ii), the proof of (i) is similar. 
\medskip

(a) By our assumption there exists a positive integer $N>n+1$ such that
\begin{equation*}
l_{n+1}\cdots l_{N}\succ \alpha_1\cdots\alpha_{N-n}.
\end{equation*}

By Lemma \ref{l:leftrightcontinu} \eqref{key411} we may choose a lazy sequence $(c_i)\prec(l_i)$ satisfying
\begin{equation*}
c_1\cdots c_N=l_1\cdots l_N.
\end{equation*}
Take $z=\pi_Q((c_i))$, then $(c_i)$ is the lazy expansion of $z$ and $z<x$. 
If $(d_i)$ is the lazy expansion of a number $v\in [z,x]$, then $(d_i)$  begins with $l_1\cdots l_N$ by the monotonicity part for  lazy  expansions in Lemma \ref{l:monotonicity}. 
We have thus
\begin{equation}\label{41}
d_n=0 \qtq{and} d_{n+1}d_{n+2}\cdots\succ\alpha(Q),
\end{equation}		
and hence $v\notin \U_Q$ by the definition of $\U_Q$.

We claim that \eqref{41} also holds if $(d_i)$ is the quasi-lazy expansion of a number $v\in [z,x)$.
This follows from the preceding paragraph if $m(v)=l(v)$.
Otherwise choose a number $v<t<x$ such that $m(t)=l(t)$.
This is possible because by Remark \ref{r11} there are only countable many numbers $t$ such that $m(t)\ne l(t)$, and the interval $(v,x)$ is uncountable.
Then we have
\begin{equation*}
l(z)\preceq m(z)\preceq m(v)\prec m(t)=l(t)\prec l(x),
\end{equation*}
and we conclude by recalling that both $l(z)$ and $l(x)$ start with $l_1\cdots l_N$.
Using the definition of $\V_Q$, we infer from \eqref{41} that
$v\notin \V_Q$.
\medskip
			
(b) If $j>n$ and $l_j(x)=0$, then $(c_i)=l_1(x)\cdots l_n(x)1^{\infty}$ is  the lazy expansion of some $y>x$ by Lemma \ref{l:trunca} (ii). 
If $(d_i)$ is the lazy expansion of a number $w\in [x,y]$, then $(d_i)$ also begins with $l_1\cdots l_n$ and hence 	
\begin{equation*}
d_{n+1}d_{n+2}\cdots\succeq l_{n+1}(x)l_{n+2}(x)\cdots\succ\alpha(Q).
\end{equation*}
The first inequality follows again from the monotonicity part of Lemma \ref{l:monotonicity} (iii). 
Therefore the relations \eqref{41} holds again, and therefore $w\notin\U_Q$.
\end{proof}

\begin{lemma}\label{c:rightneighbor}\label{l44}
Fix $Q\in\A\setminus\C$, then for each $x\in J_Q\setminus\V_Q$ there exists two numbers $y<x$ and $z>x$   such that $[y,z]\cap \V_Q=\emptyset$.
\end{lemma}

\begin{proof}
Let $x\in J_Q\setminus\V_Q$.
By the definition of $\V_Q$, we have either
\begin{equation}\label{42}
a_j(x)=1\qtq{and}\sigma^j(a(x))\prec \mu(Q),
\end{equation} 
for some $j\ge 1$, or
\begin{equation*}
m_i(x)=0\qtq{and}\sigma^i(m(x))\succ\alpha(Q)
\end{equation*} 
for some $i\ge 1$. 
By symmetry we only consider the first case.

First we observe that the condition of Lemma \ref{l:neighbor} (i) is satisfied, and hence there exists a $x>x$ such that $[x,z]\cap \V_Q=\emptyset$.

Indeed this condition coincides with \eqref{42} if $a(x)=b(x)$.
Otherwise $b(x)$ is finite, and if $b_n=1$ is its last nonzero digit, then
\begin{equation*}
b_{n+1}b_{n+2}\cdots=0^\infty\prec\mu(Q);
\end{equation*}
the last inequality follows from our assumption that
$Q\in\A\setminus\C$.
	
It remains to find a $y<x$ such that $[y,x]\cap \V_Q=\emptyset$.
By \eqref{42} that there exists an integer $k>j$ such that 
\begin{equation*}
a_{j+1}\cdots a_k\prec \mu_1\cdots\mu_{k-j}.
\end{equation*}
Applying  Lemma \ref{l:leftrightcontinu} (i), there exists a number $y<x$ such that $a(t)=(c_i)$ starts with $a_1\cdots a_k$ for every $t\in [y,x]$.
Then 
\begin{equation*}
c_j=0\qtq{and}
c_{j+1}\cdots c_k\prec \mu_1\cdots\mu_{k-j},
\end{equation*}
whence $t\notin\V_Q$.
\end{proof}

\begin{lemma}\label{l:close set}
Fix $Q\in\A$.
\begin{enumerate}[\upshape (i)]
\item The set $\V_Q$ is closed.\label{l:close set-2}\label{l45}

\item ${\overline\U}_Q\subseteqq \V_Q.$\label{l:close set-3}
\end{enumerate}
\end{lemma}
	
\begin{proof}
The case $Q\in\C$ has already been proved in Lemma \ref{l23} (iii).
Henceforth we assume that $Q\in\A\setminus\C$.
\medskip

(i) We prove that the complement  of $\V_Q$ is open. 
Given any $x\in J_Q\setminus\V_Q$, writing $a(x)=(a_i)$ and $m(x)=(m_i)$ for brevity, there exists an integer $n\ge 1$ such that either
\begin{equation*}
a_n=1, \qtq{and} \sigma^n(a(x))\prec\mu(Q),
\end{equation*}
or 
\begin{equation*}
m_n=0, \qtq{and} \sigma^n(m(x))\succ\alpha(Q).
\end{equation*}
By symmetry we consider the first possibility. 

Choose a sufficiently large $\ell$  such that
\begin{equation}\label{e:450}
a_{n+1}\cdots a_{n+\ell}\prec\mu_1(Q)\cdots \mu_\ell(Q).
\end{equation}
By Lemma \ref{c:rightneighbor} there exists a $z>x$ such that $[x, z]\cap\V_Q=\emptyset$. 
We consider the left neighborhood $(y,x]$ of $x$ with
\begin{equation*}
y:=\pi_Q(a_1\cdots a_{n+\ell}0^\infty)<x.
\end{equation*}
Then $a_1\cdots a_{n+\ell}0^\infty$ is the greedy expansion of $y$ by Lemma \ref{l:trunca}, and the quasi-greedy expansion of every number $p\in(y,x]$ starts with the block $a_1\cdots a_{n+\ell}$, and therefore $p\notin \V_Q$ by \eqref{e:450}.
It follows from these relations that $y<x<z$ and  that $(y,z)\cap\V_Q=\emptyset$.
\medskip

(ii) Since $\U_Q\subseteqq\V_Q$ by definition, this is a consequence of (i).
\end{proof}

For the next lemma we recall that a set $A\subseteq(1,M+1]$ is \emph{closed from above} (respectively\emph{from below}) if the limit of every decreasing (respectively increasing) sequence of elements in $A$ belongs to $A$.

\begin{lemma}\label{l:closed} \label{l46}
Fix $Q\in \A\setminus \C$.
\begin{enumerate}[\upshape (i)]
\item If $\V_Q\setminus\U_Q=A_Q$, then $\U_Q$ is closed from above. 

\item If $\V_Q\setminus\U_Q=B_Q$, then $\U_Q$ is closed from below. 
	
\item  If $Q\in \V \setminus\overline{\U }$, then $\U_Q$ is closed.
\end{enumerate}
\end{lemma}

\begin{proof}
(i) It suffices to prove for each $x \in J_Q\setminus\U_Q$, there exists a $z>x$ such that $[x, z)\cap U_Q=\emptyset$.
In case $x\in J_Q\setminus\V_Q$ this follows from Lemma \ref{c:rightneighbor}. 

Otherwise we have $x\in \V_Q\setminus\U_Q=A_Q$.
Then $x$ has a a finite greedy expansion, and then it satisfies the condition of Lemma \ref{l:neighbor} (i). 
Hence we obtain that there exists $z>x$ such that $[x,z]\cap \U_Q=\emptyset$. 
\medskip	

(ii) The proof is similar to (i), now using Lemmas \ref{c:rightneighbor} and \ref{l:neighbor} (ii).
\medskip	

(iii) This follows from (i), (ii) because $\V_Q\setminus\U_Q=A_Q=B_Q$
by Lemmas \ref{l:HuBarZou} and \ref{l:condition} \eqref{c2}, \eqref{c3}.
\end{proof}
	 
For the proof of Theorem \ref{T:topology1} (ii) we need the following two lemmas:

\begin{lemma}\cite[Theorem 2.1]{DeVKomLor2016}\label{l:important1}\label{l47}
Let $Q\in \A$.
\begin{enumerate}[\upshape (i)]
\item Assume that
\begin{equation*}
\sigma^j((x_i))\prec \alpha(Q)\text{ whenever }x_j=0.
\end{equation*} 
Then there exists a sequence  $1<k_1<k_2<\cdots$ of positive integers such that for each $i\geq 1$,
\begin{equation*}
x_{k_i}=0, \qtq{and} x_{n+1}\cdots x_{k_i}\prec\alpha_1\cdots \alpha_{k_{i}-n} \qtq{if} 1\leq n<k_i \qtq{and} x_n=0.
\end{equation*}

\item Assume that
\begin{equation*}
\sigma^{j}((x_i))\succ\mu(Q)\text{ whenever }x_j=1.
\end{equation*} 
Then there exists a sequence $1<\ell_1<\ell_2<\cdots$ of positive integers such that for each $i\geq 1$,
\begin{equation*}
x_{\ell_i}=1, \qtq{and} x_{n+1}\cdots x_{\ell_i}\succ\mu_1\cdots \mu_{\ell_{i}-n} \qtq{if} 1\leq n<\ell_i \qtq{and} x_n=1.
\end{equation*}
\end{enumerate}
\end{lemma}

\begin{remark}\label{r:important1}\label{r48}
Only Part (i) of Lemma \ref{l:important1} was proved in \cite{DeVKomLor2016}, but Part (ii) hence follows by symmetry.
\end{remark}

\begin{lemma}\label{l:density1}\label{l49}
Let $Q\in \V$.
\begin{enumerate}[\upshape (i)]
\item For each $x\in \U_Q$ there exists a sequence $(x^k)$  in $A_Q$ such that $b(x^k)\to b(x)$ and $x^k\to x$.
Moreover, $(x^k)$ may be chosen to be increasing if $x\in \U_Q\setminus \{0\}$, and decreasing if $x=0$.
			
\item  For each $x\in \U_Q$ there exists a sequence $(x^k)$ in $B_Q$ such that $l(x^k)\to l(x)$ and $x^k\to x$. 
Moreover, $(x^k)$ may be chosen to be decreasing if $x\in \U_Q\setminus \{1/(q_1-1)\}$, and increasing if $x=1/(q_1-1)$.
\end{enumerate}
\end{lemma}

\begin{proof}
The idea of the following proof originates from \cite[Lemma 5.1 ]{{DeVKomLor2022}}.
\medskip

(i) If $x=0$, then we may choose the quasi-greedy sequences $(x_i^k):=0^k\alpha(Q)$ for $k=1, 2,\ldots.$ 
It is clear that $(x_i^k)\searrow 0^\infty$, and hence  $x^k:=\pi_Q((x_i^k))\rightarrow 0$ as $k\rightarrow\infty$. 
We have seen in Lemma \ref{l31} that $x^1\in\V_Q$; a simple adaptation of the proof of Lemma \ref{l31} shows that $x^k\in\V_Q$ for every $k$.
Finally, $x^k\in A_Q$ because its greedy expansion  $b(x^k)=0^{k-1}10^\infty$ is finite.
		 
Now let $x\in \U_Q\setminus \{0\}$, and let $(x_i)$ denote its unique expansion.
We recall from Lemma \ref{l:important1} (ii) that there exists a sequence $1<\ell_1<\ell_2<\cdots$ of positive integers such that for each $i\geq 1$,
\begin{equation}\label{e:460}
x_{\ell_i}=1, \qtq{and} x_{n+1}\cdots x_{\ell_i}\succ\mu_1\cdots \mu_{\ell_{i}-n} \qtq{if} 1\leq n<\ell_i \qtq{and} x_n=1.
\end{equation}
		
Now consider for each $k\ge 1$ the finite greedy sequence
\begin{equation*}
(b^k_j):=x_1\cdots x_{\ell_{k}}0^{\infty},
\end{equation*}		
and set $x^k:=\pi_Q((b^{k}_j))$.
It is clear that ${(b^k_j)}\nearrow (x_i)$ and $x^k\nearrow x$ as $k\rightarrow \infty$.
It remains to prove that $x^k\in \V_Q\setminus\U_Q$ for each $k\geq 1$.	

Since the quasi-greedy expansion
\begin{equation*}
a(x^k)=x_1\cdots x^{-}_{\ell_{k}}\alpha(Q)
\end{equation*}
of $x^k$ is different from its greedy expansion, $x^k\notin \U_Q$.

If $Q\in \C$, then $x^k\in J_Q=\V_Q$ by Lemma \ref{l23} (iii).
If $Q\in \V\setminus\C$, then the relation $x^k\in \V_Q$ will follow by Lemma \ref{l:quasi-gr-la} if we show that $m(x^k)=a(x^k)$.
Since $a(x^k)$ is doubly infinite, and hence co-infinite by Lemma \ref{l:doublyinfinite}, it is sufficient to show that
\begin{equation*}
x_{i+1}\cdots x^{-}_{\ell_{k}}\alpha(Q)\succeq \mu(Q) \qtq{whenever} x_{i}=1 \text{ and } 1\le i\le \ell_k-1.
\end{equation*}	
This is true because 
\begin{equation*}
x_{i+1}\cdots x^{-}_{\ell_{k}}\succeq \mu_1\cdots \mu_{\ell_k-i}
\end{equation*}
by \eqref{e:460}, and $\alpha(Q)\succeq \sigma^{\ell_k-i}(\mu(Q)$ by our assumption $Q\in \V$ (see the different cases of Lemma \ref{l:HuBarZou}).
\medskip

(ii)  If $x=1/(q_1-1)$, then we choose the quasi-lazy sequences $(x^k_i):=1^k\mu(Q)$ for $k=1, 2,\ldots.$ 
It is clear that $(x^k_i)\nearrow 1^\infty$, and hence  $x^k:=\pi_Q((x_i^k))\rightarrow 1/(q_1-1)$ as $k\rightarrow\infty$. 
Furthermore, $x^k\in B_Q$ because $l(x^k)=1^{k-1}01^\infty$. 
The rest of the proof is similar to (i).
\end{proof}
	
\begin{proof}[Proof of Theorem \ref{T:topology1} (i), (iii), (iv), and (ii) for $Q\in\V$]\ 

(i) It was proved in Lemma \ref{l:close set}.
\medskip 

(ii) If $Q\in\V$, then the relation $|\V_Q\setminus\U_Q|=\aleph_0$ follows from Remark \ref{r16} and Lemma \ref{l31}, and the density of $\V_Q\setminus \U_Q$ in $\V_Q$ follows from Lemma \ref{l:density1}.
\medskip

(iii) Since $Q\in \V \setminus\overline{\U }$ by assumptions, $\U_Q$ is closed by Lemma \ref{l:closed} (iii).
Next we show that each $x\in \V_Q\setminus{\U}_Q$ is isolated in $\V_Q$.

It follows from Lemma \ref{l:condition} \eqref{c2}, \eqref{c3} that $x$ has a finite greedy expansion and a co-finite lazy expansion.
Therefore by Lemma \ref{l:neighbor} (i), (ii) there exist two numbers $z>x$ and $z<x$ such that $(x,z]\cap \V_Q=\emptyset$ and $[y,x)\cap \V_Q=\emptyset$.

Since $\V_Q\setminus\U_Q\ne\emptyset$ by Lemmas \ref{l31}, it
has isolated points, and therefore $\V_Q$ is not a Cantor set.
\medskip

(iv)  This is proved in Lemma \ref{l23} (iii).
\end{proof}	

\begin{proof}[Proof of Corollary \ref{c116qqq}]
The first two equivalences follow from the definitions of $\U_Q$ and $\V_Q$. 

To prove the third relation, we assume that $Q\notin \overline{\U}$. 
We have to prove that at least one of the numbers $\ell_Q$ and $\mu(Q)$ is outside $\overline{\U}_Q$.

If $Q\in\A\setminus\V$, then $a(r_Q)=\alpha(Q)$ and $m(\ell_Q)=\mu(Q)$ satisfy  one of the conditions \eqref{30}, \eqref{31}, \eqref{32} of Lemma \ref{l:HuBarZou}. 
By the definition of $\V_Q$ this implies that at least one of the numbers $\ell_Q$ and $\mu(Q)$ is even outside $\V_Q\supseteqq\overline{\U}_Q$.

If $Q\in\V \setminus \overline{\U}$, then $a(r_Q)=\alpha(Q)$ and $m(\ell_Q)=\mu(Q)$ satisfy the condition \eqref{29} of  Lemma \ref{l:HuBarZou}.
By the definition of $\U_Q$ this implies that none of the numbers $\ell_Q$ and $r_Q$ belongs to $\U_Q$.
We complete the proof by recalling that under the condition \eqref{29} we have $\overline{\U}_Q=\U_Q$ by  Theorem  \ref{T:topology1} (iv).
\end{proof}

\section{Proof of Proof of Theorems \ref{T:topology1} (v)--(vi) and \ref{T:topology2}} \label{s5}     

The following lemma plays a crucial role in this section.
Let $x=\pi_Q((x_i))\in J_Q$, we recall that a real number $x\in \U_Q$ if and only if the following two conditions are satisfied:
\begin{equation}\label{51}
\begin{split}
&\sigma^j((x_i))\prec\alpha(Q) \qtq{ whenever }x_j=0,\\
&\sigma^j((x_i))\succ\mu(Q) \qtq{ whenever }x_j=1.
\end{split}
\end{equation}
 
\begin{lemma}\label{l:density3}\label{l51}
Fix $Q\in\A$.
If $\sigma^i(\mu(Q))\prec\alpha(Q)$ and $\mu(Q)\prec\sigma^j(\alpha(Q))$ for all $i,j\ge 1$.\footnote{This assumption is satisfied in cases  \eqref{21}, \eqref{23}, \eqref{26}, \eqref{28} of Lemma \ref{l:HuBarZou}.}
Then:	
\begin{enumerate}[\upshape (i)]
\item  For each $x\in A_Q$, there exists a sequence $(a^\ell_j)$ such that $x^{\ell}=\pi_Q((a^{\ell}_k))\in \U_Q$ and $x^\ell\nearrow x$ as $\ell\to\infty$.

\item For each $x\in B_Q$, there exists a sequence $(b^\ell_k)$ such that $x^{\ell}=\pi_Q((b^{\ell}_k))\in \U_Q$ and $x^\ell\searrow x$ as $\ell\to\infty$.
\end{enumerate} 
\end{lemma}

\begin{proof}
It follows from assumptions and from Lemma \ref{l:HuBarZou} that $Q\in\V$.

If $Q\in\C$, then $\V_Q\setminus\U_Q$ is a countable set in the interior of $J_Q$ by Lemma \ref{l23} (iii); in particular, it does not contain any non-degenerate interval.
It follows that if  $x\in \V_Q\setminus\U_Q$, then every left and every right neighborhood of $x$ meets its complementer set in $J_Q$, i.e., the set $\U_Q$.
This implies the existence of the required sequences $(a^\ell_j)$ and $(b^\ell_j)$.

Henceforth we assume that $Q\in\V\setminus\C$.
\medskip

(i) As usual, we write sometimes $\mu=(\mu_i):=\mu(Q)$ and $\alpha=(\alpha_i):=\alpha(Q)$ for brevity.

Let $x\in A_Q$. 
From Lemma \ref{l:lA} (i) we have
\begin{equation*}
b(x)=a_1\cdots a_{n-1}10^\infty \qtq{and}a(x)=a_1\cdots a_{n-1}0\alpha(Q).
\end{equation*} 
We are going to construct for each $\ell\in\N$ a sequence $(a^{\ell}_i)\prec a(x)$,  starting with 
\begin{equation*}
a_1\cdots a_{n-1}0\alpha_1\cdots\alpha_{\ell},
\end{equation*}
and satisfying the conditions \eqref{51} with $(x_i):=(a^{\ell}_i)$.
Then we will have 
\begin{equation*}
(a^{\ell}_i)\to a(x),\quad
x^{\ell}:=\pi_Q((a^{\ell}_i))\to x,\qtq{and}
x^{\ell}\in\U_Q\qtq{for all}\ell,
\end{equation*} 
Furthermore, since $(a^{\ell}_i)\prec a(x)$ for every $\ell$, taking a subsequence if needed, the sequences $(a^1_i), (a^1_i), \ldots$ and $(x^{\ell})$ will be increasing, too.

We turn to the construction.
We fix an arbitrary $\ell\in\N$, and henceforth we do not indicate the dependence on $\ell$.

\emph{First step.}
Applying Lemma \ref{l:important1} (ii) with $(x_i):=(\alpha_i)=\alpha(Q)$ we choose an integer $m_1\ge\ell$ such that
\begin{equation*}\label{step1}
\alpha_{m_1}=1,\text{ and }
\alpha_{k+1}\cdots\alpha_{m_1}\succ\mu_1\cdots\mu_{m_1-k}
\text{ whenever }1\le k<m_1\text{ and }\alpha_k=1.
\end{equation*}

If $1\le k<m_1$ and $\alpha_k=0$, then
\begin{equation*}
\alpha_{k+1}\cdots\alpha_{m_1}\mu
\preceq \alpha_1\cdots\alpha_{m_1-k}\mu
\prec\alpha
\end{equation*}
because $\mu\prec \sigma^{m_1-k}(\alpha)$ by our assumption.

\emph{Second step.}
Since there are only finitely many such positive integers $k<m_1$, if $m_2\ge m_1$ is a sufficiently large integer, then we have
\begin{equation}\label{step2}
\alpha_{k+1}\cdots\alpha_{m_1}\mu_1\cdots\mu_{m_2}
\prec\alpha_1\cdots\alpha_{m_1+m_2-k}\text{ whenever }1\le k<m_1\text{ and }\alpha_k=0.
\end{equation}
Furthermore, since $\sigma^{m_1}(\alpha)\succ\mu$ by our assumption, by choosing a larger $m_2$ if necessary, we may also assume that
\begin{equation*}
\mu_1\cdots\mu_{m_2}\prec\alpha_{m_1+1}\cdots\alpha_{m_1+m_2}
\end{equation*}
Finally, applying Lemma \ref{l:important1} (i) with $(x_i):=\mu$, we choose an integer $m_2\ge m_1$ such that \eqref{step2} and the following condition are satisfied:
\begin{equation*}\label{step3}
\mu_{m_2}=0,\text{ and }
\mu_{k+1}\cdots\mu_{m_2}\prec\alpha_1\cdots\alpha_{m_2-k}
\text{ whenever }1\le k<m_2\text{ and }\mu_k=0.
\end{equation*}

\emph{Third step.}
If $1\le k<m_2$ and $\mu_k=1$, then
\begin{equation*}
\mu_{k+1}\cdots\mu_{m_2}\alpha
\succeq \mu_1\cdots\mu_{m_1-k}\alpha
\succ\mu
\end{equation*}
because $\alpha\succ \sigma^{m_2-k}(\mu)$ by our assumption.
Since there are only finitely many such $k$s, if $m_3\ge m_2$ is a sufficiently large integer, then we have
\begin{equation}\label{step4}
\mu_{k+1}\cdots\mu_{m_2}\alpha_1\cdots\alpha_{m_3}
\succ\mu_1\cdots\mu_{m_2+m_3-k}\text{ whenever }1\le k<m_2\text{ and }\mu_k=1.
\end{equation}
Applying Lemma \ref{l:important1} (ii) with $(x_i):=\alpha$, we choose an integer {\color{blue} $m_3\ge m_2$ } such that \eqref{step4} and the following condition are satisfied:
\begin{equation*}\label{step5}
\alpha_{m_3}=1,\text{ and }
\alpha_{k+1}\cdots\alpha_{m_3}\succ\mu_1\cdots\mu_{m_3-k}
\text{ whenever }1\le k<m_3\text{ and }\alpha_k=1.
\end{equation*}

Continuing by induction, we obtain a sequence
\begin{equation*}
\alpha_1\cdots\alpha_{m_1}\mu_1\cdots\mu_{m_2}\alpha_1\cdots\alpha_{m_3}\mu_1\cdots\mu_{m_4}\cdots
\end{equation*} 
satisfying the conditions \eqref{51}.

We claim that the sequences 
\begin{equation*}
(a^{\ell}_i):=a_1\cdots a_{n-1}0\alpha_1\cdots\alpha_{m_1}\mu_1\cdots\mu_{m_2}\alpha_1\cdots\alpha_{m_3}\mu_1\cdots\mu_{m_4}\cdots,\quad \ell=1,2,\ldots
\end{equation*}
have the required properties. 
The inequality $(a^{\ell}_i)\prec a(x)$ follows from \eqref{step2b} because $(a^{\ell}_i)$ and $a(x)$ start with 
\begin{equation*}
a_1\cdots a_{n-1}0\alpha_1\cdots\alpha_{m_1}\mu_1\cdots\mu_{m_2}\qtq{and}
a_1\cdots a_{n-1}0\alpha_1\cdots\alpha_{m_1}\alpha_{m_1+1}\cdots\alpha_{m_1+m_2},
\end{equation*}
respectively.

It remains to check the conditions \eqref{51}.
We have already seen that they are satisfied for $j>n$.
They are also satisfied for $j=n$ by \eqref{step2b}, because the $n$th digit of $(a^{\ell}_i)$ is equal to zero, and $\sigma^n(a^{\ell}_i)$ and $\alpha$ start with
\begin{equation}\label{step2b}
\alpha_1\cdots\alpha_{m_1}\mu_1\cdots\mu_{m_2}\qtq{and}\alpha_1\cdots\alpha_{m_1}\alpha_{m_1+1}\cdots\alpha_{m_1+m_2},
\end{equation}
respectively.

If $1\le j<n$ and $a^{\ell}_j=0$, then \eqref{51} holds because
\begin{equation*}
a^{\ell}_{j+1}\cdots a^{\ell}_n
=b_{j+1}\cdots b_n^-
\prec b_{j+1}\cdots b_n\preceq\alpha_1\cdots\alpha_{n-j};
\end{equation*}
the last inequality follows from the lexicographic characterization of greedy expansions.

Finally we consider the case where $1\le j<n$ and $a^{\ell}_j=1$.
We have to show that
\begin{equation}\label{52}
a_{j+1}\cdots a_{n-1}0\alpha_1\alpha_2\cdots
\succ\mu_1\mu_2\cdots .
\end{equation}
Since we consider now the case $Q\in\V\setminus\C$, $m(x)=a(x)$ by  Lemma \ref{l:quasi-gr-la}.
Hence the quasi-lazy expansion $m(x)$ starts with $a_1\cdots a_{n-1}0$, and therefore 
\begin{equation}\label{53}
a_{j+1}\cdots a_{n-1}0\succeq\mu_1\cdots\mu_{n-j}
\end{equation}
by Lemma \ref{l:monotonicity} (iv).
Furthermore, 
\begin{equation}\label{54}
\alpha_1\alpha_2\cdots\succ \mu_{n-j+1}\mu_{n-j+2}\cdots.
\end{equation}
by our assumption $\sigma^{n-j}(\mu)\prec\alpha$, and \eqref{52} follows from \eqref{53} and \eqref{54}.
\medskip

(ii) The proof is analogous to the proof of (i).
\end{proof}

The two results of following lemmas discuss the situations when $\sigma^i(\alpha(Q)=\mu(Q)$, $\sigma^j(\mu(Q)=\alpha(Q)$ for some $i,j\geq1.$

\begin{lemma}\label{l:isolated}\label{l52}\
Let $Q\in\A$.

\begin{enumerate}[\upshape (i)]
\item Every $x\in A_Q$ is isolated in $\V_Q$ from the right.
Furthermore, if $\mu(Q)=\sigma^j(\alpha(Q)$ for some $j\ge 1$,\footnote{This assumption is satisfied in cases   \eqref{22}, \eqref{27}, \eqref{29} of Lemma \ref{l:HuBarZou}.} then
$x$ is also isolated from the left. 
\color{black}

\item Every $x\in B_Q$ is isolated in $\V_Q$ from the left.
Furthermore, if $\sigma^j(\mu(Q)=\alpha(Q)$ for some $j\ge 1$,\footnote{This assumption is satisfied in cases   \eqref{25}, \eqref{24}, \eqref{29} of Lemma \ref{l:HuBarZou}.} then 
$x$ is also isolated from the right.
\end{enumerate}
\end{lemma}

\begin{proof}
It follows from  Lemma \ref{l:HuBarZou} and our assumptions that $Q\in\A\setminus\C$, so that $\mu\succ 0^{\infty}$ and $\alpha\prec 1^{\infty}$.
\medskip

(i) If $x\in A_Q$, then 
\begin{equation*}
b(x)=a_1\cdots a_{n-1}10^\infty \qtq{and} a(x)=a_1\cdots a_{n-1}0\alpha
\end{equation*}
by Lemma \ref{l:lA} (i) for some $n\ge 1$.

Since $\mu\succ 0^{\infty}$, $\mu$ starts with $0^k1$ for some positive integer $k$.
If $y>x$ is sufficiently close to $x$, then $a(y)$ starts with $a_1\cdots a_{n-1}10^{k+1}$; then $a_n(y)=1$ and $\sigma^n(a(y))\prec\mu$, and therefore $y\notin\V_Q$ by the definition of  $\V_Q$.
This proves that $x$ is isolated in $\V_Q$ from the right.

Now assume that $\alpha=\alpha_1\cdots\alpha_j\mu$ for some $j\ge 1$.
Then $\alpha_j=1$.
Indeed, assume by the contrary that $\alpha_j=0$.
Then, since $\alpha_1=1$, there exists a positive integer $k<j$ such that $\alpha_k=1$ and $\alpha_{k+1}=\cdots\alpha_j=0$.
Then $\sigma^k(\alpha)=0^{j-k}\mu\prec\mu$, contradicting the relation $x\in \V_Q$.
\medskip

If $y<x$ is sufficiently close to $x$, then $a(y)$ starts with $a_1\cdots a_{n-1}0\alpha_1\cdots\alpha_j$, and $\sigma^{n+j}(a(y))\prec\mu$.
Since $a_{n+j}(y)=\alpha_j=1$, this implies again that $y\notin\V_Q$.
\medskip

(ii) The proof is similar to (i).
\end{proof}

\begin{lemma}\label{l53}\ 
\begin{enumerate}[\upshape(i)]
\item Assume that $\V_Q\setminus\U_Q=B_Q$.\footnote{This assumption is satisfied in cases   \eqref{22}, \eqref{27}, \eqref{29}, \eqref{31} of Lemma \ref{l:HuBarZou}.} 
Then:
\begin{enumerate}[\upshape(a)]
\item $\U_Q$ is closed $\Longleftrightarrow1/(q_0(q_1-1))\notin\overline\U_Q$.
\item  If, moreover, $A_Q\neq \emptyset$,\footnote{This occurs in cases   \eqref{22}, \eqref{27}, \eqref{29} of Lemma \ref{l:HuBarZou}.}  then  $1/(q_0(q_1-1))\in\overline\U_Q\Longleftrightarrow\V_Q\setminus\overline{\U}_Q$ is discrete;
\end{enumerate}

\item Assume that $\V_Q\setminus\U_Q=A_Q$.\footnote{This assumption is satisfied in cases   \eqref{23}, \eqref{24}, \eqref{29}, \eqref{30} of Lemma \ref{l:HuBarZou}.} 
Then:

\begin{enumerate}[\upshape(a)]
\item $\U_Q$ is closed  $\Longleftrightarrow 1/q_1\notin\overline\U_Q$.

\item If, moreover, $B_Q\neq \emptyset$,\footnote{This occurs in cases   \eqref{22}, \eqref{27}, \eqref{29}.} then  $1/q_1\in\overline\U_Q\Longleftrightarrow\V_Q\setminus\overline{\U}_Q$ is discrete.
\end{enumerate}
\end{enumerate}
\end{lemma}

\begin{proof}
(i-a)
Since $1/(q_0(q_1-1))=\pi_Q(1\mu)=\pi_Q(01^{\infty})$, $1/(q_0(q_1-1))\notin\U_Q$.
Therefore $\U_Q$ is not closed if $1/(q_0(q_1-1))\in\overline\U_Q$.

Conversely, assume that $\U_Q$ is not closed.
Since  $\U_Q$ is closed from below by Lemma \ref{l46},  there exists a sequence of numbers $z_k\in\U_Q$ such that $z_k\searrow z\in\overline\U_Q\setminus\U_Q$, and then $m(z_k)\rightarrow m(z)$ by Lemma \ref{l41}.
Since $\overline\U_Q\setminus\U_Q\subseteqq\V_Q\setminus\U_Q=B_Q$ by our assumption, $m(z)$ ends with $1\mu$, i.e., $\sigma^{\ell}(m(z))=1\mu$ for some $\ell\ge 0$.
Then $\sigma^{\ell}(m(z_k))\to 1\mu$, and therefore $\pi_Q({\sigma^{\ell}(m(z_k))})\rightarrow\pi_Q(1\mu)=1/(q_0(q_1-1))$.
Since $\pi_Q({\sigma^{\ell}(m(z_k))})\in\U_Q$ for every $k$, this proves that $1/(q_0(q_1-1))\in\overline\U_Q$.
\medskip

(i-b) Given a point $x\in B_Q\setminus A_Q$,\footnote{This is possible in cases   \eqref{22} and \eqref{27}.} and write $m(x)=a(x)=(a_i)$.
By definition there exists a smallest positive integer $n$ such that 
\begin{equation}\label{e58-1}
a_n=1,\quad
a(x)=a_1\cdots a_n\mu,\qtq{and}
\sigma^i(a(x))\prec \alpha\qtq{whenever}a_i=0.
\end{equation}
Therefore, by Lemma \ref{l:important1} there exists a sequence  $1<k_1<k_2<\cdots$ of integers such that for each $i\geq 1$,
\begin{equation}\label{e58}
a_{k_i}=0, \qtq{and} a_{j+1}\cdots a_{k_i}\prec\alpha_1\cdots \alpha_{k_{i}-j} \qtq{whenever} 1\leq j<k_i \qtq{and} a_j=0.
\end{equation} 
Furthermore,
\begin{equation}\label{e59}
a_{j+1}\cdots a_n\succ \mu_1\cdots\mu_{n-j} \qtq{whenever} 1\le j<n \qtq{and} a_j=1.
\end{equation} 
Indeed, otherwise using Remark \ref{r22} (i) we would have
\begin{equation*}
\mu\preceq a_{j+1}\cdots a_n\mu
\preceq \mu_1\cdots\mu_{n-j}\mu
\preceq\mu_1\cdots\mu_{n-j}\sigma^{n-j}(\mu)=\mu;
\end{equation*}
this would imply $a_j=1$ and $a(x)=a_1\cdots a_j\mu$, contradicting the minimality of $n$.

It follows from \eqref{e58} and \eqref{e59} that each of the points
\begin{equation*}
y_i:=\pi_Q(a_1\cdots a_{k_i}\alpha),\quad i=1,2,\ldots
\end{equation*}
belongs to $\V_Q$, and is different from $x$.
Since they obviously converge to $x$, $x$ is not isolated in $\V_Q$. 

If $\U_Q$ is closed, then $(y_i)$ has a subsequence belonging to $\V_Q\setminus\U_Q$, and we conclude that $\V_Q\setminus\U_Q=\V_Q\setminus\overline\U_Q$ is not discrete.
Using (i-a) we conclude that if $A_Q\subsetneqq B_Q$ and $1/(q_0(q_1-1))\notin \overline{\U}_Q$, then $\V_Q\setminus\U_Q=\V_Q\setminus\overline\U_Q$ is not discrete.

Now assume that $1/(q_0(q_1-1))\in \overline{\U}_Q$.
Since $A_Q\ne\emptyset$ by our assumption, one of the cases \eqref{22}, \eqref{27} and \eqref{29} of Lemma \ref{l:HuBarZou} holds, so that we may apply Lemma  \ref{l:isolated} (i) to conclude that each point of $A_Q$ is isolated in $\V_Q$.

In case \eqref{29} we may also apply Lemma  \ref{l:isolated} (ii) to conclude that each point of $B_Q$ is isolated in $\V_Q$, so that in case \eqref{29} each point of $\V_Q\setminus{\U}_Q$  is isolated in $\V_Q$.\footnote{This has already been proved in a different way in Theorem \ref{T:topology1} (iii) at the end of Section \ref{s4}.}

Henceforth we consider the cases \eqref{22} and \eqref{27}.
We claim that $B_Q\setminus\A_Q\subseteq \overline{\U}_Q$; this will imply the inclusion $\V_Q\setminus \overline{\U}_Q\subseteq A_Q$, and hence that each point of $\V_Q\setminus \overline{\U}_Q$ is isolated in $\V_Q$.

Since $1/(q_0(q_1-1))\in\overline\U_Q\setminus\U_Q$, by Lemma \ref{l:closed} (ii) there exists a sequence $(z_k)$ in $\U_Q$ such that $z_k\searrow 1/(q_0(q_1-1))$.
Applying Lemma \ref{l41} this yields the relations
\begin{equation}\label{e510}
a(z_k)=m(z_k)\rightarrow m(1/(q_0(q_1-1)))=1\mu.
\end{equation}

For each $x\in B_Q\setminus\A_Q$ with $m(x)=a(x)=(a_i)$ satisfying \eqref{e58-1}--\eqref{e59} with a minimal $n$, the formula
\begin{equation*}
y^k:=\pi_Q(a_1\cdots a_{n-1}a(z_k))
\end{equation*}
defines a sequence satisfying $y^k\searrow x$, and the proof will be completed if we show that $y^k\in \U_Q$ for every sufficiently large $k$.

Since $x\in B_Q\setminus\A_Q$, we have
\begin{equation*}
\sigma^j(a(x))\prec\alpha\qtq{whenever}a_j=0,
\end{equation*}
and, using the minimality of $n$,
\begin{equation*}
\sigma^j(a(x))\succ\mu\qtq{whenever}1\le j\le n-1\qtq{and}a_j=1.
\end{equation*}
Therefore there exists an integer $\ell>n$  such that
\begin{align*}
&a_{j+1}\cdots a_\ell\prec\alpha_1\cdots\alpha_{\ell-j}\qtq{whenever}1\le j\le n-1\qtq{and}a_j=0,
\intertext{and}
&a_{j+1}\cdots a_\ell\succ\mu_1\cdots\mu_{\ell-j}\qtq{whenever}1\le j\le n-1\qtq{and}a_j=1.
\end{align*}

Since $m(y^k)\to (a_i)$,  $m(y^k)$ starts with $a_1\cdots a_\ell$ for every sufficiently large $k$, and then the lexicographic conditions ensuring $y^k\in\U_Q$ are satisfied for $j=1,\ldots, n-1$ by the choice of $m$.
The lexicographic conditions are also satisfied for $j\ge n$ because $z_k\in\U_Q$ and $\sigma^{n-1}(y_k)=a(z_k)$.
\medskip

(ii) follows from (i) by symmetry.
\end{proof}

\begin{proof}[Proof of Theorem \ref{T:topology1} (v)--(vi)]\

(v) The relation ${\overline\U}_Q=\V_Q$ follows  from Lemma \ref{l:close set} (\ref{l:close set-3}) and  Lemma \ref{l:density3}.
Since $\U_Q\subsetneqq\V_Q$ by Lemma \ref{l31}, this implies that $\U_Q$ is not closed.

Next we show that $\overline\U_Q$ has no isolated point. 
This follows by observing that for each $x\in \U_Q$, by  Lemma \ref{l:density1}, there exists a sequence $(y_i)$ in $\V_Q\setminus\U_Q$ such that $y_i\rightarrow x$, and for each $y\in \V_Q\setminus\U_Q$, by Lemma \ref{l:density3}, there exists a sequence $(z_i)$ in $\U_Q$ such that $z_i\rightarrow y$.

It remains to prove that if $Q\notin\C$, then $\overline\U_Q$ has no interior points. 
Assume on the contrary that $\overline\U_Q$ has an interior point $y$.
Then by Lemma \ref{l:trunca} (\ref{key45-1}), there also exists an interior point $x\le y$ of $\overline\U_Q$, having a finite greedy expansion. 
Then by Lemma \ref{l:neighbor} (i) there  exists a $z>x$ such that $(x,z]\cap \V_Q=\emptyset$.
But this is impossible because $x$ is an interior point of $\overline\U_Q$ and $\overline\U_Q\subseteqq\V_Q$.
\medskip
 
(vi) Lemmas \ref{l31} and \ref{l:isolated} imply that $\V_Q\setminus\U_Q$ has isolated points.
Hence $\V_Q$ is not a Cantor set, and $\overline\U_Q\subsetneqq\V_Q$.

The remaining assertions follow from Lemmas \ref{l:HuBarZou}, \ref{l:condition} and \ref{l53}.
\end{proof}

\begin{proof}[Proof of Theorem \ref{T:topology2} for $Q\in \V$]
(i) First we show that if $x_L\in A_Q$ with $b(x_L)=b_1\cdots b_{n-1}10^\infty$, then $l(x_R)=b_1\cdots b_{n-1}01^\infty$ for a suitable point $x_R\in B_Q$.

Indeed, we have $a(x_L)=b_1\cdots b_{n-1}0\alpha$ by Lemma \ref{l:lA}.
Since $a(x_L)$ is a quasi-greedy sequence, by Lemma \ref{l:monotonicity} (ii) we have
\begin{equation}\label{e:512}
b_j\cdots b_{n-1}0\succeq \mu_1\cdots\mu_{n-j}\qtq{whenever}1\le j<n\qtq{and}b_j=1.
\end{equation}
By Lemma \ref{l:monotonicity} (iii) this implies that $b_1\cdots b_{n-1}01^\infty$ is the lazy expansion of some number $x_R$, and then by Lemma \ref{l:lB} we have $m(x_R)=b_1\cdots b_{n-1}1\mu$.

It remains to show that $x_R\in B_Q$.
Since $m(x_R)$ ends with $1\mu$, by Lemma \ref{l:quasi-gr-la} it is sufficient to show that $m(x_R)=a(x_R)$.
Since $m(x_R)=b_1\cdots b_{n-1}1\mu$ is doubly infinite by Lemma \ref{l:doublyinfinite}, this will follow from the relations 
\begin{equation*}
\sigma^j(b_1\cdots b_{n-1}1\mu)\preceq \alpha\qtq{whenever}m_j(x_R)=0.
\end{equation*}
For $j>n$ this follows from the relations $\sigma^{j-n}(\mu)\preceq \alpha$.
For $j<n$ with $b_j=0$ we have 
\begin{equation*}
b_{j+1}\cdots b_{n-1}1\preceq \alpha_1\cdots\alpha_{n-j}
\end{equation*}
because $b_1\cdots b_{n-1}10^\infty$ is a greedy sequence, and hence
\begin{equation*}
\sigma^j(b_1\cdots b_{n-1}1\mu)
\preceq \alpha_1\cdots\alpha_{n-j}\mu
\preceq \alpha_1\cdots\alpha_{n-j}\sigma^{n-j}(\alpha)=\alpha.
\end{equation*}
We have used the relations $\sigma^{j-n}(\mu)\preceq \alpha$ and $\mu\preceq\sigma^{n-j}(\alpha)$ that hold for all $Q\in\V$ by Lemma \ref{l:HuBarZou}.

By symmetry, if $x_R\in B_Q$ with   $l(x_R)=b_1\cdots b_{n-1}01^\infty$, then $b(x_L)=b_1\cdots b_{n-1}10^\infty$ for a suitable point $x_L\in A_Q$.

We claim that $(x_L, x_R)\cap \V_Q=\emptyset$ for every $x_L\in A_Q$.
Indeed, assume on the contrary that there exists an  $x\in (x_L, x_R)\cap \V_Q$ with some $x_L\in A_Q$, and write $b(x_L)=b_1\cdots b_{n-1}10^\infty$.
Then $a(x)=m(x)$ by Lemma \ref{l:quasi-gr-la}, and therefore
\begin{equation*}
b_1\cdots b_{n-1}0\alpha=a(x_L)\prec a(x)=m(x)\prec m(x_R)=b_1\cdots b_{n-1}1\mu.
\end{equation*} 
It follows that $(c_i):=a(x)=m(x)$ starts with $b_1\cdots b_{n-1}$.
If  $c_n=0$, then $c_{n+1}c_{n+2}\cdots\preceq\alpha$ because $(c_i)$ is a quasi-greedy sequence, but this contradicts the relation $b_1\cdots b_{n-1}0\alpha\prec a(x)$.
Similarly, if  $c_n=1$, then $c_{n+1}c_{n+2}\cdots\succeq\mu$ because $(c_i)$ is also a quasi-lazy sequence, and this contradicts the relation $m(x)\prec b_1\cdots b_{n-1}1\mu$.

Since $|\V_Q\setminus \U_Q|=|A_Q\cup B_Q|=\aleph_0$ by Theorem \ref{T:topology1} (ii), there are $\aleph_0$ such intervals $(x_L,x_R)$.

It remains to show that the intervals $(x_L,x_R)$ cover the set $J_Q\setminus \V_Q$.
Take an arbitrary point $x\in J_Q\setminus \V_Q$.
Then there exists a smallest integer $N \ge 1$ such that either
\begin{equation*}
m_N(x)=0\qtq{and}\sigma^N(m(x))\succ\alpha,
\end{equation*}
or
\begin{equation*}
a_N(x)=1\qtq{and}\sigma^N(a(x))\prec\mu.
\end{equation*}
By symmetry we consider only the first case.
Writing $m(x)=(m_i)$ for simplicity, first we observe that $(c_i):=m_1\cdots m_{N-1}01^\infty$ is a lazy sequence by Lemma \ref{l:trunca} (ii).
Write $l(x_R)=(c_i)$, then $m(x_R)=m_1\cdots m_{N-1}1\mu$ by Lemma \ref{l:rela-quasi-gl}.
We are going to show that
\begin{equation*}
\sigma^j(m_1\cdots m_{N-1}1\mu)\preceq \alpha\qtq{whenever}m_j(x_R)=0;
\end{equation*}
this will imply $m(x_R)=a(x_R)$ and then $x_R\in B_Q$ as in the first part of the proof.

As before, the case $j>N$ is obvious.
If $j<N$, then 
\begin{equation}\label{e:512-2}
m_{j+1}\cdots m_{N-1}0\prec \alpha_1\cdots\alpha_{N-j}.
\end{equation}
Indeed, the weak inequality $\preceq$ follows from the minimality of $N$.
Furthermore, equality cannot hold because this would imply
\begin{equation*}
m_{j+1}m_{j+2}\cdots
=\alpha_1\cdots\alpha_{N-j}m_{N+1}m_{N+2}\cdots
\succ \alpha_1\cdots\alpha_{N-j}\alpha\succeq \alpha_1\cdots\alpha_{N-j}\alpha_{N-j+1}\cdots=\alpha,
\end{equation*}
contradicting the choice of $N$ again.

It follows from \eqref{e:512-2} that $m_{j+1}\cdots m_{N-1}1\preceq \alpha_1\cdots\alpha_{N-j}$, and therefore, since $\mu\preceq\sigma^{N-j}(\alpha)$,
\begin{equation*}
m_{j+1}\cdots m_{N-1}1\mu\preceq \alpha_1\cdots\alpha_{N-j}\sigma^{N-j}(\alpha)=\alpha,
\end{equation*}
as required.

Since $x_R\in B_Q$, the corresponding interval $(x_L,x_R)$ is given by  $x_L\in A_Q$ such that
\begin{equation*}
a(x_L)=m(x_L)=m_1\cdots m_{N-1}0\alpha\qtq{and}
a(x_R)=m(x_R)=m_1\cdots m_{N-1}1\mu
\end{equation*}
by the first part of the proof.
This implies the relation $x\in (x_L, x_R)$ because $m(x)$ begins with $m_1\cdots m_{N-1}0$, and satisfies $\sigma^N(m(x))\succ\alpha$ by the choice of $N$.
\medskip

(ii) For $Q\in{\V \setminus\overline{\U }}$, we know from Lemmas \ref{l53} and  \ref{l:A-B-set} (iv) that $\U_Q$ is closed, and
\begin{equation*}
\V_Q\setminus \U_Q=\V_Q\setminus \overline{\U}_Q=A_Q=B_Q
\end{equation*}
is a discrete set. 
Since $\U_Q$ is closed, and contains the endpoints of $J_Q$, the components of  $J_Q\setminus \U_Q$  are open intervals $(x_L, x_R)$ with $x_L, x_R \in \U_Q$. 
Since $\V_Q\setminus \U_Q$ is a discrete set, the elements of $\V_Q$ form in each interval $(x_L, x_R)$  an increasing sequence $(x_k)$. 
By Lemma \ref{l:density1} these sequences are infinite in both directions, with
\begin{equation*}
x_k\rightarrow x_L \text{ as } k\rightarrow-\infty,\qtq{and}x_k\rightarrow x_R \text{ as } k\rightarrow\infty.
\end{equation*} 

Since $A_Q=B_Q=\V_Q\setminus\U_Q$, every $x_k\in \V_Q\setminus\U_Q$ has a finite greedy expansion.
We are going to show that
\begin{equation*}
b(x_k)=b_1\cdots b_{n-1}10^{\infty}\Longleftrightarrow
a(x_{k+1})=b_1\cdots b_{n-1}1\mu.
\end{equation*}
We prove the implication $\Longrightarrow$; the proof of the other implication is similar.

If $b(x_k)=b_1\cdots b_{n-1}10^{\infty}$, then $a(x_k)=m(x_k)=b_1\cdots b_{n-1}0\alpha$, and therefore
\begin{equation}\label{510}
\begin{split}
&b_{i+1}\cdots b_{n-1}1\preceq \alpha_1\cdots\alpha_{n-i}\qtq{if}1\le i<n\qtq{and}b_i=0,\\
&b_{i+1}\cdots b_{n-1}1\succ b_{i+1}\cdots b_{n-1}0\succeq \mu_1\cdots \mu_{n-i}\qtq{if}1\le i<n\qtq{and}b_i=1.
\end{split}
\end{equation}
Furthermore, if $x>x_k$ and $x\in\V_Q$, then $a(x)=m(x)\succeq b_1\cdots b_{n-1}1\mu$ by the definition of quasi-lazy expansions.
We complete the proof by showing that the sequence $(c_j):=b_1\cdots b_{n-1}1\mu$ is  both quasi-greedy and quasi-lazy, so that $a(x_{k+1})=m(x_{k+1})=b_1\cdots b_{n-1}1\mu$ for some number $x_{k+1}$.
Then we have obviously $x_{k+1}>x_k$, $x_{k+1}\in\V_Q$ by Lemma \ref{l:quasi-gr-la}, and $x_{k+1}\notin\U_Q$ because  $b_1\cdots b_{n-1}01^{\infty}$ is another expansion of $x_{k+1}$.

It follows from \eqref{510} and the inequalities $\mu\preceq\sigma^k(\alpha)\preceq\alpha$ for all $k\ge 0$ that 
\begin{align*}
\sigma^i((c_j))&\preceq\alpha_1\cdots\alpha_{n-i}\mu 
\preceq \alpha_1\cdots\alpha_{n-i}\sigma^{n-i}(\alpha)
=\alpha\qtq{if}1\le i<n\qtq{and}c_i=0,\\
\sigma^i((c_j))&\succeq\mu\qtq{if}1\le i<n\qtq{and}c_i=1.
\end{align*}
Using \eqref{510} and the inequalities $\mu\preceq\sigma^k(\mu)\preceq\alpha$ for all $k\ge 0$, we conclude that the  sequence $(c_j)$ is both quasi-greedy and quasi-lazy, as required.

Since $J_Q\setminus\U_Q$ is the disjoint union of the open intervals $(x_L,x_R)$, the endpoints $x_L, x_R$ belong to $\U_Q$.
\end{proof}

\section{Proof of Theorems \ref{T:expansion} (iv) and  \ref{T:topology1} (ii), (vii) and (viii) for $Q\in\A\setminus\V$}\label{s6}
	
In this section we  mainly discuss the topological properties of sets $\U_Q$ and $\V_Q$ when $Q\in \A\setminus{\V }$. 
As usual we use the notations
\begin{equation*}
\alpha=(\alpha_i):=\alpha(Q)\qtq{and}
\mu=(\mu_i)=:\mu(Q).
\end{equation*}

The following Lemma \ref{l:outV1} implies the new part of Lemma \ref{l:HuBarZou} with respect to the paper \cite{HuBarZou2024}:
\begin{equation*}
Q\in \A \setminus\V\Longleftrightarrow (\mu,\alpha)\text{ satisfies one of the conditions \eqref{30}--\eqref{32} of Lemma \ref{l:HuBarZou}}.
\end{equation*}

\begin{lemma}\label{l:outV1}\label{l61}\
Let $Q\in \A$.
\begin{enumerate}[\upshape(i)]
\item \label{l:outV1-1} If there exists a smallest integer $k\ge 1$ such that $\mu\succ\sigma^k(\alpha)$, then $\alpha_k=1$.
If, in addition, $\sigma^j(\mu)\preceq\alpha$ for all $j\geq k$, then in fact $\sigma^i(\mu)\prec\alpha$ for all $i\geq 0$.

\item \label{l:outV1-2} If there exists a smallest positive integer $k$ such that $\sigma^k(\mu)\succ\alpha$, then $\mu_k=0$.
If, in addition, $\mu\preceq\sigma^j(\alpha)$ for all $j\geq 1$, then in fact $\mu\prec\sigma^j(\alpha)$ for all $j\geq 1$.
\end{enumerate}		
\end{lemma}
  
\begin{proof}
(i) The case $k=1$ follows from Remark \ref{r14} (v).
Assume on the contrary that $k\ge 2$ and $\alpha_k=0$. 
Then
\begin{equation*}
\sigma^{k-1}(\alpha)=0\sigma^k(\alpha)\prec 0\mu\preceq\mu,
\end{equation*}
contradicting  the minimality of $k$.

Now assume on the contrary that the second assertion fails.
Then $\sigma^i(\mu)=\alpha$ for some $i\geq 0$; hence
\begin{equation*}
\sigma^{i+k}(\mu)=\sigma^k(\alpha)\prec\mu,
\end{equation*}
contradicting the quasi-lazy property of the sequence $\mu$.
\medskip

(ii) follows from (i) by symmetry.  	
\end{proof}
  
\begin{lemma}\label{l:outV2}\label{l62}
Fix $Q\in \A\setminus{\V }.$ 
\begin{enumerate}[\upshape (i)]
\item  Let $(\mu,\alpha)$ satisfy the condition \eqref{31} of  Lemma \ref{l:HuBarZou}
Then:
\begin{enumerate}[\upshape (a)]
\item $A_Q=\emptyset$ and $\abs{B_Q}=\aleph_0$.

\item Each $x\in \V_Q\setminus\U_Q$ has exactly $\aleph_0$ expansions if $\mu$ is periodic, and 2 expansions otherwise.

\item No expansion of any $x\in \V_Q$  ends with $\alpha$.

\item $\U_Q$ is closed from below.

\item Let $n$ be the smallest positive integer such that  $\sigma^n(\alpha)\prec \mu$; then $\alpha_n=1$ by Lemma \ref{l:outV1}, so that $\alpha'=(\alpha_1\cdots\alpha_n^-)^\infty$ is well defined. 
Furthermore,
\begin{enumerate}[\upshape (1)]
\item $\mu\preceq\sigma^i(\alpha')\preceq\alpha'\prec\alpha$ for all $i\ge 0$.

\item If $\sigma^i(\mu)\prec \alpha'$ and $\mu\prec\sigma^j(\alpha')$ for all $i, j\in \N_0$, then $\U_Q\subsetneqq\overline\U_Q=\V_Q$. 
Otherwise, $\U_Q=\overline\U_Q\subsetneqq\V_Q$, and $\V_Q\setminus{\U}_Q$ a discrete set.
\end{enumerate}

\item $\V_Q\setminus\U_Q$ is dense in $\V_Q$.
\end{enumerate}

\item  Let $(\mu,\alpha)$ satisfy the condition \eqref{30} of  Lemma \ref{l:HuBarZou}
Then:
  		
\begin{enumerate}[\upshape (a)] 
\item $B_Q=\emptyset$ and $\abs{A_Q}=\aleph_0$.

\item Each $x\in \V_Q\setminus\U_Q$ has exactly $\aleph_0$ expansions if $\alpha$ is periodic, and 2 expansions otherwise.

\item No expansion of any $x\in \V_Q$  ends with $\mu$.

\item $\U_Q$ is closed from above.
\item Let $n$ be the smallest positive integer such that  $\sigma^n(\mu)\succ \alpha$; then $\mu_n=0$ by Lemma \ref{l:outV1}, so that $\mu'=(\mu_1\cdots\mu_n^+)^\infty$ is well defined. 
Furthermore,
\begin{enumerate}[\upshape (1)]
\item $\mu\prec\mu'\preceq\sigma^i(\mu')\preceq\alpha$ for all $i\ge 0$.

\item If $\sigma^i(\mu')\prec \alpha$ and $\mu'\prec\sigma^j(\alpha)$ for all $i, j\in \N_0$, then
$\U_Q\subsetneqq\overline\U_Q=\V_Q$. 
Otherwise, $\U_Q=\overline\U_Q\subsetneqq\V_Q$, and $\V_Q\setminus{\U}_Q$ a discrete set.
\end{enumerate}
\item $\V_Q\setminus\U_Q$ is dense in $\V_Q$.
\end{enumerate}

\item Let $(\mu,\alpha)$ satisfy the condition \eqref{32} of  Lemma \ref{l:HuBarZou}
Then:
  		
\begin{enumerate}[\upshape (a)]  
\item $\U_Q=\overline{\U}_Q=\V_Q$.

\item No expansion of any $x\in \V_Q$  ends with $\mu$ or $\alpha$.			
\end{enumerate}

\end{enumerate}
\end{lemma}

\begin{proof}
(i-a) First we show that $A_Q=\emptyset$. 
Assume on the contrary that there exists an $x\in A_Q$. 
Then $a(x)$ ends with $0\alpha$.
It follows from our assumption and from Lemma \ref{l:outV1} that $\alpha_k=1$ and $\mu\succ\sigma^k(\alpha)$ for some $k\ge 1$.
Therefore $a(x)$ ends with $1\sigma^k(\alpha)\prec 1\mu$, contradicting the definition of $x\in \V_Q$.

Since $A_Q=\emptyset$, $\abs{B_Q}=\abs{\V_Q\setminus\U_Q}=\aleph_0$ by Theorem \ref{T:topology1} (ii).
\medskip

(i-b) Let $x\in \V_Q\setminus\U_Q$, then  $x\in B_Q\setminus A_Q$ by (i-a).
Therefore $m(x)=a(x)$ by Lemma \ref{l:quasi-gr-la}, and $a(x)=b(x)$  by Proposition \ref{p111qqq} because $x\notin A_Q$.
We conclude by applying Lemma \ref{l:lB}.
\medskip

(i-c) Let $x\in \V_Q$, and assume on the contrary that $x$ has an expansion $(x_i)$ ending with $\alpha$.
Then by the condition \eqref{31} in  Lemma \ref{l:HuBarZou} there exists an integer $k\ge 1$ such that $x_k=1$ and $\sigma^k((x_i))\prec\mu$.
This implies that $(x_i)\ne m(x)$; in particular, $x\in \V_Q\setminus\U_Q$.

Using the last property, we infer from  (i-a) that $x\in B_Q$.
Therefore, applying Lemma \ref{l:lB} and using again the property $(x_i)\ne m(x)$ we conclude that $(x_i)$ ends with $01^\infty$.
This implies that $\alpha$ ends with $1^\infty$, and then $\alpha=1^\infty$ by Lemma \ref{l:monotonicity} (ii).
(Indeed, if $\alpha$ had a last zero digit $\alpha_n=0$, then we would have $\sigma^n(\alpha)=1^\infty\succ\alpha$, contradicting the lexicographic characterization of quasi-greedy expansions.)
But this is contradiction because for $\alpha=1^\infty$ the assumption $\sigma^j(\alpha)\prec\mu$ of the condition \eqref{31} in  Lemma \ref{l:HuBarZou} is not satisfied.
\medskip

(i-d) Since $\V_Q\setminus\U_Q=B_Q$, every  $x\in \V_Q\setminus\U_Q$ has a co-finite lazy expansion, and then satisfies the condition of Lemma \ref{l:neighbor} (ii). 
Therefore there exists a number $z<x$ such that $[z, x]\cap \U_Q=\emptyset$. 

The same conclusion holds for every $x\in J_Q\setminus\V_Q$, too, by applying  Lemma \ref{c:rightneighbor} instead of Lemma \ref{l:neighbor}.

The two properties imply that $\U_Q$ is closed from below.
\medskip

(i-e) Assume on the contrary that $\alpha_n=0$.
Then $n\ge 2$ because $\alpha_1=1$, and 
\begin{equation*}
\sigma^{n-1}(\alpha)=0\sigma^n(\alpha)\prec\sigma^n(\alpha)\prec\mu,
\end{equation*}
contradicting the minimality of $n$.
\medskip

(i-e-1) We claim that 
\begin{equation}\label{e:71}
\alpha_{i+1}\cdots\alpha_n\succ \mu_1\cdots\mu_{n-i}\qtq{for all} 0\le i\le n-1.
\end{equation}
The case $i=0$ is obvious because $\alpha_1=1>0=\mu_1$.
Next assume that \eqref{e:71} fails for some $1\le i\le n-1$.
Then, using our assumption $\sigma^n(\alpha)\prec \mu$ we obtain the relations
\begin{equation*}
\alpha_{i+1}\cdots\alpha_n\;\sigma^n(\alpha)\prec\mu_1\cdots\mu_{n-i}\;\mu\preceq \mu_1\cdots\mu_{n-i}\;\mu_{n-i+1}\cdots=\mu,
\end{equation*}
contradicting again the minimality of $n$.

Next we claim that
\begin{equation}\label{e:72}
\alpha_1\cdots\alpha_n\succ \mu_{k+1}\cdots\mu_{k+n} \qtq{for all} k\ge 0.
\end{equation} 
The case $k=0$ is obvious again.
Assume on the contrary that \eqref{e:71} fails for some $k\ge 1$.
Then we have
\begin{equation*}
\sigma^k(\mu)
=\mu_{k+1}\cdots\mu_{k+n}\sigma^{k+n}(\mu)
\succeq \mu_{k+1}\cdots\mu_{k+n}\mu
\succ\mu_{k+1}\cdots\mu_{k+n}\sigma^n(\alpha)
\succeq\alpha_1\cdots\alpha_n\sigma^n(\alpha)
=\alpha,
\end{equation*}
contradicting one of the the assumptions in case \eqref{31}. 

Now for each $i\ge 0$ we have obviously $\sigma^i(\alpha')\preceq\alpha'\prec \alpha $, and we infer from \eqref{e:71} and \eqref{e:72} that
\begin{equation*}
\sigma^i(\alpha')
=\alpha_{i+1}\cdots\alpha_n^-(\alpha_1\cdots\alpha_n^{-})^\infty\succeq
	\mu_1\cdots\mu_{n-i}\mu_{n-i+1}\cdots=\mu.
\end{equation*}

We will need in the proof of (i-e-2) the following property: for any sequence $(c_k)$,
\begin{equation}\label{63}
\text{if}\quad\alpha'\prec \sigma^i((c_k))\preceq\alpha\qtq{for some}c_i=0,\qtq{then}(c_k)\notin\V_Q'.
\end{equation}
Assume on the contrary that a sequence $(c_k)\in\V_Q'$ satisfies $\alpha'\prec \sigma^i((c_k))\preceq\alpha$ for some $c_i=0$.
Then there exists an integer $m\ge 0$ such that $\sigma^i((c_k))$ starts with $(\alpha_1\cdots\alpha_n^-)^m$, and the following word of length $n$ is $\succeq\alpha_1\cdots\alpha_n$.
On the other hand, since $(c_k)\in \V'_Q$, $\sigma^{i+mn}((c_k))\preceq\alpha$.
We infer from the last two observations that $\sigma^i((c_k))$  starts with $(\alpha_1\cdots\alpha_n^-)^m\alpha_1\cdots\alpha_n$.
Using the definition of $\V_Q'$ it follows that
\begin{equation*}
\alpha_1\cdots\alpha_n\mu\preceq\sigma^{i+mn}((c_k))\preceq\alpha,
\end{equation*}
contradicting our assumption $\sigma^n(\alpha)\prec\mu$.
\medskip
\medskip

(i-e-2) Set $\U'_Q:=\pi^{-1}_Q(\U_Q)$ and let $\V'_Q$ be the set of unique doubly infinite expansions of the elements of $\V_Q$; the are well defined by Proposition \ref{p17qqq}.
Furthermore, we introduce the sets
\begin{equation*}
\begin{split}
&\U'_{Q'}:=\set{(c_k)\in \set{0,1}^\infty: \sigma^i((c_k))\prec \alpha' \text{ whenever } c_i=0; \sigma^i((c_k))\succ \mu \text{ whenever } c_i=1},\\
&\V'_{Q'}:=\set{(c_k)\in \set{0,1}^\infty: \sigma^i((c_k))\preceq \alpha' \text{ whenever } c_i=0; \sigma^i((c_k))\succeq \mu \text{ whenever } c_i=1},\\
&\U_{Q'}:=\pi_Q(\U'_{Q'}),\\
&\V_{Q'}=\pi_Q(\V'_{Q'}).
\end{split}
\end{equation*}
Since $\alpha'\prec\alpha$, we infer from the definitions that  
$\U'_{Q'}\subseteq\U'_Q$ and $\V'_{Q'}\subseteq\V'_Q$. 
In fact, $\V'_Q=\V'_{Q'}$.
For otherwise there exists a sequence $(c_k)\in \V'_Q\setminus\V'_{Q'}$, and then the lexicographic conditions in \eqref{63} are satisfied for some $i$, contradicting our assumption that $(c_k)\notin\V'_{Q'}$.

We infer from the relations $\U'_{Q'}\subseteq\U'_Q$ and $\V'_{Q'}=\V'_Q$ that 
\begin{equation}\label{e63}
\U_{Q'}\subseteq\U_Q,\quad
\overline{\U}_{Q'}\subseteq\overline{\U}_Q\qtq{and}
\V_Q=\V_{Q'}.
\end{equation} 
\medskip

Now we distinguish three subcases.
\medskip

\emph{First subcase.} 
Assume that $\sigma^i(\mu)\prec \alpha'$ and $\mu\prec\sigma^j(\alpha')$ and  for all $i, j\ge 0$.
Then $(\mu,\alpha')$ satisfies Lemma \ref{l:HuBarZou} \eqref{23} or \eqref{28}, and applying Theorem \ref{T:topology1} (vi), we obtain that
\begin{equation*}
\U_{Q'}\subsetneqq\overline{\U}_{Q'}=\V_{Q'}.
\end{equation*} 
Combining this with \eqref{e63} we get
\begin{equation*}
\V_{Q}=\V_{Q'}=\overline{\U}_{Q'}\subseteq\overline{\U}_Q\subseteq\V_Q,
\end{equation*}
whence $\overline{\U}_Q=\V_Q$.
Since $\U_Q\ne\V_Q$ by Lemma \ref{l31}, we conclude that $\U_Q\subsetneqq\overline{\U}_Q=\V_Q$.
\medskip

\emph{Second subcase.} 
Assume that $\sigma^i(\alpha')=\mu$ for some $i\ge 1$.Then, since $\alpha'$ is periodic, $(\mu,\alpha')$ satisfies Lemma \ref{l:HuBarZou} \eqref{29}, and we infer from Theorem \ref{T:topology1} (iii) and Lemma \ref{l46} (iii) that $\U_{Q'}=\overline{\U}_{Q'}\subsetneqq \V_{Q'}$ and $\V_{Q'}\setminus \overline{\U}_{Q'}$ is a discrete set.
	
We claim that $\U_Q=\U_{Q'}$. 
Assume on the contrary that $\U_Q\ne\U_{Q'}$, then by \eqref{e63} there exists a point $x\in \U_Q\setminus\U_{Q'}$ and then $(c_k):=a(x,Q)$ satisfies for some $i\ge 1$ the relations
\begin{equation*}
c_i=0,\qtq{and}	(\alpha_1\cdots\alpha_n^-)^\infty\preceq \sigma^i((c_k)) \prec\alpha.
\end{equation*}
Since $\alpha'=(\alpha_1\cdots\alpha_n^-)^\infty$ is not a unique expansion in  double-base $Q$ by our assumption $\sigma^i(\alpha')=\mu$, we cannot have $(\alpha_1\cdots\alpha_n^-)^\infty=\sigma^i((c_k))$.
Therefore $x\notin\V_Q$ by \eqref{63}, contradicting our assumption $x\in\U_Q$.
We have thus $\U_Q=\U_{Q'}$, and hence also $\overline\U_Q=\overline\U_{Q'}$.
Since $\V_Q=\V_{Q'}$ by \eqref{e63}, we conclude from the relations $\U_{Q'}=\overline{\U}_{Q'}\subsetneqq \V_{Q'}$ that $\U_{Q}=\overline{\U}_{Q}\subsetneqq \V_{Q}$  and $\V_{Q}\setminus \overline{\U}_{Q}$ is a discrete set.
\medskip

\emph{Third subcase.} 
If $\sigma^t(\mu)=\alpha'$ for some $t\ge 1$, then $\U_Q$ is closed.
Indeed, we already know from (i-d) that $\U_Q$ is closed from below.
It remains to show that $\U_Q$ is closed from above. 

Assume on the contrary that a decreasing sequence $(x^k)$ in $\U_Q$ converges to some point $x\notin\U_Q$.Then $x\in\V_Q\setminus\U_Q$ because $\overline\U_Q\subseteq\V_Q$, and then  $x\in B_Q\setminus A_Q$ by (i-a).
By Proposition \ref{p111qqq} the last property implies that $m(Q, x)=a(Q, x)=b(Q, x)$  ends with $1\mu$.

Since $\sigma^t(\mu)=\alpha'$ for some $t\ge 1$, $m(x,Q)=a(x,Q)=b(x,Q)=a_1\cdots a_s\mu_1\cdots\mu_t\alpha'$ for some $s\ge 1$.
By Lemmas \ref{l41} and \ref{l:trunca} there exists a $z>x$, close enough to $x$, such that
\begin{equation*}
b(z,Q)=a_1\cdots a_s\mu_1\cdots\mu_t\alpha_1\cdots\alpha_n0^{\infty}.
\end{equation*}
Then for every $y\in (x,z)$, we have 
\begin{equation*}
b(y,Q)=(b_i)=a_1\cdots a_s\mu_1\cdots\mu_t(\alpha_1\cdots\alpha_n^-)^m \alpha_1\cdots\alpha_nc_1c_2\cdots
\end{equation*}
with some positive integer $m$ and $c_1c_2\cdots\prec\sigma^n(\alpha)$ by Lemma \ref{l:monotonicity}.
Since $\sigma^n(\alpha)\prec\mu$ by our assumption \eqref{31}, hence $b_{s+t+mn+n}=\alpha_n=1$ is followed by $c_1c_2\cdots\prec\mu$, so that $y\notin\V_Q$ by Lemma \ref{l:monotonicity}.
Therefore $(x, z)\cap \V_Q=\emptyset$, contradicting the existence of the sequence $(x^k)$ at the beginning of the proof.

We have shown that $\U_Q$ is closed.
Since $\U_Q\subsetneqq\V_Q$ by Lemma \ref{l31}, we conclude that $\U_Q=\overline{\U}_Q\subsetneqq\V_Q$.

We have also shown that every $x\in \V_Q\setminus {\U}_Q$ is isolated from the right in $V_Q$.
Since $\V_Q\setminus \overline{\U}_Q=B_Q$ by (i-a), $x\in \V_Q\setminus {\U}_Q$ is also isolated from the left in $V_Q$ by Lemma \ref{l:isolated}.
Therefore $\V_Q\setminus \overline{\U}_Q$ is a discrete set.
\medskip

(f) It follows from (e-i) and the condition Lemma \ref{l:HuBarZou} \eqref{31} that  $(\mu, \alpha')$ satisfies  one of the conditions Lemma \ref{l:HuBarZou} \eqref{21}--\eqref{29}.
Therefore $\V_{Q'}\setminus \U_{Q'}$ is dense in $\V_{Q'}$ by Theorem \ref{T:topology1} (iii).
Furthermore, $\V_{Q}=\V_{Q'}$ by \eqref{e63}.
This implies the density of $\V_Q\setminus \U_Q$ in $\V_Q$
if $\U_{Q}=\U_{Q'}$.

Otherwise we have $\U_{Q'}\subsetneqq\U_{Q}$ by \eqref{e63}, and it remains to find for each fixed $x\in\U_{Q}$ a sequence of points $y^k\in\V_Q\setminus \U_Q$ converging to $x$.

If $x\in\U_{Q}\setminus\U_{Q'}$, then we may apply Lemma \ref{l:important1} (ii) to $a(x,Q)=(a_i)$: there exists a sequence $1<\ell_1<\ell_2<\cdots$ of  integers such that for each $i\geq 1$,
\begin{equation}\label{e64uj}
a_{\ell_i}=1, \qtq{and} a_{j+1}\cdots a_{\ell_i}\succ\mu_1\cdots \mu_{\ell_{i}-j} \qtq{whenever} 1\leq j<\ell_i \qtq{and} a_j=1.
\end{equation}
Since 
\begin{equation*}
y^k=(x_i):=\pi_Q( a_{1}\cdots a_{\ell_k}\mu )\rightarrow x \text { as } \ell_k\rightarrow \infty,
\end{equation*}
it remains to show that $y^k\in \V_{Q}\setminus\U_Q$ for every $k$.

Since $x_{\ell_k}=a_{\ell_k}=1$ is followed by $\mu$, $y^k\notin \U_Q$.
If $x_j=1$ for some $j\ge 1$, then $\sigma^j((x_i))\succ \mu$ by \eqref{e64uj} if $j<\ell_k$, and $\sigma^j((x_i))=\sigma^{j-\ell_k}(\mu)\succeq \mu$ if $j\ge\ell_k$.

It remains to show that $\sigma^j((x_i))\preceq \alpha$ whenever $x_j=0$.
For this first we infer from \eqref{63} that
\begin{equation}\label{e63-2}
\U'_{Q}=\set{(c_k)\in \set{0,1}^\infty: \sigma^i((c_k))\preceq \alpha' \text{ whenever } c_i=0; \sigma^i((c_k))\succ \mu \text{ whenever } c_i=1}.
\end{equation}

Now let $x_j=0$ for some $j\ge 1$; we have to show that $\sigma^j(y^k)\preceq\alpha$.

If $j<\ell_k-n$, then using \eqref{e63-2} we get
\begin{equation*}
x_{j+1}\cdots x_{j+n}\le \alpha_1\cdots\alpha_n^-\prec\alpha_1\cdots\alpha_n,
\end{equation*}
and therefore $\sigma^j(y^k)\prec\alpha$.

If $\ell_k-n\le j<\ell_k$, then using the relation $\sigma^j((a_i))\prec\alpha$ we obtain that
\begin{equation*}
\sigma^j(y^k)\preceq\alpha_1\cdots\alpha_{\ell_k-j}\mu
\preceq \alpha 
\end{equation*}
because $\mu\preceq\sigma^{\ell_k-j}(\alpha)$ by the minimality of $n$.

Finally, if $j\ge\ell_k$, then
\begin{equation*}
\sigma^j(y^k)=\sigma^{j-\ell_k}(\mu)\prec\alpha
\end{equation*}
by the condition \eqref{31}.

If $x\in \U_{Q'}$, then by Lemma \ref{l:density1} there exists a sequence $y^k\rightarrow x$ with $y^k\in B_{Q'}\subseteq\V_{Q'}\setminus\U_{Q'}$ for every $k$.
We complete the proof by observing that $y^k\in\V_Q\setminus\U_Q$. 
Since $\V_Q=\V_{Q'}$, this follows by observing that $a(y^k,Q)=m(y^k,Q)$ ends with $1\mu$, and  the unique expansion of an element of $\U_Q$ cannot end with $1\mu$ by the lexicographic characterization of $\U_Q$.
\medskip

(ii) It follows  from (i) by symmetry.
\medskip

(iii-a) By the same argument as the proof of (i-a) and (ii-a), one may show that $A_Q=\emptyset$ and  $B_Q=\emptyset$. 
Therefore $\U_Q=\V_Q$, and hence $\U_Q=\overline{\U}_Q=\V_Q$ by the general relations $\U_Q\subseteqq\overline{\U}_Q\subseteqq\V_Q$.
\medskip

(iii-b) By (iii-a) every $x\in \V_Q$ has a unique expansion $(x_i)$.
If $(x_i)$ ends with $\mu$, then by our assumption \eqref{32} there exists a $j\ge 1$ such that $x_j=0$ and $\sigma^j((x_i))\succ\alpha$.
Similarly, if $(x_i)$ ends with $\alpha$, then by \eqref{32} there exists a $j\ge 1$ such that $x_j=1$ and $\sigma^j((x_i))\prec\mu$.
Both inequalities contradict the definition of $x\in \U_Q$
\end{proof}

\begin{proof}[Proof of Theorem \ref{T:expansion} (iv)]
This follow from Lemma \ref{l:outV2}.
\end{proof}	

\begin{proof}[Proof of Theorem \ref{T:topology1} (ii), (vii) and (viii)]
The required results follow from Lemmas \ref{l53} and \ref{l:outV2}.
\end{proof}	

\section{Examples}\label{s7}

In this section the conditions \eqref{21}--\eqref{32} refer to the cases of Lemma \ref{l:HuBarZou}, and the items in the examples are also labeled with these conditions.

\begin{examples}\label{e18}\label{l71}
All cases of Lemma \ref{l:HuBarZou} may occur.
Indeed, the following pairs of sequences $(\mu,\alpha)$ satisfy the corresponding conditions \eqref{21}--\eqref{32}, respectively, and each pair $(\mu,\alpha)$ is equal to $(\mu(Q),\alpha(Q))$ for some $Q\in \A$ by \cite[Theorem 1]{HuBarZou2024}.
\begin{enumerate}[\upshape (i)]
\item $\mu=0(01)^\infty$ and $\alpha=1(110)^\infty$,
\item $\mu=0(01)^\infty$ and $\alpha=110(01)^\infty$,
\item $\mu=001(110)^\infty$ and $\alpha=1(110)^\infty$, 
\item $\mu=0(01)^\infty$ and $\alpha=(110)^\infty$, 
\item $\mu=(01)^\infty$	 and $\alpha=1(110)^\infty$, 
\item $\mu=0(01)^\infty$ and $\alpha=(10)^\infty$, 	
\item $\mu=(01)^\infty$ and $\alpha=11(01)^\infty$, 
\item $\mu=(01)^\infty$	and $\alpha=(110)^\infty$ (outside $\C$)
\item $\mu=(00011)^\infty$ and $\alpha=(11000)^\infty$, 
\item $\mu=00(110)^\infty$ and $\alpha=(10)^\infty$,
\item $\mu=0(01)^\infty$ and $\alpha=11(001)^\infty$,
\item $\mu=00(110)^\infty$ and $\alpha=11(001)^\infty$.
\end{enumerate}
\end{examples}

\begin{examples}\label{e72}
We recall from Remark \ref{r19} that the sets of double-bases satisfying one of the conditions  \eqref{24}, \eqref{27}, \eqref{28} and \eqref{29} are countable.
Now we show that each of these sets is countably infinite, while the other eight sets have $2^{\aleph_0}$ elements.
By symmetry it is sufficient to consider the cases \eqref{21}, \eqref{22}, \eqref{23}, \eqref{24}, \eqref{28}, \eqref{29}, \eqref{30} and \eqref{32}.

\begin{enumerate}
\item [(i)] is satisfied for all sequences
\begin{equation*}
\mu\in 0\set{01,011}^{\infty}\qtq{and}\alpha\in 111\set{01,011}^{\infty}.
\end{equation*}

\item [(ii)] is satisfied for all sequences
\begin{equation*}
\mu\in 0\set{01,011}^{\infty}\qtq{and}\alpha=111\mu.
\end{equation*}

\item [(iv)] is satisfied for all sequences
\begin{equation*}
\mu\in 0\set{01,011}^{\infty}\qtq{and}\alpha=(1110)^{\infty}.
\end{equation*}

\item [(vi)] is satisfied for all sequences
\begin{equation*}
\mu=00(1^k0)^{\infty}\qtq{and}\alpha=(1^k0)^{\infty},\quad k\in\N.
\end{equation*}

\item [(viii)] is satisfied for all sequences
\begin{equation*}
\mu=(001)^{\infty}\qtq{and}\alpha=(1^k0)^{\infty},\quad k\in\N.
\end{equation*}

\item [(ix)] is satisfied for all sequences
\begin{equation*}
\mu=(01^k)^{\infty}\qtq{and}\alpha=(1^k0)^{\infty},\quad k\in\N.
\end{equation*}

\item [(x)] is satisfied for all sequences
\begin{equation*}
\mu\in 00\set{11110,111110}^{\infty}\qtq{and}\alpha\in 111(01)^{\infty}.
\end{equation*}

\item [(xii)] is satisfied for all sequences
\begin{equation*}
\mu\in 00\set{11110,111110}^{\infty}\qtq{and}\alpha= 111(0001)^{\infty}.
\end{equation*}
\end{enumerate}
\end{examples}

\begin{examples}\label{ex:116}\label{l73}\
We illustrate Theorem \ref{T:topology1} and Table \ref{t:two-base-A-B}.
Since the cases \eqref{25}, \eqref{26}, \eqref{24}, \eqref{30} of Lemma \ref{l:HuBarZou} are the reflections of \eqref{22}, \eqref{23}, \eqref{27}, \eqref{31}, respectively, by symmetry we consider only the cases \eqref{21}, \eqref{22}, \eqref{23}, \eqref{27}, \eqref{28}, \eqref{29}, \eqref{31}, \eqref{32}.

We recall that $\U_Q\subsetneqq\V_Q$ in cases \eqref{21}--\eqref{31} by Lemma \ref{l31}.

Since most of the following properties readily follow from the definitions and the lexicographic descriptions, the proofs are omitted.

\begin{enumerate}[\upshape (i)]
\item [(i),] (iv), (viii) Set
\begin{equation*}
(\mu(Q^i),\alpha(Q^i)):=
\begin{cases}
\bj{0(001)^\infty,1(1110)^\infty}&\text{for }i=1,\\
\bj{0(001)^\infty,(1110)^\infty}&\text{for }i=4,\\
\bj{(001)^\infty,(1110)^\infty}&\text{for }i=8,
\end{cases}
\end{equation*}
then $Q^1, Q^4, Q^8$ satisfy (i), (iv) and (viii), respectively. 
Note that $\U_{Q^i}\subsetneqq\V_{Q^i}$ because
\begin{equation*}
1/q_1\pi_{Q^i}\bj{10^{\infty}}\in A_{Q^i}\qtq{and}1/(q_0(q_1-1))=\pi_{Q^i}\bj{01^{\infty}}\in B_{Q^i}
\end{equation*}
by an easy lexicographic verification.

Write $\mu(Q^i)=(\mu^i_j)$ and $\alpha(Q^i)=(\alpha^i_j)$ for brevity.
If $x\in A_{Q^{i}}$, then $a(x)=m(x)=a_1\cdots a_n\alpha(Q^i)$ with $a_n=0$. 
A direct verification shows that
\begin{equation*}
\pi_{Q^i}(a_1\cdots a_n\alpha^i_1\cdots\alpha^i_k (10)^\infty)\rightarrow x \text{ as } k\rightarrow\infty,
\end{equation*}
and
\begin{equation*}
\pi_{Q^i}(a_1\cdots a_n\alpha^i_1\cdots\alpha^i_k (10)^\infty)\in \U_Q^i
\end{equation*}
for every $k$.

Similarly, if $x\in B_Q^i$  then $a(x)=m(x)=a_1\cdots a_n\mu(Q^i)$  with $a_n=1$,
\begin{equation*}
\pi_{Q^i}(a_1\cdots a_n\mu^i_1\cdots\mu^i_k (10)^\infty)\rightarrow x\text{ as } k\rightarrow\infty,
\end{equation*}
and
\begin{equation*}
\pi_{Q^i}(a_1\cdots a_n\mu^i_1\cdots\mu^i_k (10)^\infty)\in \U_{Q^i}
\end{equation*}
for every $k$.
This shows that $\V_{Q^i}\subseteqq{\overline\U}_{Q^i}$.

Since $\U_{Q^i}\subsetneqq\V_{Q^i}$  and $\V_{Q^i}$ is closed, we conclude that $\U_{Q^i}\subsetneqq{\overline\U}_{Q^i}=\V_{Q^i}$. 

\item[(ii-a)] \label{ex:6}
$(\mu,\alpha):=\bj{0(01)^\infty,110(01)^\infty}$  satisfies \eqref{22}.\footnote{This is Example \ref{e18} (ii).} 
The unique expansions are
\begin{equation*}
0^{\infty},\quad 1^{\infty},\quad 0(01)^\infty,\qtq{furthermore}0^k(10)^\infty\qtq{and}1^k(01)^\infty\qtq{for}k\ge 0,
\end{equation*}
whence $\U_Q$ is closed.
Furthermore, $\V_Q\setminus\overline\U_Q$ is not discrete because
the points
\begin{equation*}
\pi_Q\bj{10(01)^\infty} \qtq{and}\pi_Q\bj{10(01)^k0110(01)^\infty}\qtq{belong to}\V_Q\setminus\overline\U_Q,
\end{equation*} 
and
\begin{equation*}
\pi_Q\bj{10(01)^k0110(01)^\infty}\rightarrow\pi_Q\bj{10(01)^\infty}.
\end{equation*}

\item[(ii-b)] \label{ex:5}
$(\mu,\alpha):=\bj{0(01)^\infty,1110(01)^\infty}$ satisfies \eqref{22}.
The expansions $110(01)^k(10)^\infty$ are unique, and they converge to $110(01)^\infty$ as $k\to\infty$, but the limit expansion is not unique.
Hence $\U_Q$ is not closed.

\item[(vii-a)]  \label{ex:4}
$(\mu,\alpha):=\bj{(01)^\infty,11(01)^\infty}$ satisfies \eqref{27}.\footnote{This is Example \ref{e18} (vii).}
Now $\U_Q=\set{0, 1/(q_1-1)}$, so that $\U_Q$ is closed. 

Furthermore, $\V_Q\setminus\overline\U_Q$ is not discrete because
the points
\begin{equation*}
\pi_Q\bj{(01)^\infty} \qtq{and}\pi_Q\bj{(01)^k011(01)^\infty}\qtq{belong to}\V_Q\setminus\overline\U_Q,
\end{equation*} 
and
\begin{equation*}
\pi_Q\bj{(01)^k011(01)^\infty}\rightarrow\pi_Q\bj{(01)^\infty}.
\end{equation*}

\item[(vii-b)]  \label{ex:4-2}
$(\mu,\alpha):=\bj{(001)^\infty,111(01)^\infty}$ satisfies \eqref{27}.
The expansions $11(001)^k(10)^\infty$ are unique, they converge to $11(001)^\infty$, but $11(001)^\infty$ is not unique.
Hence $\U_Q$ is not closed.

\item[(ix)] 
$(\mu,\alpha):=\bj{(01)^\infty,(10)^\infty}$ satisfies \eqref{29}.
We have $\U_Q=\set{0, 1/(q_1-1)}$ and 
\begin{equation*}
\V_Q\setminus\U_Q
=\set{\pi_Q(0^k(10)^\infty), \pi_Q(1^k(01)^\infty)\ :\ k\ge 0},
\end{equation*}
$\U_Q=\overline\U_Q\subsetneqq\V_Q$.

\item[(xi-a)] \label{ex:2} 
$(\mu,\alpha):=\bj{(01)^\infty,11(001)^\infty}$ satisfies \eqref{31}.\footnote{This is slightly different from Example \ref{e18} (xi).}
Like in the preceding example, we have $\U_Q=\set{0, 1/(q_1-1)}$ and 
\begin{equation*}
\V_Q\setminus\U_Q
=\set{\pi_Q(0^k(10)^\infty), \pi_Q(1^k(01)^\infty)\ :\ k\ge 0},
\end{equation*}
whence $\U_Q=\overline\U_Q\subsetneqq\V_Q$, and  
$\V_Q\setminus\overline{\U}_Q$ is discrete.

\item[(xi-b)]  \label{ex:3} 
$(\mu,\alpha):=\bj{(01)^\infty,111(001)^\infty}$ satisfies \eqref{31}.\footnote{This is different from Example \ref{e18} (xi).}
Now $\U_Q$ is not closed, because the expansions $(01)^k(011)^\infty$ are unique, they converge to $(01)^\infty$, but $(01)^\infty$ is not unique.

\item[(xii)] \label{ex:1}  
$(\mu,\alpha):=\bj{00(110)^\infty,11(001)^\infty}$ satisfies \eqref{32}.\footnote{This is Example \ref{e18} (xii).} 
We have
\begin{equation*}
\U_Q=\overline\U_Q=\V_Q
=\set{0, \frac{1}{q_1-1}}\cup\set{\pi_Q(0^k(10)^\infty), \pi_Q(1^k(01)^\infty)\ :\ k\ge 0}
\end{equation*}
by a direct verification.
\end{enumerate}

\end{examples}
	
\begin{examples}\label{e37}\label{l74}\ 
We illustrate Lemma \ref{l:condition}.

\begin{enumerate}[\upshape (i)]
\item[(ii)]
$(\mu,\alpha):=\bj{0(01)^\infty,110(01)^\infty}$ satisfies \eqref{22}.\footnote{This is Example \ref{e18} (ii).}
If $x\in A_Q$, then $a(x)$ ends with $0\alpha(Q)=0110(01)^\infty$.
Since $m(x)=a(x)$, hence $m(x)$ ends with $10(01)^\infty=1\mu(Q)$, so that $x\in B_Q$.

We have thus $A_Q\subseteqq B_Q$.
The inclusion is strict because $\pi_Q(0(01)^\infty)\in B_Q\setminus A_Q$.

\item[(ix)] 
$(\mu,\alpha):=\bj{(00011)^\infty,(11000)^\infty}$ satisfies \eqref{29}.\footnote{This is Example \ref{e18} (ix).}
If $x\in A_Q$, then
\begin{equation*}
a(x)=m(x)=a_1\cdots a_j(11100)^\infty=a_1\cdots a_j111(00111)^\infty\in B_Q
\end{equation*}
for some $j\ge 1$ with $a_j=0$, whence $A_Q\subseteqq B_Q$. 
Similarly,  $B_Q\subseteqq A_Q$.
\end{enumerate}
\end{examples}


\begin{thebibliography}{10}

\bibitem{All2017}
P. Allaart, 
On univoque and strongly univoque sets. 
Adv. Math. 308 (2017), 575--598.
	
\bibitem{AllKon2019}
P. Allaart, D. Kong, 
On the continuity of the Hausdorff dimension of the univoque set,
Adv. Math. 354 (2019), 106729, 24 pp.
\href{https://doi.org/10.1016/j.aim.2019.106729}{10.1016/j.aim.2019.106729}

\bibitem{AlcBarBakKon2019}
R. Alcaraz Barrera, S. Baker, D. Kong, 
Entropy, topological transitivity, and dimensional properties of unique q-expansions. 
Trans. Amer. Math. Soc. 371 (2019), no. 5, 3209--3258.
			
\bibitem{BaiKom2007}
C. Baiocchi, V. Komornik, Greedy and quasi-greedy expansions in non-integer bases, arXiv:0710.3001 (2007).
\href{https://doi.org/10.48550/arXiv.0710.3001}{10.48550/arXiv.0710.3001}	

\bibitem{Bak2014}
S. Baker,  
Generalized golden ratios over integer alphabets. 
Integers 14 (2014), Paper No. A15, 28 pp.

\bibitem{BakSte2017}
S. Baker,  W. Steiner,  
On the regularity of the generalised golden ratio function. 
Bull. Lond. Math. Soc. 49 (2017), no. 1, 58--70. 

\bibitem{ChaCisDaj2023}
\'{E}. Charlier, C. Cisternino, K. Dajani, 
Dynamical behavior of alternate base expansions, 
{Ergodic Theory Dynam. Systems.} {43} (2023), 827--860.
\href{https://doi.org/10.1017/etds.2021.161}{10.1017/etds.2021.161}	

\bibitem{DajDeV2007}
K. Dajani, M. de Vries, 
Invariant densities for random $\beta$-expansions. 
J. Eur. Math. Soc. (JEMS) 9 (2007), no. 1, 157--176
\href{https://doi.org/10.4171/JEMS/76}{10.4171/JEMS/76}	

\bibitem{DarKat1995}
Z. Daróczy, I. Kátai, 
\emph{On the structure of univoque numbers,} 
Publ. Math. Debrecen 46 (1995), no. 3--4, 385--408.
\href{https://doi.org/10.5486/pmd.1995.1614}{10.5486/pmd.1995.1614}

\bibitem{DeVKom2009}
M. de Vries, V. Komornik, 
Unique expansions of real numbers, 
{Adv. Math.} {221} (2009), 390--427.
\href{https://doi.org/10.1016/S0723-0869(02)80010-X}{10.1016/S0723-0869(02)80010-X}	

\bibitem{DeVKom2011}
M. de Vries, V. Komornik, 
A two-dimensional univoque set. 
Fund. Math. 212 (2011), no. 2, 175--189.
\href{https://doi.org/10.4064/fm212-2-4}{10.4064/fm212-2-4}

\bibitem{DeVKomLor2016}
M. de Vries, V. Komornik, P. Loreti, 
Topology of the set of univoque bases, 
{Topology Appl.} {205} (2016), 117--137.
\href{https://doi.org/10.1016/j.topol.2016.01.023}{10.1016/j.topol.2016.01.023}	

\bibitem{DeVKomLor2022}
M. de Vries, V. Komornik, P. Loreti, 
Topology of univoque sets in real base expansions, 
{Topology Appl.} {312} (2022), 108085.
\href{https://doi.org/10.1016/j.topol.2022.108085}{10.1016/j.topol.2022.108085}	

\bibitem{ErdHorJoo1991}
P. Erdős, M. Horváth, I. Joó, 
On the uniqueness of the expansions \(1 = \sum q^{-n_i}\),
{Acta Math. Hungar.} 58 (1991), 333--342.
\href{https://doi.org/10.1007/BF01903963}{10.1007/BF01903963}	

\bibitem{ErdJoo1992}
P. Erdős, M. Horváth, I. Joó, 	
On the number of expansions $1=sum q^{-n i}$,
Annales Univ. Sci. Budapest 35 (1992) 129--132.

\bibitem{ErdJooKom1990}
P. Erdős, I. Joó, V. Komornik, 
Characterization of the unique expansions $1=\sum^\infty_{i=1}q^{-n_i}$ and related problems, 
{Bull. Soc. Math. France.} {118} (1990), 377--390.
\href{https://doi.org/10.24033/bsmf.2151}{10.24033/bsmf.2151}	


\bibitem{GleSid2001}
P. Glendinning, N. Sidorov, 
Unique representations of real numbers in non-integer bases,
Mat. Res. Letters 8 (2001), 535--543.
\href{https://doi.org/10.4310/MRL.2001.v8.n4.a12}{10.4310/MRL.2001.v8.n4.a12}	

\bibitem{HuBarZou2024}
Y. Hu, R. Alcaraz Barrera and Y. Zou, Topological and dimensional properties of univoque bases in double-base expansions, \textit{Topology Appl.} 327 (2025), 109294. \href{https://doi.org/10.1016/j.topol.2025.109294}{10.1016/j.topol.2025.109294}

\bibitem{KalKonLanLi2020} 
C. Kalle, D. Kong, N. Langeveld, W. Li, 
The $\beta$-transformation with a hole at 0, 
Ergodic Theory Dynam. Systems. {40} (2020), 2482--2514.
\href{https://doi.org/10.1017/etds.2019.12}{10.1017/etds.2019.12}

\bibitem{Kom2012}
V. Komornik, 
Unique infinite expansions in noninteger bases. 
Acta Math. Hungar. 134 (2012), no. 3, 344--355.
\href{https://doi.org/10.1007/s10474-011-0148-5}{10.1007/s10474-011-0148-5}	

\bibitem{KomLaiPed2011}
V. Komornik, A. C. Lai, M. Pedicini,  
Generalized golden ratios of ternary alphabets. 
J. Eur. Math. Soc. (JEMS) 13 (2011), no. 4, 1113--1146.
\href{ https://doi.org/10.4171/jems/277}{10.4171/jems/277}		

\bibitem{KomLor2007} 
V. Komornik, P. Loreti, 
On the topological structure of univoque sets, 
J. Number Theory. {122} (2007), 157--183.
\href{https://doi.org/10.1016/j.jnt.2006.04.006}{10.1016/j.jnt.2006.04.006}	

\bibitem{KomKonLi2017}
V. Komornik, D. Kong, W. Li, 
Hausdorff dimension of univoque sets and devil's staircase. 
Adv. Math. 305 (2017), 165--196.
\href{https://doi.org/10.1016/j.aim.2016.03.047}{10.1016/j.aim.2016.03.047}		

\bibitem{KomLor1998} 
V. Komornik, P. Loreti,  
{\em Unique developments in noninteger bases},  
Amer. Math. Month\-ly, 105 (1998), 636--639.
\href{https://doi.org/10.2307/2589246}{10.2307/2589246}	

\bibitem{KomLuZou2022} 
V. Komornik, J. Lu, Y. Zou, 
Expansions in multiple bases over general alphabets, 
{Acta Math. Hungar.} {166} (2022), 481--506. 
\href{https://doi.org/10.1007/s10474-022-01231-4}{10.1007/s10474-022-01231-4}	 

\bibitem{KomSteZou2024} 
V. Komornik, W. Steiner, Y. Zou, 
Unique double base expansions, 
{Monatsh. Math.} {204} (2024), 513–542.
\href{https://doi.org/10.1007/s00605-024-01973-z}{10.1007/s00605-024-01973-z}		

\bibitem{Li2021} 
Y. Li, 
Expansions in multiple bases, 
{Acta Math. Hungar.} {163} (2021), 576--600. 
\href{https://doi.org/10.1007/s10474-020-01094-7}{10.1007/s10474-020-01094-7}		

\bibitem{Neu2021}
J. Neunhäuserer, Non-uniform expansions of real numbers, \textit{Mediterr. J. Math.} 18 (2021), Paper No. 70.
\href{https://doi.org/10.1007/s00009-021-01723-7}{10.1007/s00009-021-01723-7}		 

\bibitem{Par1960} 
W. Parry, 
On the $\beta$-expansions of real numbers, 
{Acta Math. Acad. Sci. Hungar.} {11} (1960), 401--416. 
\href{https://doi.org/10.1007/BF02020954}{10.1007/BF02020954}


\bibitem{Ped2005} 
M. Pedicini,  
Greedy expansions and sets with deleted digits. 
Theoret. Comput. Sci. 332 (2005), no. 1-3, 313--336.
\href{https://doi.org/10.1016/j.tcs.2004.11.002}{10.1016/j.tcs.2004.11.002}

\bibitem{Ren1957} 
A. Rényi, 
Representations for real numbers and their ergodic properties, 
{Acta Math. Acad. Sci. Hungar.} {8} (1957), 477--493. 
\href{https://doi.org/10.1007/BF02020331}{10.1007/BF02020331}	

\bibitem{Sid2003} 
N. Sidorov, 
Almost every number has a continuum of $\beta$-expansions. 
Amer. Math. Monthly 110 (2003), no. 9, 838--842.
\href{https://doi.org/10.2307/3647804}{10.2307/3647804}	

\bibitem{Sid2009} 
N. Sidorov, 
Expansions in non-integer bases: lower, middle and top orders, 
{J. Number Theory.} {129} (2009), 741--754. 
\href{https://doi.org/10.1016/j.jnt.2008.11.003}{10.1016/j.jnt.2008.11.003}

\bibitem{Ste2020}
W. Steiner, 
Thue--Morse--Sturmian words and critical bases for ternary alphabets. 
Bull. Soc. Math. France 148 (2020), no. 4, 597--611.
\href{https://doi.org/10.24033/bsmf.2817}{10.24033/bsmf.2817}	

\bibitem{ZouLiLuKom2021}
Y. Zou, J. Li, J. Lu, V. Komornik, 
Univoque graphs for non-integer base expansions,
Sci. China Math. {64} (2021), 2667--2702.
\href{https://doi.org/10.1007/s11425-020-1763-5}{10.1007/s11425-020-1763-5}

\end{thebibliography}
\end{document}